\documentclass[12pt]{amsart}
\usepackage{amsmath}
\usepackage{hyperref}
\usepackage{amssymb}
\usepackage{latexsym}
\usepackage{amscd}
\usepackage{graphicx} 
\usepackage{amsthm}
\usepackage{mathrsfs}
\usepackage{xypic}
\usepackage{bm}
\usepackage[all]{xy}
\usepackage{color}

\newdimen\AAdi%
\newbox\AAbo%
\def\AAk#1#2{\s_etbox\AAbo=\hbox{#2}\AAdi=\wd\AAbo\kern#1\AAdi{}}%
\def\AAr#1#2#3{\s_etbox\AAbo=\hbox{#2}\AAdi=\ht\AAbo\raise#1\AAdi\hbox{#3}}%
\font\tenmsb=msbm10 at 12pt \font\sevenmsb=msbm7 at 8pt
\font\fivemsb=msbm5 at 6pt
\newfam\msbfam
\textfont\msbfam=\tenmsb \scriptfont\msbfam=\sevenmsb
\scriptscriptfont\msbfam=\fivemsb
\def\Bbb#1{{\tenmsb\fam\msbfam#1}}
\newtheorem{thm}{Theorem}[section]
\newtheorem{lem}{Lemma}[section]
\newtheorem{cor}{Corollary}[section]
\newtheorem{rem}{Remark}[section]
\newtheorem{pro}{Proposition}[section]
\newtheorem{defi}{Definition}[section]

\newcommand{\ba}{\begin{array}}
\newcommand{\ea}{\end{array}}

\newcommand{\Section}[2]{\setcounter{equation}{0}
\allowdisplaybreaks
\section[#1]{#2}}

\def\n{\nabla}
\def\bn{\overline\nabla}
\def\ir#1{\mathbb R^{#1}}

\def\f#1#2{\frac{#1}{#2}}

\def\grs#1#2{\bold G_{#1,#2}}

\def\pd#1#2{\frac {\partial #1}{\partial #2}}

\def\a{\alpha}
\def\be{\beta}

\def\p#1{\partial #1}

\def\de{\delta}
\def\De{\Delta}
\def\e{\eta}
\def\ep{\varepsilon}

\def\G{\Gamma}
\def\g{\gamma}
\def\k{\kappa}
\def\la{\lambda}
\def\La{\Lambda}
\def\om{\omega}
\def\Om{\Omega}
\def\th{\theta}
\def\Th{\Theta}
\def\si{\sigma}
\def\Si{\Sigma}
\def\r{\rho}

\def\w{\wedge}

\def\Hess{\mbox{Hess}}
\def\R{\mathbb R}

\def\H{\mathcal{H}}

\def\lan{\langle}
\def\ran{\rangle}
\def\ra{\rightarrow}

\def\bn{\bar{\nabla}}

\subjclass{53A07, 49Q15.}

\begin{document}
\title
[Minimal graphs in Euclidean space]
{Minimal graphs of arbitrary codimension in Euclidean space with bounded $2-$dilation}

\author{Qi Ding}
\address{Shanghai Center for Mathematical Sciences, Fudan University, Shanghai 200438, China}
\email{dingqi@fudan.edu.cn}
\author{J. Jost}
\address{Max Planck Institute for Mathematics in the Sciences, Inselstr. 22, 04103 Leipzig, Germany}
\email{jost@mis.mpg.de}
\author{Y.L. Xin}
\address{Institute of Mathematics, Fudan University, Shanghai 200433, China}
\email{ylxin@fudan.edu.cn}

\thanks{The  authors  are  partially
supported by NSFC and SFMEC }

\begin{abstract}
For any $\La>0$, let $\mathcal{M}_{n,\La}$ denote the space containing all locally Lipschitz minimal graphs of dimension $n$ and of arbitrary codimension $m$ in Euclidean space $\R^{n+m}$ with uniformly bounded
2-dilation $\La$ of their graphic functions. In this paper, we show that this is a natural class to extend structural results known for codimension one. In particular,  
we prove that any tangent cone $C$ of $M\in \mathcal{M}_{n,\La}$ at infinity has multiplicity one. This enables us to get a 
Neumann-Poincar$\mathrm{\acute{e}}$ inequality on stationary indecomposable components of $C$. A corollary is a Liouville theorem for $M$.
For  small $\La>1$(we can take any $\La <\sqrt{2}$), we prove that (i) for $n\le7$, $M$ is flat; (2) for $n>8$ and a non-flat $M$, any tangent cone of
$M$ at infinity is a multiplicity one quasi-cylindrical minimal cone in $\R^{n+m}$ whose singular set has dimension $\le n-7$. 
\end{abstract}
\maketitle

\maketitle \tableofcontents
\Section{Introduction}{Introduction}

It has been a central aim of the theory of minimal graphs in Euclidean space
to derive conditions under which an entire $n$-dimensional minimal graph, that
is, a graph defined on all of $\R^n$, of codimension $m$, that is, sitting in
$\R^{n+m}$, is affine linear. This is the famous Bernstein problem. Bernstein himself proved it for two-dimensional entire minimal graphs in $\R^3$. For codimension 1, but in 
higher dimensions, this holds for $n\le 7$ through  successive efforts of
Fleming \cite{f}, De Giorgi \cite{dg}, Almgren \cite{al} and culminating with Simons \cite{ss}. However, it is no longer
true for $n\ge 8$ by an example of Bombieri, De Giorgi and Giusti \cite{b-d-g}.
If we assume, however, in addition, that the graph has a
bounded gradient, then this holds for any dimension $n$ by a result of Moser \cite{m}. This is the 
so-called weak version of the Bernstein Theorem.

Research on the  Bernstein theorem was a crucial motivation for the great development of geometric measure theory.
It is well known that an entire codimension 1 minimal graph $\Si$ in $\R^{n+1}$ is area-minimizing, i.e., the current associated with $\Si$ is a minimizing current.
Fleming \cite{f} proved that any tangent cone $C_\Si$ of $\Si$ at infinity is a singular area-minimizing cone in $\R^{n+1}$, which implies that it is a stable minimal hypersurface with multiplicity one.
De Giorgi \cite{dg} further showed that $C_\Si$ is \emph{cylindrical}, i.e,
$C_\Si$  isometrically splits off a factor $\R$.

In higher codimension,
Almgren \cite{af} derived  sharp codimension 2 estimates for the singular sets of minimizing currents.
In \cite{ds1,ds2,ds3}, De Lellis-Spadaro developed a new approach to the regularity of minimizing currents and could in particular rederive Almgren's structure theorem. In general, however, minimal graphs of higher codimension are not minimizing. Nevertheless, some general structural results about minimal graphs in higher codimension are available.
Utilizing the
graph property, it is possible to study the structure and rigidity of minimal graphs of arbitrary codimension under some conditions but without a minimizing assumption.

In this paper, we approach this issue via studying tangent cones of minimal graphs at infinity.
First of all, we need some condition on minimal graphs to guarantee that they have Euclidean volume growth.
For that purpose, we now introduce the concept of bounded
$k$-dilation. As we shall see, this condition provides a natural generalization of the bounded slope condition that has been used in many other papers about higher codimension Bernstein theorems.

Let $f:\, M_1\rightarrow M_2$ be a locally Lipschitz map between
Riemannian manifolds $M_1,M_2$.
For an integer $k\ge1$, $f$ is said to have $k$-\emph{dilation} $\le\La$ for
some constant $\La\ge0$ if $f$ maps each $k$-dimensional submanifold $S\subset
M_1$ to an image $f(S)\subset M_2$ with
$\mathcal{H}^k(f(S))\le\La\mathcal{H}^k(S)$ where $\mathcal{H}^k$ denotes
$k$-dimensional Hausdorff measure. We note that this condition is the more
restrictive the smaller $k$ (up to the constant $\La$). For $k=1$, it simply means that $f$ is
$\La$-Lipschitz. This is, of course, a strong condition. In this paper, we
shall explore the case $k=2$, that is, consider maps, or more precisely locally Lipschitz
minimal graphs given by $f:\R^n\to \R^m$ with bounded 2-dilation.
But even for $n=m=2$, minimal graphs have not necessarily bounded 2-dilation for their graphic functions, as
the 2-dimensional minimal graphs given by $(\mathrm{Re}\, e^z,-\mathrm{Im}\, e^z):\,\R^2\rightarrow \R^2$.
We should note,
however, that these examples have zero $k$-dilation for $k\ge
3$. Therefore, we cannot hope for a good theory on tangent cones of minimal graphs at infinity when we only restrict some $k$-dilation for $k\ge
3$.

For any constant $\La\ge0$,
let $\mathcal{M}_{n,\La}$ denote the space containing all the locally Lipschitz
minimal graphs over $\R^n$ of arbitrary codimension $m\ge1$ with 2-dilation of
their graphic functions $\le\La$. 
Here, the codimension $m$ is bounded by a constant depending only on $n,\La$ using a result of Colding-Minicozzi \cite{c-m} (which will be explained later), and this is the reason why the notation $\mathcal{M}_{n,\La}$ does not contain $m$.
De Giorgi \cite{dg0} obtained the regularity of codimension 1 locally Lipschitz minimal graphs (see also Moser \cite{m0}), while high codimensional locally Lipschitz minimal graphs may really have singularities.
We can already provide the following geometric intuition.
\begin{enumerate}

\item $\mathcal{M}_{n,\La}$ contains every codimension one minimal graph in $\R^{n+1}$.

\item The product of Euclidean space $\R^{\ell}$ and any minimal graphical hypersurface in $\R^{n+1}$ is contained in $\mathcal{M}_{n,\La}$. 

\item $\sqrt{\La}$-Lipschitz graphs are in
$\mathcal{M}_{n,\La}$. In particular, the minimal Hopf cones in $\R^{2m}\times\R^{m+1}$ for $m=2,4,8$
constructed by Lawson-Osserman \cite{l-o} are in some
$\mathcal{M}_{n,\La}$. See \cite{d-y}\cite{x-y-z} for more examples.

\end{enumerate}
However, having a bounded Lipschitz constant is a much stronger condition than having bounded 2-dilation. For bounded domains  $\Om\subset\R^n$, we have constructed many examples of $n$-minimal graphs over $\Om$ of 2-dilation $\le1$ with arbitrary large slope, where they do not live in any $(n+1)$-dimensional Euclidean subspace \cite{d-j-x}.

Having thus sketched the basic setting, we can explain the two main objectives of this paper. The first objective is the development of the theory of minimal graphs
of arbitrary codimension in Euclidean space  with uniformly bounded
2-dilation of their graphic functions.
The principal aim is to understand the geometric structure, including
multiplicity and stability, of such minimal graphs at infinity when the
graphic functions are allowed to grow faster than linearly.
Without Euclidean volume growth, geometric measure theory cannot say much about the possible limits of minimal graphs at infinity.
Therefore, our first crucial issue will be to derive Euclidean volume growth from bounded 2-dilation.
Importantly, their tangent cones at infinity have multiplicity one (see Theorem \ref{main2} below), which plays an essential role in establishing the
Neumann-Poincar$\mathrm{\acute{e}}$ inequality on
stationary indecomposable components of the tangent cones.

This general theory prepares the ground for the second objective,
to study the rigidity of minimal graphs of arbitrary codimension in Euclidean space without the assumption of the bounded slope.
Under the condition of bounded 2-dilation for graphic functions, we prove a Liouville theorem for minimal graphs (see Theorem \ref{main1}) with Neumann-Poincar$\mathrm{\acute{e}}$ inequality, which generalizes the classical Liouville theorem obtained by Bombieri-De Giorgi-Miranda \cite{b-g-m}.
Concerning the
Bernstein theorem, we prove that there exists a constant $\La>1$ such that for $n\le7$ each $M\in\mathcal{M}_{n,\La}$ is flat, for $n\ge8$ and any non-flat $M\in\mathcal{M}_{n,\La}$, any tangent cone of $M$ at infinity is a quasi-cylindrical minimal cone (see Theorem \ref{main4}).
Here, the dimension 7 is sharp by the counter-examples by Bombieri-De Giorgi-Giusti \cite{b-d-g}, and a quasi-cylindrical minimal cone is exactly a
cylinder  in the
codimension one case (see the definition below).
Moreover, the constant $\La$ can be arbitrarily chosen $<\sqrt{2}$. The constant $\sqrt{2}$ is not essential for our whole theory but plays a role in making the volume functional subharmonic (see Corollary \ref{Lalogv}).

Let us now be more specific.
In \cite{c-l-y}, Cheng-Li-Yau estimated the codimension for each minimal cone in Euclidean space via its density.
Colding-Minicozzi \cite{c-m} proved the dimension estimates for coordinate functions on more general minimal submanifolds of Euclidean volume growth.
In \S 3, using the result of  \cite{c-m}, we prove that every
$M\in\mathcal{M}_{n,\La}$ has Euclidean volume growth with the density bounded
by a constant depending only on $n,\La$ (see Lemma \ref{VGM*} for details). In
particular, $M$ lives in a Euclidean subspace with the codimension bounded by a constant depending on $n,\La$.
Without the condition of bounded 2-dilation, of course minimal graphs need not
have Euclidean volume growth, like the  2-dimensional minimal graph by
$(\mathrm{Re}\, e^z,-\mathrm{Im}\, e^z):\,\R^2\rightarrow \R^2$ that we have already mentioned above (see Remark \ref{LEVG} for details).

Let $v$ denote the {\it slope} of a minimal graph $M$ over $\R^n$ defined by $\sqrt{\mathrm{det} g_{ij}}$, where
$g_{ij}dx_idx_j$ is the metric of $M$ induced from the ambient Euclidean space $\R^{n+m}$. In fact, $v=\tilde{v}\circ\g$, where $\tilde{v}$ is a natural function in the Grassmannian manifold $\grs{n}{m}$ and $\g$ stands for the Gauss map. The slope describes how far $M$ is from the fixed $n$-plane $\R^n$.
We will explain its geometric meaning later in detail from the perspective of  the
Grassmannian manifold.
For $0\le\La<1$ and $M\in \mathcal{M}_{n,\La}$ with codimension $m\ge1$, the corresponding $\tilde{v}$ is convex in  $\grs{n}{m}$, and then $v=\tilde{v}\circ\g$ is
subharmonic. This leads to flatness of $M$ proved by M.T. Wang \cite{wang} under the bounded slope condition.
When $\La>1$, the function $\tilde{v}$ in general
is not convex. In fact, our  class $\mathcal{M}_{n,\La}$ is much richer when $\La> 1$ is larger.

If $M\in\mathcal{M}_{n,\La}$ is minimizing, then it is not difficult to prove the multiplicity one of the tangent cone of $M$ at infinity. However, the multiplicity one holds without the minimizing condition. Moreover, we can show the stablity of the tangent cone in some small neighborhood via the slope function $v$ in \S 4.
\begin{thm}\label{main2}
Let $M$ be a locally Lipschitz minimal graph over $\R^n$ of codimension $m\ge1$ with bounded
2-dilation of its graphic function. Then any tangent cone $C$ of $M$ at infinity has multiplicity one. Moreover, if a tangent cone of $C$ contains a line perpendicular to the $n$-plane $\{(x,0^m)\in\R^n\times\R^m|\, x\in\R^n\}$, then it is a cylindrical stable minimal cone in an $(n+1)$-dimensional Euclidean subspace.
\end{thm}
\begin{rem}
We do not know whether the cylindrical stable minimal cone obtained in Theorem \ref{main2} is minimizing. Even we have no a priori dimensional estimates on its singular set of the cylindrical stable minimal cone.
\end{rem}
In fact, we prove a somewhat stronger version than Theorem \ref{main2}, where
 we do not  need to require that the graphs are entire.
The proof of multiplicity one is somewhat difficult because of the complex interplay between geometry and analysis for
submanifolds of high codimensions.
Our strategy therefore consists in using
the structure of the minimal surface system to treat
the higher codimension case as a perturbation of the codimension one case with error terms including some quantities
from the other codimensions.
The key idea is projecting the minimal graph $M$ to a hypersurface $M'$ in a suitable $(n+1)$-dimensional Euclidean subspace.
In general $M'$ is no more minimal, but from $M$ we can get effective estimates up to a set of arbitrary small measure in the scaling sense.

In 1969, Bombieri-De Giorgi-Miranda \cite{b-g-m} showed a Liouville theorem for solutions to the minimal
surface equation via interior gradient estimates (see also the exposition in chapter 16 of \cite{g-t}).
For high codimensions, M.T. Wang proved a Liouville type theorem for minimal graphs with positive
graphic functions under the area-decreasing condition \cite{w2}. This
condition also means that the graphic functions have $2$-dilation bounded by $\La<1$.

Let $\mathbf{B}_r$ denote the ball in $\R^{n+m}$ with the radius $r$ and centered at the origin.
Let $B_r$ denote the ball in $\R^{n}$ with the radius $r$ and centered at the origin.
{Inspired} by Bombieri-Giusti  \cite{b-g}, we establish the Neumann-Poincar$\mathrm{\acute{e}}$ inequality on stationary indecomposable components of tangent cones of minimal graphs at infinity, and get Harnack's inequality for positive harmonic functions on the components.
Then we can get the following Liouville theorem without the subharmonic functions in terms of the gradient functions on minimal graphs in \S 6, which generalizes Bombieri-De
Giorgi-Miranda's result in \cite{b-g-m}, and improves Wang's result in \cite{w2}.
\begin{thm}\label{main1}
Let $M=\mathrm{graph}_u$ be a locally Lipschitz minimal graph over $\R^n$ of codimension $m\ge2$ with bounded
2-dilation of $u=(u^1,\cdots,u^m)$.
If
\begin{equation}\aligned\label{ua**0}
\limsup_{r\rightarrow\infty}\left(r^{-1}\sup_{\mathbf{B}_r\cap M}u^\a\right)\le0
\endaligned
\end{equation}
for each $\a\in\{2,\cdots,m\}$, and
\begin{equation}\aligned\label{u1**0}
\liminf_{r\rightarrow\infty}\left(r^{-1}\sup_{B_r}u^1\right)<\infty,
\endaligned
\end{equation}
then $M$ is flat.
\end{thm}
In the proof of Theorem \ref{main1}, we need a De Giorgi type result, i.e., every tangent cone of $M$ at infinity is a cylinder if it lives in an $(n+1)$-dimensinoal Euclidean subspace. Using this, we can reduce the problem to the codimension 1 case with the Neumann-Poincar$\mathrm{\acute{e}}$ inequality. Then we can finish the proof by the regularity of codimension 1 Lipschitz minimal graphs from De Giorgi \cite{dg0}. 

According to this  introduction, people have an essentially complete understanding of the classical case of codimension 1 Bernstein theoerm. The question then naturally arises what we can say for higher codimension.

We can now  briefly review the Bernstein type theorems for minimal graphs of bounded slope of codimension $m\ge2$. Bounded slope condition is an adequate generalization of bounded gradient to higher codimension.
For $m\ge2$, Chern and Osserman \cite{c-o} showed that any 2-dimensional minimal graph of bounded slope in $\R^{2+m}$ is flat, which was generalized in \cite{h-s-v}(for $m=2$) and \cite{j-x-y1} without bounded slope.
Barbosa and Fischer-Colbrie proved this for 3-dimensional minimal graphs of bounded slope in \cite{b}\cite{fc}.
Recently, Assimos-Jost \cite{a-j} proved a
Bernstein type theorem for minimal graphs of bounded slope in codimension $m=2$.
For dimension $\ge4$ and codimension
$\ge3$, this no longer holds, by an example of Lawson and Osserman \cite{l-o}.
However, Bernstein type theorem holds for small slope $v$ such as Simons \cite{s}, Hildebrandt-Jost-Widman \cite{h-j-w}, Jost-Xin \cite{j-x}, Jost-Xin-Yang
\cite{j-x-y}. In particular, any minimal graph of slope $\le3$ in Euclidean space is flat \cite{j-x-y}.

But we may also ask whether there exist other natural conditions that ensure a higher codimension Bernstein theorem.
We point out that when $v=\tilde{v}\circ\g\le 3$, $\tilde{v}$ need no longer be convex
on the Grassmannian manifold, the target manifold of the Gauss map of
$M$. Note that $v=9$ in Lawson-Osserman's example mentioned above.
Hence, if every minimal graph $M\in\mathcal{M}_{n,\La}$ is flat, then $\La$ must be small.
Without the conditions \eqref{ua**0}\eqref{u1**0} of Theorem \ref{main1}, we can study the structure of tangent cones of minimal graphs at infinity for small $\La>1$ using Theorem \ref{main2} and the Neumann-Poincar$\mathrm{\acute{e}}$ inequality.

Now let us introduce a concept 'quasi-cylindrical' for studying unbounded slope case.
For an integer $1\le k\le n$ and a $k$-varifold $V$ in $\R^{n+m}$, $V$ is said to be {\it high-codimensional quasi-cylindrical} ({\it quasi-cylindrical}
for short) 
if there are a countably $(k-1)$-rectifiable set $E$ in $\R^n$, and there is a countably $1$-rectifiable normalized curve $\g_x:\,\R\rightarrow\R^m$ for almost all $x\in E$ such that the set $E_\g\triangleq\{(x,y)\in\R^n\times\R^m|\, x\in E,\, y\in\g_x\}$ satisfies
$$\mathcal{H}^n\left((\mathrm{spt}V\setminus E_\g)\cup(E_\g\setminus \mathrm{spt}V)\right)=0.$$
Note that  'quasi-cylindrical' is simply 'cylindrical' for $m=1$.
Hence, quasi-cylindrical varifolds can be seen as a generalization of cylindrical varifolds in the case of high codimensions in Euclidean space.
\begin{thm}\label{main4}
There exists a constant $\La>1$ such that if $M$ is a locally Lipschitz minimal graph over $\R^n$ of codimension $m\ge1$ with
2-dilation of its graphic function bounded by $\La$, then either $M$ is flat, or $M$ is non-flat with $n\ge8$. Furthermore, for non-flat $M$, any tangent cone of
$M$ at infinity is a multiplicity one quasi-cylindrical minimal cone in $\R^{n+m}$ whose singular set has dimension $\le n-7$. 
\end{thm}
In the codimension one case, Theorem \ref{main4} for $n\le7$ has been proved by Simons \cite{ss}, and the dimension 7 is sharp from \cite{b-d-g}. So Theorem \ref{main4} is a generalization in high codimension.
Actually, our proof is based on Simons' result that any stable minimal regular hypercone $C$ in $\R^{k}$ is flat for $k\le7$, where Simons' result holds allowing that $C$ has singularities in some sense, see \cite{wn}. For the case $n\ge8$ of Theorem \ref{main4}, the singular set of dimension $\le n-7$ is obtained through stable minimal hypercones combining Theorem \ref{main2}.
The only purpose of the upper bound  for
the constant $\La$ in Theorem \ref{main4} is to ensure that there is a constant $\de>0$ (which
may depend on $\La,n,m$) such that there holds
\begin{equation}\aligned\label{DeMdeBM**0}
\De_M\log v\ge\de|B_M|^2
\endaligned
\end{equation}
on the minimal graph $M$ (see Corollary \ref{Lalogv}), where $\De_M,B_M$ denote the Laplacian, the second fundamental form of $M$, respectively.
Noting that \eqref{DeMdeBM**0} always holds with $\de=1$ for $n=2$ no matter how large $\La$ is (see Proposition 2.2 in \cite{fc} for instance). 
Hence, Theorem \ref{main4} immediately implies the following Bernstein theorem for 2-dimensional entire minimal graphs, which generalizes the results in \cite{c-o}\cite{h-s-v}.
\begin{cor}
Let $M$ be a locally Lipschitz minimal graph over $\R^2$ of codimension $m\ge1$ with bounded
2-dilation of its graphic function, then $M$ is flat.
\end{cor}

In general, the constant $\La$ in Theorem \ref{main4} can be arbitrarily chosen $<\sqrt{2}$, and we do not know whether $\sqrt{2}$ is sharp, though it appears naturally for the subharmonicity of $\log v$.
In fact, we can find a slightly weaker condition for the
Bernstein theorem in
all dimensions in the situation of bounded slope (see Theorem \ref{BDMGAFFINE}).

\Section{Preliminaries}{Preliminaries}

Let $\R^k$ denote the Euclidean space for each integer $k\ge1$, and $0^k$ denote the origin of $\R^k$.
Let $B^k_r(x)$ be the ball in $\R^k$ with the radius $r$ and centered at $x\in\R^k$, and $B_r(x)=B^n_r(x)$ for convenience.
Let $\mathbf{B}_r(\mathbf{x})$ be the ball in $\R^{n+m}$ with the radius $r$ and centered at $\mathbf{x}\in\R^{n+m}$.
We denote $B_r=B_r(0^n)$, $\mathbf{B}_r=\mathbf{B}_r(0^{n+m})$ for convenience.
We always use $D$ to  denote the derivative on $\R^n$.
For any subset $E$ in $\R^n$, for any constant $0\le s\le n$ we define
$\mathcal{H}^s(E)$ to be the $s$-dimensional Hausdorff measure of $E$.
Let $\omega_k$ denote the $k$-dimensional Hausdorff measure of $B^k_1(0^k)$.
We use the summation convention and agree on the ranges of indices:
$$1\leq  i,j,l\leq n,\; 1\leq \a,\be\leq m$$
unless otherwise stated.

Let $M$ be an $n$-dimensional smooth Riemannian manifold, and $M\ra \R^{n+m}$ be an isometric immersion. Let $\n$ and $\bn$ be Levi-Civita connections on $M$ and $\R^{m+n}$, respectively.
Here, $\n$ is induced from $\bn$ naturally.
The second fundamental form $B_M$ on the submanifold $M$ is defined by $B_M(\xi,\e)=\bn_\xi\e-\n_\xi\e=(\bn_\xi\e)^N$ for any vector fields $\xi,\e$ along $M$, where $(\cdots)^N$
denotes the projection onto the normal bundle $NM$ (see \cite{x} for instance).
Let $e_1,\cdots,e_n$ be a local orthonormal frame field near a considered
point in $M$. Let $|B_M|^2$ denote the square norm of $B_M$, i.e.,
$|B_M|^2=\sum_{i,j=1}^n|B_M(e_i,e_j)|^2$ (This notation should not be confused
with that  of the  ball $B_r$).
Let $H$ denote the mean curvature vector of $M$ in $\R^{n+m}$ defined by the trace of $B_M$, i.e., $H=\sum_{i=1}^nB_M(e_i,e_i)$.
This is a normal vector field.
$M$ is said to be \emph{minimal} if $H\equiv0$ on $M$. More generally, $M$ has parallel mean curvature if $\n H\equiv0$.

\subsection{Grassmannian manifolds and Gauss maps}
We will study minimal submanifolds in ambient Euclidean space. The target manifolds
of the Gauss map of minimal submanifolds are Grassmannian manifolds. For convenience of later
application, let us describe the geometry of the Grassmannian manifolds.
In $\R^{n+m}$ all the
oriented $n$-subspaces constitute the Grassmann manifold
$\grs{n}{m}$, which is the Riemannian symmetric space of compact
type $SO(n+m)/SO(n)\times SO(m)$, where $SO(k)$ denotes the $k$-dimensional special orthogonal group for each integer $k$.
$\grs{n}{m}$ can be viewed as a submanifold of some  Euclidean space
via the Pl\"ucker embedding. The restriction of the
Euclidean inner product on $M$ is denoted by $w:\grs{n}{m}\times \grs{n}{m}\ra \R$
\begin{equation}\aligned\label{www}
w(P,Q)=\lan e_1\w\cdots\w e_n,f_1\w\cdots\w f_n\ran=\det W,
\endaligned
\end{equation}
where $P$ is presented by a unit $n$-vector $e_1\w\cdots\w e_n$, $Q$ is presented  by another unit $n$-vector $f_1\w\cdots
\w f_n$, and  $W=\big(\lan e_i,f_j\ran\big)$ is an $(n\times n)$-matrix. It is well-known that
$$W^T W=O^T \left(\begin{array}{ccc}
            \mu_1^2 &   &  \\
                    & \ddots &  \\
                    &        & \mu_n^2
            \end{array}\right) O$$
with $O$ an orthogonal matrix and $0\leq \mu_i^2\leq 1$ for each $i$. Putting $p\triangleq\min\{m,n\}$, then
at most $p$ elements in $\{\mu_1^2,\cdots, \mu_n^2\}$ are not
equal to $1$. Without loss of generality, we can assume
$\mu_i^2=1$ whenever $i>p$.

For a unit vector $\xi=\sum_i a_i e_i\in P$, let $\xi^*$ denote its projection into $Q$, i.e.,
$$\xi^*=\sum_{j,l} a_j\lan e_j,f_l\ran f_l.$$
Then
\begin{equation}\aligned\nonumber
\lan\xi,\xi^*\ran=\sum_{i,j,l}a_ia_j\lan e_i, f_l\ran\lan e_j, f_l\ran.
\endaligned
\end{equation}
Certainly, the matrix $(\sum_l\lan e_i,f_l\ran\lan e_j,f_l\ran)_{i,j}$ has eigenvalues $\mu_1^2,\cdots,\mu_n^2.$
Hence, we can introduce the Jordan angles $\th_1,\cdots,\th_n$ between two points $P,\; Q\in \grs{n}{m}$ defined by
$$\th_i=\arccos(\mu_i), \qquad 1\leq i\leq n,$$
which are actually critical values of the angle between a nonzero vector $\xi$ in $P$ and its
orthogonal projection $\xi^*$ in $Q$ as $\xi$ runs through $P$ (see Wong \cite{w} or Xin \cite{x} for further details).

We also note that the $\mu_i^2$ can be expressed as
\begin{equation}\label{di1a}
\mu_i^2=\cos^2\th_i=\frac{1}{1+\la_i^2},
\end{equation}
so that
\begin{equation}\label{di2}
\la_i=\tan\th_i, \end{equation}
where $\la_i$ has explicit meaning in studying graphs of high codimensions (We will explain it later).
The distance between $P$ and $Q$ is defined by
\begin{equation}\label{di}
d(P, Q)=\sqrt{\sum\th_i^2}. \end{equation}
Let $E_{i\a}$ be the matrix with $1$
in the intersection of row $i$ and column $\a$ and $0$ otherwise.
Then, $\sec\th_i\sec\th_\a E_{i\a}$ form an orthonormal basis of
$T_P\grs{n}{m}$ with respect to the canonical Riemannian metric $g$ on $\grs{n}{m}$ (compatible to \eqref{di}). Denote its dual frame
by $\tilde{\om}_{i\a}.$ Then $g$ can be written as
\begin{equation}\label{m2}
g=\sum_{i, \a}\tilde{\om}_{i\a}^2.
\end{equation}
Denote $\tilde{\om}_{\be\a}=0$ for $\be\ge n+1$.

Now we fix $P_0\in \grs{n}{m}.$
Denote
\begin{equation}\label{BbbW0PP0}
\Bbb{W}_0:=\{P\in \grs{n}{m}|\, w(P,P_0)>0\},
\end{equation}
where $w$ is defined in \eqref{www}. The Jordan angles between $P$ and $P_0$ are defined by $\{\th_i\}$.
Let $\Bbb{T}^{2,\La}$ be a $2-$bounded subset of $\Bbb{W}_0$ defined by
\begin{equation}\label{BbbT2La}
\Bbb{T}^{2,\La}=\{P\in \Bbb{W}_0|\, \tan\th_i\tan\th_j<\La\ \mathrm{for\ every}\ i\neq j\}.
\end{equation}
In \cite{j-x}, we already have the largest geodesic convex subset
$B_{JX}(P_0)$, which is defined by sum of any two Jordan angles being less than $\pi\over 2$
for any point $P\in B_{JX}(P_0)$. It is easily seen that
$$\Bbb{T}^{2,1}=B_{JX}(P_0).$$
Hence, the distance function from $P_0$ is convex on $\Bbb{T}^{2,1}$, but is no longer convex on $\Bbb{T}^{2,\La}$ when $\La>1$.

Our fundamental quantity will be
\begin{equation}
\tilde{v}(\cdot, P_0):=w^{-1}(\cdot, P_0)\qquad \text{ on }\;\Bbb{W}_0.
\end{equation}
It is easily seen that
\begin{equation}\label{v}
\tilde{v}(P,P_0)=\prod_{i=1}^n
\sec\th_i=\prod_{i=1}^n \sqrt{1+\la_i^2},
\end{equation}
where $\th_1,\cdots,\th_n$ denote the Jordan angles between $P$
and $P_0$.
In this terminology, from (3.8) in \cite{x-y1}, we get
\begin{equation}\label{dw3}
d\tilde{v}(\cdot, P_0) =\sum_{1\leq j\leq p}\la_j\,\tilde{v}(\cdot, P_0)\tilde{\om}_{jj},
\end{equation}
and
\begin{equation}\label{hess}
\tilde{v}(\cdot,P_0)^{-1}\Hess(\tilde{v}(\cdot,P_0))
=g+\sum_{\a,\be}\la_\a\la_\be(\tilde{\om}_{\a\a}\otimes \tilde{\om}_{\be\be}+
\tilde{\om}_{\a\be}\otimes \tilde{\om}_{\be\a}).
\end{equation}
Combining \eqref{dw3}\eqref{hess}, it follows that (see (3.9) in \cite{x-y1} for instance)
\begin{equation}\label{Hel}
\Hess \log \tilde{v}(\cdot,P_0)=g+\sum_{1\leq j\leq p}\la_j^2\tilde{\om}_{jj}^2
+\sum_{1\leq i,j\leq p,i\neq j}\la_i\la_j\tilde{\om}_{ij}\otimes \tilde{\om}_{ji}.
\end{equation}

Let $M$ be an $n$-dimensional smooth submanifold in $\R^{n+m}$.
Around any point $p\in M$, we choose an
orthonormal frame field $e_i,\cdots, e_{n+m}$ in $\R^{n+m},$ such
that $\{e_i\}$ are tangent to $M$ and $\{e_{n+\a}\}$ are normal to
$M.$
Let $\{\om_1,\cdots,\om_{n+m}\}$ be its dual frame field
so that the metric on $M$ is $\sum_i \om_i^2$ and the Euclidean metric in $\ir{n+m}$ is
$$\sum_{i}\om_i^2+\sum_{\a}\om_{n+\a}^2.$$
The Levi-Civita  connection forms $\om_{ab}$ of $\ir{n+m}$ are
uniquely determined by the equations
$$\aligned
&d\om_{a}=\om_{ab}\wedge\om_b,\cr &\om_{ab}+\om_{ba}=0,
\endaligned$$
where $a, b=1,\cdots, n+m$.
Moreover, we have the equations
\begin{equation}\label{str}
\om_{i\  n+\a}=h_{\a, ij}\om_j,
\end{equation}
where $h_{\a, ij}=\lan \bn_{e_i}e_j,e_{n+\a}\ran$ are the coefficients of the second fundamental form $B_M$ of $M$ in $\R^{n+m}.$
The Gauss map $\g: M\ra \grs{n}{m}$ is defined by
$$\g(p)=T_p M\in \grs{n}{m}$$
via the parallel translation in $\R^{n+m}$ for every $p\in M$. We also have
\begin{equation}\label{edg}
|d\g|^2=\sum_{\a,i,j}h_{\a, ij}^2=|B_M|^2.
\end{equation}
Up to an isotropic group $SO(n)\times SO(m)$ action, we can assume $\om_{i\  n+\a}=\g^*\tilde{\om}_{i \a}$ at $p$
(see section 8.1 in \cite{x} for instance).
Combining (\ref{str}), we obtain
\begin{equation}\label{hij}
\g^*\tilde{\om}_{i\a}=h_{\a, ij}\om_j\qquad \mathrm{at}\ p.
\end{equation}
By the Ruh-Vilms theorem \cite{r-v},  the mean curvature of $M$
is parallel if and only if its Gauss map $\g$ is a harmonic map.

Now we define a function
\begin{equation}
v\triangleq \tilde{v}(\cdot,P_0)\circ \g\qquad \mathrm{on}\ \ M,
\end{equation}
which will play a basic role in this paper.
Using the composition formula, in conjunction with (\ref{Hel}),
(\ref{edg}) and (\ref{hij}), and the fact that  $\tau(\g)=0$ (the
tension field of the Gauss map vanishes \cite{r-v}), we can deduce the following
important formula (see also Lemma 1.1 in \cite{fc} or Prop. 2.1 in \cite{wang}).
\begin{pro}
Let $M$ be an $n$-dimensional smooth submanifold in $\ir{n+m}$ with parallel mean curvature. Then at any considered point $p$
\begin{equation}\label{Delogv}
\De_M\ln v=|B_M|^2+\sum_{i,j}\la_i^2h_{i,i j}^2 +\sum_{l,i\neq
j}\la_i\la_jh_{i,j l}h_{j,i l},
\end{equation}
where $\De_M$ is the Laplacian on $M$, $h_{\a,ij}$ are defined in \eqref{str}. Let $\n_M$ denote the Levi-Civita connection of $M$. Combining \eqref{dw3} and \eqref{Delogv}, one has
\begin{equation}\aligned\label{DeMv-1}
\De_M v^{-1}=&\De_M e^{-\log v}=-v^{-1}\De_M\log v+v^{-1}|\n_M\log v|^2\\
=&-v^{-1}\left(\sum_{\a,i,j}h^2_{\a,ij}+\sum_{l,i\neq j}\la_i\la_jh_{i,jl}h_{j,il}-\sum_{l,i\neq j}\la_i\la_jh_{i,il}h_{j,jl}\right).
\endaligned
\end{equation}
\end{pro}

\subsection{Varifolds and currents from geometric measure theory}

Let us recall Almgren's notion of {\it varifolds} from  geometric measure
theory (see \cite{l-ya}\cite{s} for more details), which is a
generalization of submanifolds.
For a set $S$ in $\R^{n+m}$, we call $S$ \emph{$n$-rectifiable} if $S\subset
S_0\cup F_1(\R^n)$, where $\mathcal{H}^n(S_0)=0$, and $F_1:\,
\R^n\rightarrow\R^{n+m}$ is a Lipschitz mapping.
More general, we call $S$ \emph{countably $n$-rectifiable} if $S\subset S_0\cup\bigcup_{k=1}^\infty F_k(\R^n)$, where $\mathcal{H}^n(S_0)=0$, and $F_k:\, \R^n\rightarrow\R^{n+m}$ are Lipschitz mappings for all integers $k\ge1$.
By Rademacher's theorem, a countably $n$-rectifiable set has tangent spaces at almost every point.
Suppose $\mathcal{H}^n(S)<\infty$. Let $\th$ be a positive locally $\mathcal{H}^n$ integrable function on $S$.
Let $|S|$ be the varifold associated with the set $S$.
The associated varifold $V=\th|S|$ is called a \emph{rectifiable $n$-varifold}. $\th$ is called the \emph{multiplicity} function of $V$. In particular, the multiplicity of $|S|$ equal to one on $S$.
If $\th$ is integer-valued, then $V$ is said to be an \emph{integral varifold}. Associated to $V$, there is a Radon measure $\mu_V$ defined by $\mu_V=\mathcal{H}^n\llcorner\th$, namely,
$$\mu_V(W)=\int_{W\cap S}\th(y)d\mathcal{H}^n(y)\qquad \mathrm{for\ each\ open}\ W\subset\R^{n+m}.$$

For an open set $U\subset\R^{n+m}$, $V$ is said to be \emph{stationary} in $U$ if
\begin{equation}\aligned
\int \mathrm{div}_S Yd\mu_V=0
\endaligned
\end{equation}
for each $Y\in C^\infty_c(U,\R^{n+m})$. Here, $\mathrm{div}_S Y$ is the divergence of $Y$ restricted on $S$.
When we say an $n$-dimensional \emph{minimal cone} $C$ in $\R^{n+m}$, we mean that $C$ is an integral stationary varifold with support being a cone.
One of the most important properties of the stationary varifold $V$ is that the function
\begin{equation}\aligned\label{ratio}
\r^{-n}\mu_V(\mathbf{B}_\r(\mathbf{x}))
\endaligned
\end{equation}
is monotone non-decreasing
for $0<\r<\r_0$ with $\r_0\le d(\mathbf{x},\p V)$ and $\mathbf{x}\in\R^{n+m}$.
By Rademacher's theorem, we can define the derivative $\n_V$ on $V$ for
Lipschitz functions almost everywhere (see Definition 12.1 in \cite{s} for
instance).

Let $\{V_j\}_{j\ge0}$ be a sequence of integral stationary $n$-varifolds with the Radon measure $\mu_{V_j}$ associated to $V_j$ satisfying
$$\sup_{j\ge1}\mu_{V_j}(W)<\infty,\qquad\mathrm{for\ each}\ W\subset\subset U.$$
By compactness theorem of varifolds (see Theorem 42.7 and its proof in \cite{s}), there are a subsequence $V_{j'}$ and an integral stationary $n$-varifold $V_\infty$ such that $V_{j'}$ converges to $V_\infty$ in the varifold (Radon measure) sense.

Let us recall Sobolev inequality on a stationary varifold $V$ in $U$.
Michael-Simon \cite{m-s} proved the following Sobolev inequality on $V$ (actually, for general submanifolds of mean curvature type).
There is a constant $c_n>0$ depending only on $n$ such that
\begin{equation}\aligned\label{SobM0}
\left(\int|f|^{\f n{n-1}}d\mu_V\right)^{\f{n-1}n}\le c_n\int|\n_V f|d\mu_V
\endaligned
\end{equation}
holds for each Lipschitz function $f$ with compact support in $U$. Recently, Brendle \cite{br} obtained the sharp Sobolev constant for minimal submanifolds of codimensions $\le2$ (see \cite{l-w-w} for the relative version).

Let $\mathcal{D}^n(U)$ denote the set including all smooth $n$-forms on the open $U\subset\R^{n+m}$ with compact supports in $U$. Denote $\mathcal{D}_n(U)$ be the set of $n$-currents in $U$, which are continuous linear functionals
on $\mathcal{D}^n(U)$. For each $T\in \mathcal{D}_n(U)$ and each open set $W$ in $U$, one defines the mass of $T$ on $W$ by
\begin{equation*}\aligned
\mathbb{M}(T\llcorner W)=\sup_{|\omega|_U\le1,\omega\in\mathcal{D}^n(U),\mathrm{spt}\omega\subset W}T(\omega)
\endaligned
\end{equation*}
with $|\omega|_U=\sup_{x\in U}\lan\omega(x),\omega(x)\ran^{1/2}$.
Let $\p T$ be the boundary of $T$ defined by $\p T(\omega')=T(d\omega')$ for any $\omega'\in\mathcal{D}^{n-1}(U)$.
For $T\in\mathcal{D}_n(U)$ , $T$ is said to be an \emph{integer multiplicity current} if it can be expressed as
$$T(\omega)=\int_S\th\lan \omega,\xi\ran,\qquad \mathrm{for\ each}\ \omega\in \mathcal{D}^n(U),$$
where $S$ is a countably $n$-rectifiable subset of $U$, $\th$ is a locally $\mathcal{H}^n$-integrable positive integer-valued function, and $\xi$ is an orientation on $S$, i.e.,  $\xi(x)$ is an $n$-vector representing the approximate tangent space $T_xS$ for $\mathcal{H}^n$-a.e. $x$.
Let $f:\ U\rightarrow \R^p$ be a $C^1$-mapping with $p\ge n$, and $f_*\xi$ denote the push-forward of $\xi$, which is an orientation of $f(S)$ in $\R^p$. We define $f(T)\in \mathcal{D}_n(\R^p)$ by letting
\begin{equation}\label{PushforwfT}
f(T)(\omega)=\int_S\th\lan \omega\circ f,f_*\xi\ran=\int_{y\in f(S)}\left\lan\omega(y),\sum_{x\in f^{-1}(y)\cap S_*}\th(x)\f{f_*\xi(x)}{|f_*\xi(x)|}\right\ran
\end{equation}
for each $\omega\in \mathcal{D}^n(\R^p)$,
where $S_*=\{x\in S|\, |f_*\xi(x)|>0\}$. It's clear that $f(T)$ is an integer multiplicity current in $\R^p$.

Let $|T|$ denote the varifold associated with $T$, i.e., $|T|=\th|S|$.
If both $T$ and $\p T$ are integer multiplicity rectifiable currents, then $T$ is called an \emph{integral current}.
Federer and Fleming \cite{ff}(see also 27.3 Theorem in \cite{s}) proved a compactness theorem (or referred to as a closure theorem):
a sequence of integral currents $T_j\in\mathcal{D}_n(U)$ with $\mathbb{M}(T_j)$ and $\mathbb{M}(\p T_j)$ uniformly
bounded admits a subsequence that converges weakly to an integral current.
For an integer $k\ge1$,
we recall that an integral current $T\in\mathcal{D}_k(\R^{n+m})$ is \emph{decomposable} in $U$ (see Bombieri-Giusti \cite{b-g}) if there exist integral currents $T_1,T_2\in\mathcal{D}_k(U)$ with $T_1\llcorner U,T_2\llcorner U\neq0$ such that
\begin{equation}\aligned\nonumber
\mathbb{M}(T\llcorner W)=\mathbb{M}(T_1\llcorner W)+\mathbb{M}(T_2\llcorner W),\qquad \mathbb{M}(\p T\llcorner W)=\mathbb{M}(\p T_1\llcorner W)+\mathbb{M}(\p T_2\llcorner W)
\endaligned
\end{equation}
for any $W\subset\subset U$. Here, $T_1,T_2$ are called \emph{components} of $T\llcorner U$. On the contrary, $T$ is said to be \emph{indecomposable} in $U$.
If $T$ is decomposable (indecomposable) in any open $W\subset\subset\R^{n+m}$, then we say $T$ \emph{decomposable(indecomposable)} for simplicity.

All $n$-dimensional minimal graphs in $\R^{n+1}$ are
area-minimizing, which is not true for the higher codimension case in general.
Bombieri-Giusti \cite{b-g} proved that every codimension one area-minimizing current in Euclidean space is indecomposable, and the following uniform Neumann-Poincar$\mathrm{\acute{e}}$ inequality holds on any area-minimizing hypersurface $\Si$ in $\R^{n+1}$.
There is a constant $c_n>0$ depending only on $n$ such that for any $x\in \Si$, $r>0$, any Lipschitz function $f$ on $B_r(x)$
\begin{equation}\aligned\label{PoincareM0}
\min_{k\in\R}\left(\int_{B_r(x)\cap \Si}|f-k|^\f{n}{n-1}\right)^{\f{n-1}n}\le c_n\int_{B_{c_nr}(x)\cap \Si}|\n_\Si f|.
\endaligned
\end{equation}
The uniform Neumann-Poincar$\mathrm{\acute{e}}$ inequality plays a significant
role in the mean value inequality for superharmonic functions.
As applications, they got several impressive results for minimal graphs of codimension 1.
Note that \eqref{PoincareM0} does not hold for all  minimal hypersurfaces; for
instance, the catenoid is a counterexample.

\subsection{Lipschitz graphs of high codimensions}

Let $f:\, M_1\rightarrow M_2$ be a locally Lipschitz map between Riemannian manifolds $M_1$ and $M_2$.
For each point $p\in M_1$ and each integer $k>0$, let $\La^kdf|_p:\,\La^kT_pM_1\rightarrow\La^kT_{f(p)}M_2$ be the $k$-Jacobian map induced by the differential $df|_p:\,T_pM_1\rightarrow T_{f(p)}M_2$ at differentiable points of $f$. Let $\k_1,\cdots,\k_n$ denote the singular values of the Jabobi matrix $df$ at any considered differentiable point of $f$.
We let $|\La^2df|$ be the 2-dilation of $f$ defined by
\begin{equation}\label{DEF2-dil}
|\La^2df|=\sup_{i\neq j}|\k_i\k_j|,
\end{equation}
while $|\La^1df|$ is the 1-dilation, i.e. the Lipschitz norm $\mathrm{Lip}\,f=|\La^1df|=\sup_{i}|\k_i|.$
$f$ is said to have 2-dilation bounded by $\La$ for some constant $\La\ge0$ if $|\La^2df|\le\La$ a.e. on $M_1$.
Let $\mathbf{Lip}\, f=\sup_{x\in M_1}\mathrm{Lip}\, f(x)$ denote the Lipschitz constant of $f$ on $M_1$.

In the case $f$ being a locally Lipschitz map from an open $\Om\subset\Bbb{R}^n$ into $\Bbb{R}^m$, its graph defines a locally Lipschitz submanifold $M$ in $\Bbb{R}^{n+m}.$
Let $\Om_*$ be the largest subset of $\Om$ such that $f$ is $C^1$ on $\Om_*$, and $\mathrm{reg} M=\{(x,f(x))|\, x\in\Om_*\}$. $\mathrm{reg} M$ is called the \emph{regular} part of $M$.
Now we have the usual Gauss map from $\mathrm{reg} M$.
Let $\{\mathbf{E}_1,\cdots, \mathbf{E}_{n+m}\}$ be the standard orthonormal basis of $\Bbb R^{m+n}$. 
At each point in $\mathrm{reg} M$ its image $n$-plane $P$ under the Gauss map is spanned by $\tilde{f}_1,\cdots,\tilde{f}_n$ with
$$\tilde{f}_i=\mathbf{E}_i+\sum_\a f^\a_i\mathbf{E}_{n+\a},$$
where $f^\a_i=\frac{\p {f^\a}}{\p {x^i}}$.
Let $\th_1,\cdots,\th_n$ be the Jordan angles between $\mathbf{E}_1\wedge\cdots\wedge \mathbf{E}_n$ and $\tilde{f}_1\w\cdots\w \tilde{f}_n$.
Let $\xi$ denote an \emph{orientation} of $M$ defined by
\begin{equation}\label{orientation}
\xi=\f1{|\tilde{f}_1\w\cdots\w \tilde{f}_n|}\tilde{f}_1\w\cdots\w \tilde{f}_n
\end{equation}
on $\mathrm{reg} M$ with $|\tilde{f}_1\w\cdots\w \tilde{f}_n|^2=\det(\de_{ij}+\sum_\a f^\a_if^\a_j)$.
In particular, $\xi$ is continuous on $\mathrm{reg} M$, and $\xi(x)$ represents the tangent space $T_xM$ as a unit $n$-vector for each $x\in\mathrm{reg} M$.
We denote $[|M|]\in\mathcal{D}_n(\Om\times\R^m)$ as the $n$-current associated with $M$ and its orientation $\xi$.
Let $\la_i=\tan\th_i$ as \eqref{di2}, then
$\la_1,\cdots,\la_n$ are the singular values of $df$ at each point $x\in \Om$. Namely,
$\la_i^2$ are eigenvalues of the matrix $\left(\sum_\a\frac{\p f^\a}{\p x^i}\frac {\p f^\a}{\p x^j}\right).$
Hence, the Gauss image of any point in $\mathrm{reg} M$ is spanned by orthonormal vectors
$$\f{1}{\sqrt{1+\la_i^2}}(\mathbf{E}_i+\la_i \mathbf{E}_{n+i}).$$

Suppose that the $n$-plane $P_0$ in \eqref{BbbW0PP0} is spanned by $\mathbf{E}_1,\cdots,\mathbf{E}_n$.
Recalling \eqref{BbbT2La}, we have a conclusion:
\begin{pro}
For a locally Lipschitz graph $M=\mathrm{graph}_f$ in $\Bbb{R}^{n+m}$, the image under the Gauss map from $\mathrm{reg} M$ lies in $2-$bounded subset
$\Bbb{T}^{2,\La}\subset \Bbb{W}_0 \subset\grs{n}{m}$ if and only if the defining map $f$ has bounded $2-$dilation by $\La$.
\end{pro}

Let $u=(u^1,\cdots,u^m)$ be a locally Lipschitz (vector-valued) function on an open $\Om\subset\R^n$.
Let $g_{ij}=\de_{ij}+\sum_{\a=1}^m\p_{x_i}u^\a\p_{x_j}u^\a$, and $(g^{ij})$ be the inverse matrix of $(g_{ij})$ for almost every point in $\Om$.
Let $M$ be the graph of the function $u$, which is countably $n$-rectifiable.
We can define the {\it slope} function of $M$ by
\begin{equation}\aligned\label{v}
v=\sqrt{\det g_{ij}}=\sqrt{\det\left(\de_{ij}+\sum_{\a=1}^m \f{\p u^\a}{\p x_i}\f{\p u^\a}{\p x_j}\right)}
\endaligned
\end{equation}
$\mathcal{H}^n$-a.e. on $\Om$.
Note that we also see $u,v$ as the functions on $M$ by letting $u(x,u(x))=u(x)$ and $v(x,u(x))=v(x)$ for every $x\in\Om$.
If the varifold associated with $M$ is stationary, then all the coordinate functions are weakly harmonic on $M$ (see \cite{c-m1} for instance), i.e., all the $x_1,\cdots,x_n$ and $u^1,\cdots, u^\a$ are weakly harmonic on $M$.
Namely, for any Lipschitz function $\phi$ on $\Om$ with compact support on $\Om$, there holds
\begin{equation}\aligned\label{Weaku}
\sum_{j=1}^n\int_{\Om}\ vg^{ij}\f{\p \phi}{\p x_j}=0\qquad \mathrm{for\ each}\ i=1,\cdots,n,
\endaligned
\end{equation}
and
\begin{equation}\aligned\label{Weaku}
\sum_{i,j=1}^n\int_{\Om}\ vg^{ij}\f{\p u^\a}{\p x_i}\f{\p \phi}{\p x_j}=0\qquad \mathrm{for\ each}\ \a=1,\cdots,m.
\endaligned
\end{equation}
We call $M$ a \emph{minimal graph} if and only if the varifold associated with $M$ is
stationary. Unlike minimal graphs of codimension one, minimal graphs of higher
codimensions may be only Lipschitz as in the  examples by Lawson-Osserman \cite{l-o}.
For simplicity, we say that the graph $M$ has 2-dilation bounded by $\La$, if its graphic function $u:\Om\to\R^m$ has 2-dilation bounded almost everywhere by $\La$.
Namely, $|\La^2du|\le\La$ $\mathcal{H}^n$-a.e. on $\Om$.

From the interior regularity theorem of Morrey, $u$ is smooth at the differentiable points.
Then from \eqref{Weaku}, $u$ satisfies the following minimal surface system
\begin{equation}\aligned\label{Nms*}
\De_M u^\a=\f1{v}\sum_{i,j=1}^n\f{\p}{\p x_i}\left(v g^{ij}\f{\p u^\a}{\p x_j}\right)=0
\endaligned
\end{equation}
at differentiable points of $u$.
By \cite{o} (or \cite{l-o}), \eqref{Nms*} is equivalent to
\begin{equation}\aligned\label{Nms}
\sum_{i,j=1}^ng^{ij}\f{\p^2 u^\a}{\p x_i\p x_j}=0
\endaligned
\end{equation}
at differentiable points of $u$.

\Section{Volume estimates for minimal graphs}{Volume estimates for
minimal graphs}

\medskip

Let $M$ be a locally Lipschitz minimal graph of the graphic function $u=(u^1,\cdots,u^m)$ over $B_R$ of codimension $m\ge1$ in $\R^{n+m}$. Then the induced metric of the graph
is $g_{ij}=\de_{ij}+ \pd{u^\a}{x^i}\pd{u^\a}{x^j}$, and we denote
$v=\left(\det(g_{ij})\right)^{\f{1}{2}}.$
Let $D$ denote the derivative on $\R^n$, and $\n$ denote the Levi-Civita connection of the regular part of $M$.
\begin{lem}\label{VGM}
Suppose that $u$ has  2-dilation bounded by $\La$.
Then we have a volume estimate:
\begin{equation}\label{VGM1}
\H^n(M\cap \mathbf{B}_r(\mathbf{x}))=c_{n,\La}\sqrt{m}\om_nr^n
\end{equation}
for any ball $\mathbf{B}_{2r}(\mathbf{x})$ in $B_R\times\R^{m}$, where $c_{n,\La}$ is a constant $\ge1$ depending only on $n,\La$.
\end{lem}
\begin{proof}
For proving \eqref{VGM1},
we only need to consider the ball $\mathbf{B}_r$ with $\mathbf{B}_{2r}\subset B_R\times\R^{m}$. Without loss of generality, we can assume $u^\a(0)=0$ for each $\a$.
Let $\mathfrak{u}^\a_r$ be a function on $B_{2r}$ defined by
\begin{eqnarray*}
   \mathfrak{u}^\a_r= \left\{\begin{array}{ccc}
           r  & \quad\quad  {\rm{if}} \ \ \     u^\a\geq r \\ [3mm]
            u^\a     & \quad\ \ \ \ {\rm{if}} \ \ \  |u^\a|< r \\ [3mm]
           -r  & \quad\quad\ \ {\rm{if}} \ \ \   u^\a\leq-r
     \end{array}\right..
\end{eqnarray*}
For any $\de\in(0,r]$, we define a non-negative Lipschitz function $\e$ on
$\ir{n}$ given by
$$\e=\begin{cases} 1,\hskip0.8in \text{on}\quad B_r\\
\f{(1+\de)r-|x|}{\de r},\quad \text{on}\quad
B_{(1+\de)r}\backslash B_r\\
0,\hskip0.8in \text{on}\quad \ir{n}\backslash B_{(1+\de)r}
\end{cases}.$$
From \eqref{Weaku}, for each $\a$ we have
\begin{equation}\aligned\label{IBBvua2}
0=\int_{\R^n} vg^{ij}\f{\p u^\a}{\p x_i}\f{\p(\e \mathfrak{u}^\a_r)}{\p x_j}=\int_{B_{(1+\de)r}\backslash B_r}\lan \n u^\a,\n\e\ran \mathfrak{u}^\a_rv+\int_{\{|u^\a|<r\}}v|\n u^\a|^2\e.
\endaligned
\end{equation}
We can check
\begin{equation}\label{vvua2}
v\le \sum_\a|\n u^\a|^2v+1,
\end{equation}
which is obvious when $m=1$.
Combining \eqref{IBBvua2} and \eqref{vvua2}, one has
\begin{equation}\aligned\label{EVGM0}
&\mathcal{H}^n(M\cap \mathbf{B}_r)\le\int_{B_r\cap\{|u|<r\}}v\le\int_{B_r\cap\{|u|<r\}}\left(1+v\sum_\a|\n u^\a|^2\right)\\
\le&\omega_nr^n+\sum_\a\int_{\{|u^\a|<r\}}v|\n u^\a|^2\e=\omega_nr^n-\sum_\a\int_{B_{(1+\de)r}\backslash B_r}\lan \n u^\a,\n\e\ran \mathfrak{u}^\a_rv.
\endaligned
\end{equation}
For each considered point in $B_{(1+\de)r}\backslash B_r$, we can assume
$g_{ij}=(\sec^2\th_i)\de_{ij},$
where $\la_i^2=\tan^2\th_i$ are eigenvalues of $\left(\sum_\a\f{\p{f^\a}}{\p{x^i}}\f{\p{f^\a}}{\p{x^j}}\right)$. Then $v=\left(\det(g_{ij})\right)^{\f{1}{2}}=\prod\sec\th_i.$
Moreover, there is an orthonormal matrix $(a_{\a\be})$ so that $\pd{u^\a}{x^j}=a_{\a j}\la_j=a_{\a j}\tan\th_j$ (We let $a_{\a j}=0$ for $j\ge m+1$).
Hence, for each $\a$
\begin{equation}\aligned\label{nuanue}
|\left<\n u^\a,\n \e\right>\mathfrak{u}^\a_r|v\le& r\left|g^{ij}\pd{\e}{x^i}\pd{u^\a}{x^j}\right|v=r\sum_i\left|\cos^2\th_i\pd{\e}{x^i} a_{\a i}\tan\th_i\right|v\\
=&r\sum_i\left|a_{\a i}\pd{\e}{x^i}\sin\th_i\cos\th_i \right|\prod_j\sec\th_j.
\endaligned
\end{equation}

Under the condition of the  $2-$dilation bounded by $\La$, namely, $\tan\th_i\tan\th_j\le \La$ for $i\ne j$
\begin{equation*}\aligned
&\sum_i\sin^2\th_i\cos^2\th_i\prod_j\sec^2\th_j=\sum_i\sin^2\th_i\prod_{j\neq i}(1+\tan^2\th_j)\\
\le&\sum_i\cos^2\th_i\tan^2\th_i\sum_{j\neq i}(c'+\tan^2\th_j)\le c^2,
\endaligned
\end{equation*}
where $c$ and $c'$ depend on $\La$ and $n$. Using Cauchy inequality, from \eqref{nuanue} we get
\begin{equation}\aligned\label{nuanue*}
&\sum_\a|\left<\n u^\a,\n \e\right>\mathfrak{u}^\a_r|v\le r\sqrt{\sum_{\a,i}\left|a_{\a i}\pd{\e}{x^i}\right|^2}\sqrt{\sum_{\a,i}\sin^2\th_i\cos^2\th_i\prod_j\sec^2\th_j}\\
\le& c\sqrt{m}r\sqrt{\sum_{\a,i}a_{\a i}^2\left|\pd{\e}{x^i}\right|^2}= c\sqrt{m}r\sqrt{\sum_{i}\left|\pd{\e}{x^i}\right|^2}\le \f{c\sqrt{m}}{\de}.
\endaligned
\end{equation}
Substituting \eqref{nuanue*} into \eqref{EVGM0} gives
\begin{equation*}
\mathcal{H}^n(M\cap \mathbf{B}_r)\le\omega_nr^n+\f{c\sqrt{m}}{\de}\int_{B_{(1+\de)r}\backslash B_r}dx\le\om_nr^n+\f{c\sqrt{m}}{\de}((1+\de)^n-1)\om_nr^n.
\end{equation*}
Thus,
\begin{equation*}\H^n(M\cap \mathbf{B}_r)\le\om_nr^n+\lim_{\de\to 0}\f{c\sqrt{m}}{\de}((1+\de)^n-1)\om_nr^n=(cn\sqrt{m}+1)\om_nr^n.\end{equation*}
This completes the proof.
\end{proof}
If $M$ is an entire minimal graph, the estimation in \eqref{VGM1} of Lemma \ref{VGM} can be independent of the codimension $m$.
\begin{lem}\label{VGM*}
Let $M=\mathrm{graph}_u$ be a locally Lipschitz minimal graph over $\R^n$ of codimension $m\ge2$ with $\sup_{\R^n}|\La^2du|\le\La$.
Then there is a constant $C_{n,\La}\ge1$ depending only on $n,\La$ such that $M$ is contained in some affine subspace of dimension $\le C_{n,\La}$
and for any ball $\mathbf{B}_r(\mathbf{x})$ in $\R^{n+m}$
\begin{equation}\aligned\label{VGMentireG}
\mathcal{H}^n(M\cap \mathbf{B}_r(\mathbf{x}))\le C_{n,\La}\omega_nr^n.
\endaligned
\end{equation}
\end{lem}
\begin{proof}
From Colding-Minicozzi (Corollary 1.4 in \cite{c-m}), $M$ must be contained in
some affine subspace $\mathbb{V}$ of the dimension
\begin{equation}\aligned\label{pcnVM}
\mathfrak{p}\le c_nV_M,
\endaligned
\end{equation}
where
$c_n=\f{3n}{n-1}2^{n+3}e^8$ and 
\begin{equation}\aligned
V_M=\lim_{r\rightarrow\infty}\f1{\omega_nr^n}\mathcal{H}^n(M\cap \mathbf{B}_r).
\endaligned
\end{equation}
In particular, for $\mathfrak{p}=n$, $M$ is flat.
Let $\widetilde{\mathbb{V}}$ be the linear space spanned by vectors in $\mathbb{V}$ and vectors in the $n$-plane $\{(x,0^m)\in\R^n\times\R^m|\, x\in\R^n\}$.
Then $\widetilde{\mathbb{V}}$ has dimension $\tilde{\mathfrak{p}}\le\mathfrak{p}+n$.
Hence up to an isometric transformation of $\R^m$, $M$ can be written as a graph over $\R^n$ in $\R^{\tilde{\mathfrak{p}}}$ with the graphic function $w$ satisfying $\sup_{\R^n}|\La^2dw|\le\La$.
From Lemma \ref{VGM} and \eqref{pcnVM}, we have
\begin{equation}\aligned
\mathfrak{p}\le c_nc_{n,\La}\sqrt{\tilde{\mathfrak{p}}-n}\le c_nc_{n,\La}\sqrt{\mathfrak{p}},
\endaligned
\end{equation}
which implies $\mathfrak{p}\le c_n^2c_{n,\La}^2$.
Therefore, from Lemma \ref{VGM} again, we get
\begin{equation}\aligned
V_M=\lim_{r\rightarrow\infty}\f1{\omega_nr^n}\mathcal{H}^n(M\cap \mathbf{B}_r)\le c_{n,\La}\min\{\sqrt{m},\sqrt{\mathfrak{p}}\}\le c_{n,\La}\min\{\sqrt{m},c_nc_{n,\La}\}.
\endaligned
\end{equation}
We complete the proof by the monotonicity of $r^{-n}\mathcal{H}^n(M\cap \mathbf{B}_r(\mathbf{x}))$ on $r$.
\end{proof}

The constant $c_{n,\La}$ in \eqref{VGM1} or $C_{n,\La}$ in \eqref{VGMentireG} depends on $\La$  in general. Even, without the bounded 2-dilation, minimal graphs have much faster volume growth in view of the following remark.
\begin{rem}\label{LEVG}
For a smooth harmonic function $\phi$ on $\R^2$, the graph
$M=\{(x,y,D\phi)\in\R^4|\ (x,y)\in\R^2\}$ is a special Lagrangian submanifold
in $\R^4$, and in particular,  minimal.
Let $\phi=\mathrm{Re}\, e^z=e^x\cos y$.
Then $D\phi=(\mathrm{Re}\, e^z,-\mathrm{Im}\, e^z)$.
By a straightforward computation,
\begin{equation}\label{D2phi2}
\aligned
D^2\phi D^2\phi=
\left( \begin{array}{cc}
e^x\cos y & -e^x\sin y  \\
-e^x\sin y & -e^x\cos y  \\
\end{array} \right)^2=
\left( \begin{array}{cc}
e^{2x} & 0  \\
0 & e^{2x}  \\
\end{array} \right).
\endaligned\end{equation}
For any point $(x,y,D\phi)\in \mathbf{B}_{\sqrt{3}r}^4\subset\R^4$ with $r\ge e$, we have
\begin{equation}\aligned
3r^2>x^2+y^2+|D\phi|^2=x^2+y^2+e^{2x},
\endaligned
\end{equation}
which implies
\begin{equation}\aligned
\{(x,y,D\phi)\in\R^4|\, 0<x<\log r,\, |y|<r\}\subset M\cap\mathbf{B}^{4}_{\sqrt{3}r}.
\endaligned
\end{equation}
Then with \eqref{D2phi2}
\begin{equation}\aligned
\mathcal{H}^2(M\cap\mathbf{B}^{4}_{\sqrt{3}r})\ge&\int_{0}^{\log r}\left(\int_{-r}^r\sqrt{\det(I+D^2\phi D^2\phi)}dy\right)dx\\
\ge&2r\int_0^{\log r}e^{2x}dx=r(r^2-1).
\endaligned
\end{equation}
In other words, $M$ has  volume growth strictly larger than Euclidean.

\end{rem}

Let $\Om$ be an open set in $B_R\subset\R^n$. For each $\a=1,\cdots,m$, let $M^\a$ be the graph in $\Om\times\R$ defined by
\begin{equation}\aligned
M^\a=\{(x,u^\a(x))\in\R^n\times\R|\ x\in \Om\}.
\endaligned
\end{equation}
By a diagonal argument, it is clear that
\begin{equation}\aligned\label{digDnau}
|Du^\a|^2\le |\n u^\a|^2v^2
\endaligned
\end{equation}
for each $\a$.
At any $C^1$-point of $u$, the unit normal vector of $M^\a$ can be written as
$$\nu_{M^\a}=\f1{\sqrt{1+|Du^\a|^2}}\left(-\sum_{j=1}^n\f{\p u^\a}{\p x_j}E_j+E_{n+1}\right),$$
where $E_j$ is a unit constant vector in $\R^{n+1}$ with respect to the axis $x_j$.
For each fixed $\a$, let $W$ denote a bounded open set with $n$-rectifiable $\p W$ in $\R^{n+1}$ such that
there is a constant $\g>0$ satisfying
\begin{equation}\aligned\label{Wg}
\mathcal{H}^n(\p W\cap B_r(x))\ge\g r^n\qquad \mathrm{for\ any}\ x\in \p W,\ \mathrm{any} \ r\in(0,1).
\endaligned
\end{equation}

Let $Y$ be a measurable vector field in $\Om$ defined by
\begin{equation}\aligned\label{Y}
Y=v\sum_{i,j=1}^n g^{ij}\f{\p u^\a}{\p x_i}E_j.
\endaligned
\end{equation}
By parallel transport, we obtain a Lipschitz vector field in $\Om\times\R$, still denoted by $Y$.
Let $\bn$ denote the Levi-Civita connection of $\R^{n+1}$. Then from
\eqref{Weaku} and the co-area formula,
\begin{equation}\aligned\label{YbarDphi}
0=\int_\R\left(\int_{W\cap\{x_{n+1}=t\}} \lan Y,D \phi\ran \right)dt=\int_\R\left(\int_{W\cap\{x_{n+1}=t\}} \lan Y,\bn \phi\ran \right)dt=\int_W \lan Y,\bn \phi\ran
\endaligned
\end{equation}
for any smooth function $\phi$ with compact support in $W$. For any $\ep>0$, let $W_\ep=\{x\in W|\, d(x,\p W)>\ep\}$, $\phi_\ep=1-\f1\ep d(\cdot,W_\ep)$ on $W$ and $\phi_\ep=0$ on $\R^{n+1}\setminus W$.
Let $\sigma\in C^\infty_c(\R^{n+1})$ satisfy $\int_{\R^{n+1}}\si(z)dz=1$, and $\si_\de(z)=\de^{-n}\si(z/\de)$ for each $z\in\R^{n+1}$ and each $\de\in(0,1]$.
Let $\phi_{\ep,\de}$ be a convolution of $\phi_{\ep}$ and $\si_\de$ defined by
\begin{equation}\aligned\label{wep}
\phi_{\ep,\de}(z)=(\phi_{\ep}*\si_\de)(z)=\int_{\R^{n+1}} \phi_{\ep}(y)\si_\de(z-y)dy=\int_{\R^{n+1}} \phi_{\ep}(z-y)\si_\de(y)dy.
\endaligned
\end{equation}
Then $\phi_{\ep,\de}\in C^\infty_c(\R^{n+1})$, and $\bn\phi_{\ep,\de}\to\bn\phi_{\ep}$ at differentiable points of $\phi_{\ep}$ as $\de\to0$.
Substituting $\phi_{\ep,\de}$ into \eqref{YbarDphi} implies
\begin{equation}\aligned\label{TDphiep**}
0=\lim_{\de\rightarrow0}\int_W \lan Y,\bn \phi_{\ep,\de}\ran=\int_W \lan Y,\bn \phi_\ep\ran.
\endaligned
\end{equation}
Let $\nu_{\p W}$ be the outward unit normal vector to the regular part of $\p W$.
Since $\lan Y,E_{n+1}\ran=0$ a.e. on $\Om\times\R$, we may denote $\lan Y,E_{n+1}\ran=0$ on $\p W$.
By the definition of $Y$, $\lan Y,\nu_{\p W}\ran$ is well-defined $\mathcal{H}^n$-a.e. on $\p W$.
From \eqref{Wg} and Ambrosio-Fusco-Pallara (in \cite{a-f-p}, p. 110), letting $\ep\to0$ in \eqref{TDphiep**} implies
\begin{equation}\aligned
0=\lim_{\ep\rightarrow0}\int_W \lan Y,\bn \phi_\ep\ran=\int_{\p W}\lan Y,\nu_{\p W}\ran.
\endaligned
\end{equation}
Recalling \eqref{digDnau}, we have
\begin{equation}\aligned\label{CompS}
&\int_{\p W\cap M^\a}\f{|D u^\a|^2}{\sqrt{1+|Du^\a|^2}v}\le\int_{\p W\cap M^\a}\f{|\n u^\a|^2v}{\sqrt{1+|Du^\a|^2}}\\
=&-\int_{\p W\cap M^\a}\lan Y,\nu_{M^\a}\ran\le\int_{\p W\setminus M^\a}\left|\lan Y,\nu_{\p W}\ran\right|.
\endaligned
\end{equation}

\Section{Cylindrical minimal cones from minimal graphs}{Cylindrical minimal cones from minimal graphs}

Let $M_k$ be a sequence of $n$-dimensional locally Lipschitz minimal graphs over $B_{R_k}$ in $\R^{n+m}$ of codimension $m\ge1$
with $R_k\rightarrow\infty$,
and the graphic functions $u_k=(u^1_k,\cdots,u^m_k)$ of $M_k$ satisfies $u_k(0^n)=0^m$ and $|\La^2du_k|\le\La$ a.e. for some constant $\La>0$.
We may suppose that $|M_k|$ converges in the varifold sense to a minimal cone $C$ in $\R^{n+m}$ with $0^{n+m}\in C$. From Lemma \ref{VGM}, the multiplicity of $C$ has a upper bound depending only on $n,m,\La$.

Let $\{\mathbf{E}_i\}_{i=1}^{n+m}$ be a standard basis of $\R^{n+m}$ such that each $\mathbf{x}\in M_k$ can be represented as $\sum_{i}x_i\mathbf{E}_i+\sum_\a u^\a_k(x)\mathbf{E}_{n+\a}$ with $x=(x_1,\cdots,x_n)\in B_{R_k}$.
If $\mathrm{spt}C\cap\{(0^n,y)\in\R^n\times\R^m|\, y\in\R^m\}=\{0^{n+m}\}$, i.e., $u_k$ has uniformly linear growth. Then the minimal cone $C$ has multiplicity one from Lemma \ref{Multi1}, which completes the proof of Theorem \ref{main2}.
Now we assume that there is a point $y_*=(0^n,y^*)\in \mathrm{spt}C$ with $0^m\neq y^*\in\R^{m}$.
Without loss of generality, we assume $y^*=(1,0,\cdots,0)\in\R^m$, then $y_*=\mathbf{E}_{n+1}$. Let
$$C_t=C-ty_*=C-t\mathbf{E}_{n+1}$$
for any $t\in\R$.
In other words, if we denote $C=\th_C|C|$, then $C_t=\th_C(\cdot+t\mathbf{E}_{n+1})|C_t|$ with spt$C_t=\{\mathbf{x}-t\mathbf{E}_{n+1}|\, \mathbf{x}\in\mathrm{spt}C\}$.
Since $C$ is a cone, then $C_t$ converges as $t\rightarrow\infty$ in the varifold sense to a minimal cone $C_{y_*}$ in $\R^{n+m}$,
where spt$C_{y_*}$ splits off a line $\{ty_*|\, t\in\R\}$ isometrically.
\begin{lem}\label{Cy*vmc}
$C_{y_*}$ is a cylindrical minimal cone living in an $(n+1)$-dimensional subspace $\{(x_1,\cdots,x_{n+1},0,\cdots0)\in\R^{n+m}|\,(x_1,\cdots,x_{n+1})\in\R^{n+1}\}$ of $\R^{n+m}$.
\end{lem}
\begin{proof}
For any regular point $\mathbf{x}\in\mathrm{spt}C_{y_*}$, the tangent cone of $C_{y_*}$ at $x$ is an $n$-plane with constant integer multiplicity by constancy theorem (see Theorem 41.1 in \cite{s}).
Let $T_{\mathbf{x}}C_{y_*}$ denote the tangent space of $C_{y_*}$ at $\mathbf{x}$, which is an $n$-plane with constant multiplicity.
We represent the support of $T_{\mathbf{x}}C_{y_*}$ by an $n$-vector $\tau_1\w\cdots\w \tau_n$ with orthonormal unit vectors
$\tau_i$. Since $\mathrm{spt}C_{y_*}$ splits off a line $ty_*$ isometrically, then
\begin{equation}\aligned\label{tau1nEn+1}
\tau_1\w\cdots\w \tau_n\wedge \mathbf{E}_{n+1}=0.
\endaligned
\end{equation}
From the construction of $C_{y_*}$, there is a sequence of minimal graphs $\Si_k$ in $\R^n\times\R^m$ such that $|\Si_k|$ converges to $C_{y_*}$ in the varifold sense.
Here, $\Si_k$ is a rigid motion of $M_k$.
Let $\xi_k$ be an orientation of $\Si_k$ for each $k$ (see \eqref{orientation}).
Since $|\Si_k|$ converges in the varifold sense to $C_{y_*}$, then there is a sequence $r_k\rightarrow\infty$ such that
$\left|\f1{r_k}\Si_k\cap B_{r_k^2}(y_*)\right|$ converges to $T_{\mathbf{x}}C_{y_*}$ in the varifold sense.
Hence, up to a choice of the subsequence, there is a sequence of regular points $\mathbf{x}_k=(x^k_1,\cdots,x^k_n,u_k(x^k_1,\cdots,x^k_n))\in M_k$ with $\mathbf{x}_k\rightarrow \mathbf{x}$ such that $\xi_k(\mathbf{x}_k)\rightarrow\tau_1\w\cdots\w \tau_n$.

Let $\la_{1,k},\cdots,\la_{n,k}$ be the singular values of the matrix $Du_k$ at $(x^k_1,\cdots,x^k_n)$ with $\la_{j,k}\ge0$ for all $j=1,\cdots,n$, and
$$e_{j,k}=\f1{\sqrt{1+\la^2_{j,k}}}\left(\mathbf{E}_{j}+\la_{j,k}\mathbf{E}_{n+j}\right)$$
for each integers $1\le j\le n$ and $k\ge1$.
Then $\{e_{j,k}\}_{j=1}^n$ forms an orthonormal basis of $T_{\mathbf{x}_k}M_k$ (up to a permutation of $\la_{1,k},\cdots,\la_{n,k}$ and a rotation of $\R^n$), and we can choose
$\xi_k(\mathbf{x}_k)=e_{1,k}\wedge\cdots\wedge e_{n,k}$.
From \eqref{tau1nEn+1}, the assumption $|\La^2du_k|\le\La$ a.e. and $\xi_k(\mathbf{x}_k)\rightarrow\tau_1\w\cdots\w \tau_n$, we get
$\la_{1,k}\rightarrow\infty$, $\sum_{j=2}^n\la_{j,k}\rightarrow0$ as $k\rightarrow\infty$, and then $\tau_1\w\cdots\w \tau_n=\mathbf{E}_{2}\wedge\cdots\wedge \mathbf{E}_{n+1}$ represents the orientation of $T_{\mathbf{x}}C_{y_*}$. Note that spt$C_{y_*}$ splits off the line $\{t\mathbf{E}_{n+1}|\, t\in\R\}$ isometrically.
So there is an $(n-1)$-dimensional cone $C_*$ in $\R^n$ such that $\mathrm{spt}C_{y_*}$ can be written as
$$\{(x_1,\cdots,x_n,y_1,\cdots,y_m)\in\R^n\times\R^m|\ (x_1,\cdots,x_n)\in C_*,y_{2}=\cdots=y_m=0\}.$$
This completes the proof.
\end{proof}

Let
$$g_{ij}^k=\de_{ij}+\sum_{\a=1}^m \f{\p u^\a_k}{\p x_i}\f{\p u^\a_k}{\p x_j},$$
\begin{equation}\aligned\label{vi}
v_k=\sqrt{\det g_{ij}^k}=\sqrt{\det\left(\de_{ij}+\sum_{\a=1}^m \f{\p u^\a_k}{\p x_i}\f{\p u^\a_k}{\p x_j}\right)},
\endaligned
\end{equation}
and $(g^{ij}_k)_{n\times n}$ be the inverse matrix of $(g_{ij}^k)_{n\times n}$.
Let $\pi_*$ denote the projection from $\R^{n+m}$ into $\R^{n}$ by
\begin{equation}\aligned\label{pi*****}
\pi_*(x_1,\cdots,x_{n+m})=(x_1,\cdots,x_n).
\endaligned
\end{equation}
For any $\mathbf{x}=(x_1,\cdots,x_{n+m})\in\R^{n+m}$, let $\mathbf{C}_r(\mathbf{x})=B_r(\pi_*(\mathbf{x}))\times B_r(x_{n+1},\cdots,x_{n+m})$ be the cylinder in $\R^{n+m}$, and $\mathbf{C}_r=\mathbf{C}_r(0^{n+m})$.

For studying the multiplicity of the cone $C$, we only need to prove it at any regular point of spt$C$.
Hence, it suffices to prove the following lemma.
\begin{lem}\label{multi1}
If spt$C\cong \R^n$,
then $C$ has  multiplicity one on spt$C$.
\end{lem}
\begin{proof}
From the proof of Lemma \ref{Cy*vmc}, we can assume
$$\mathrm{spt}C=\{(x_1,\cdots,x_n,y_1,\cdots,y_m)\in\R^n\times\R^m|\ x_1=0,y_2=\cdots=y_m=0\},$$
or else we have finished the proof by Lemma \ref{Multi1} in the Appendix II.
For each $k$, let $\mathbf{x}_k=(x,u_k(x))\in \R^n\times\R^m$ be a vector-valued function on $B_{R_k}$, and
$\tau_{j,k}$ be a tangent vector field of $M_k$ in $\R^{n+m}$ defined by
\begin{equation}\aligned
\tau_{j,k}=\f{\p\mathbf{x}_k}{\p x_j}=\mathbf{E}_j+\sum_\a (D_ju^\a_k) \mathbf{E}_{n+\a}\qquad a.e.,
\endaligned
\end{equation}
then $\det\left(\lan\tau_{i,k},\tau_{j,k}\ran\right)=v_k^2$ a.e. for each $k$.
Hence, the orientation of $M_k$ can be written as
$$\xi_k=\f1{v_k}\tau_{1,k}\wedge\cdots\wedge \tau_{n,k}\qquad a.e.$$
with $|\xi_k|=1$ a.e..
From Lemma 22.2 in \cite{s}, for any compact set $K$ in $\R^{n+m}$, we have
\begin{equation}\aligned\label{xiipmxi0}
\lim_{k\rightarrow\infty}\int_{K\cap M_k}\left(1-\lan\xi_k,\xi_0\ran^2\right)=0
\endaligned
\end{equation}
with $\xi_0=\mathbf{E}_{2}\wedge\cdots\wedge \mathbf{E}_{n+1}$.
Since $\xi_k$ has the expansion
\begin{equation}\aligned
\xi_k=&\f1{v_k}\mathbf{E}_{1}\wedge\cdots\wedge \mathbf{E}_{n}+\sum_{j,\a}(-1)^{n-j}\f{D_ju^\a_k}{v_k} \mathbf{E}_{1}\wedge\cdots\wedge \widehat{\mathbf{E}_j}\wedge\cdots \mathbf{E}_{n}\wedge \mathbf{E}_{n+\a}\\
&+\sum_{j_1<\cdots<j_n,j_{n-1}>n}a_{k,j_1,\cdots,j_n}\mathbf{E}_{j_1}\wedge\cdots\wedge \mathbf{E}_{j_n}\qquad a.e.
\endaligned
\end{equation}
with $|a_{k,j_1,\cdots,j_n}|\le1$,
then from \eqref{xiipmxi0} we have
\begin{equation}\aligned\label{p0}
\lim_{k\rightarrow\infty}\int_{K\cap M_k}\left(1-\f{|D_1u^1_k|^2}{v_k^2}\right)=0.
\endaligned
\end{equation}
With $|\La^2du_k|\le\La$ a.e., there are a sequence of positive numbers $\ep_k\rightarrow0$ (as $k\to\infty$) and a sequence of open sets $W_k\subset \mathbf{C}_3\cap M_k$ with $\mathcal{H}^n(W_k)<\ep_k$ such that
\begin{equation}\aligned\label{viD1u1i}
v_k\le(1+\ep_k)|D_1u^1_k|,\quad |D_1u^1_k|\ge\f1{\ep_k}\qquad \mathrm{on}\ \pi_*(\mathbf{C}_3\cap M_k\setminus W_k).
\endaligned
\end{equation}
Then for any small constant $0<\ep<1$ we have
\begin{equation}\aligned\label{AreaEst}
\mathcal{H}^n\left(M_{k}\cap\mathbf{C}_{1}\right)\le\int_{(M_{k}\setminus W_k)\cap\mathbf{C}_{1}}1+\int_{W_k}1\le(1+\ep)\int_{\pi_*((M_{k}\setminus W_k)\cap\mathbf{C}_{1})}|D_1u^1_k|+\ep
\endaligned
\end{equation}
for large $k$.

Let $\pi^*$ denote the projection from $\R^{n+m}$ into $\R^{n+1}$ by
\begin{equation}\aligned\label{pi*}
\pi^*(x_1,\cdots,x_{n+m})=(x_1,\cdots,x_{n+1}).
\endaligned
\end{equation}
For $x=(x_1,\cdots,x_{n+1})$, let $C_r(x)$ be a cylinder in $\R^{n+1}$ defined by
$$C_r(x)=B_r((x_1,\cdots,x_n))\times(x_{n+1}-r,x_{n+1}+r).$$
Let $M^*_{k}=\pi^*(M_k)$. 
For any $q\in \R^n$ with $|q|^2+|u_k(q)|^2<r^2$ and $r\in(0,3)$, we have $|q|^2+|u^1_k(q)|^2<r^2$, then it follows that
\begin{equation}\aligned\label{MiCrMi*}
\pi^*(M_k\cap\mathbf{C}_{r})\subset M_k^*\cap C_r(0^{n+1}).
\endaligned
\end{equation}
From  \eqref{ratio}, we have $\mathcal{H}^n(M_k\cap \mathbf{B}_r(\mathbf{x}))\ge\omega_n r^n$ for any $\mathbf{x}\in M_k$ and $\mathbf{B}_r(\mathbf{x})\subset B_{R_k}\times\R^m$.
Since $|M_k|$ converges to the cone $C$, then $M_k\cap K$ converges to $\mathrm{spt}C\cap K$ in the Hausdorff sense for any compact set $K$ in $\R^{n+m}$. Thus,
\begin{equation}\aligned\label{Mir+ep}
M_k^*\cap C_r(0^{n+1})\subset\pi^*(M_k\cap\mathbf{C}_{r+\ep})
\endaligned
\end{equation}
for all the sufficiently large $k$ and $r\in(0,2]$. In particular, $\p M_k^*\cap C_2(0^{n+1})=\emptyset$.
Let $W_k^*=\pi^*(W_k)\cap M^*_{k}$.
With \eqref{MiCrMi*}, we get
\begin{equation}\aligned\label{MiWi**}
\pi^*((M_k\setminus W_k)\cap\mathbf{C}_{1})\subset C_1(0^{n+1})\cap M^*_{k}\setminus W_k^*.
\endaligned
\end{equation}
From \eqref{AreaEst}\eqref{MiWi**}, it follows that
\begin{equation}\aligned\label{AreaEst1}
\mathcal{H}^n\left(M_{k}\cap\mathbf{C}_{1}\right)\le(1+\ep)\mathcal{H}^n\left(C_1(0^{n+1})\cap M^*_{k}\setminus W_k^*\right)+\ep
\endaligned
\end{equation}
for all the sufficiently large $k$.

We claim that
\begin{center}
\emph{there is a constant $k_0>0$ such that for all $k\ge k_0$ and all the $C_r(z)\subset C_2(0^{n+1})$, $C_r(z)\setminus M^*_{k}$ has only two components.}
\end{center}
Assume that $C_r(z)\setminus M^*_{i_k}$ has at least 3 components  for a sequence $i_k\rightarrow\infty$.
Without loss of generality, we assume $\p M_{i_k}^*\cap C_2(0^{n+1})=\emptyset$ for all $k\ge1$.
Put $\mathcal{I}=\{i_k\}_{k=1}^\infty$.
Since $M_k$ converges locally to $\mathrm{spt}C$ in the Hausdorff sense, then
$C_r(z)\cap M_k^*$ converges to the $n$-dimensional ball
$$B_r(z)\cap\{(x_1,\cdots,x_n,y_1)\in\R^{n+1}|\, x_1=0\}$$
in the Hausdorff sense for any $C_r(z)\subset C_2(0^{n+1})$.
Hence, for each $k\in\mathcal{I}$ there is a component $F_k$ of $C_r(z)\setminus M_k^*$ such that
\begin{equation}\aligned\label{pFiMi*}
\p F_k\setminus M_k^*\subset (B_r(z')\times\{z_{n+1}-r\})\cup (B_r(z')\times\{z_{n+1}+r\})\cup G_k
\endaligned
\end{equation}
for some closed set $G_k$ in $\p B_r(z')\times[z_{n+1}-r,z_{n+1}+r]$ with $\lim_{k\to\infty}\mathcal{H}^n(G_k)=0$,
and $z=(z',z_{n+1})$.
Let $Y_k$ be a measurable vector field defined in $\pi_*(M_k)\times\R\subset\R^{n+1}$ by
\begin{equation}\aligned\label{Yi}
Y_k=v_{k}\sum_{i,j} g^{ij}_k\f{\p u^1_k}{\p x_i}E_j.
\endaligned
\end{equation}
It's clear that $\pi'(F_k)\subset\pi_*(M_k)$, where $\pi'$ denotes the projection from $\R^{n+1}$ to $\R^n$ by
$$\pi'(x_1,\cdots,x_{n+1})=(x_1,\cdots,x_n).$$
From the argument of \eqref{nuanue*}, there is a constant $c_\La>0$ depending only on $n,m,\La$ such that
\begin{equation}\aligned\label{|Yk|cLa}
|Y_k|\le c_\La.
\endaligned
\end{equation}
From \eqref{CompS} for $Y_k$ and \eqref{pFiMi*}, we have
\begin{equation}\aligned
\int_{\p F_k\cap M^*_k}\f{|D u^1_k|^2}{\sqrt{1+|D u^1_k|^2}v_k}\le\int_{\p F_k\setminus M^*_k}\left|\lan Y_k,E_{n+1}\ran\right|\le\int_{G_k}|Y_k|\le c_\La\mathcal{H}^n(G_k).
\endaligned
\end{equation}
From \eqref{viD1u1i}, the above inequality contradicts to $\lim_{k\to\infty}\mathcal{H}^n(G_k)=0$, and we have proven the claim.

Set
$$M^*_{k,t}=M^*_{k}\cap\{y_1=t\}$$
for any $t\in\R$.
From Lemma 22.2 in \cite{s},
\begin{equation}\aligned\label{nx1ya0*}
\lim_{k\rightarrow\infty}\int_{M_k\cap \mathbf{C}_{5/2}}\left(\left|\n_{M_k}x_1\right|^2+\sum_{\a=2}^m\left|\n_{M_k}y_\a\right|^2\right)=0,
\endaligned
\end{equation}
where $\n_{M_k}$ denotes the Levi-Civita connection on $M_k$ for each $k$. From $M^*_{k}=\pi^*(M_k)$ and \eqref{Mir+ep}, it follows that
\begin{equation}\aligned\label{nx1ya0}
\lim_{k\rightarrow\infty}\int_{M^*_k\cap C_2(0^{n+1})}\left|\n_{M^*_k}x_1\right|^2=0,
\endaligned
\end{equation}
where $\n_{M^*_k}$ denotes the Levi-Civita connection on $M^*_k$ for each $k$.

Let $\g_{k,t}=C_{1}(0^{n+1})\cap M^*_k\cap\{y_1=t\}$ for any $t\in(-1,1)$.
Since $C_{1}(0^{n+1})\cap M^*_k$ is a graph, then $\pi'(\g_{k,t_1})\cap \pi'(\g_{k,t_2})=\emptyset$ for different $t_1,t_2$.
As $\g_{k,t}$ divides $C_{1}(0^{n+1})\cap M^*_k$ into two parts for each $t\in(-1,1)$, then $\pi'(\g_{k,t})$ also divides $\pi'(C_{1}(0^{n+1})\cap M^*_k)$ into two parts.
Moreover, $\pi'(C_{1}(0^{n+1})\cap M^*_k)$ is a domain that converges to the $(n-1)$-ball
$$\{(x_1,\cdots,x_n)\in\R^{n}|\, x_1=0,\, x_2^2+\cdots+x_n^2<1\}$$
in the Hausdorff sense. Note that $\g_{k,t}$ is isometric to $\pi'(\g_{k,t})$.
Let $\pi_1$ be a projection from $\R^{n+1}$ to $\R^n$ defined by $\pi_1(x_1,\cdots,x_{n+1})=(x_2,\cdots,x_{n+1})$.
Then from \eqref{nx1ya0}, it follows that
\begin{equation}\aligned\label{nx1ya00}
\lim_{k\rightarrow\infty}\mathcal{H}^{n}\left(\pi_1(C_2(0^{n+1}))\setminus\pi_1(M^*_k\cap C_2(0^{n+1}))\right)=0.
\endaligned
\end{equation}
Combining the co-area formula, $\mathcal{H}^n(W_k^*)=\mathcal{H}^n(\pi^*(W_k))<\ep_k$ and \eqref{nx1ya0}\eqref{nx1ya00},
we can choose a sequence $t_k\in(-1,1)$ and a countably $(n-1)$-rectifiable set $\g_k^*\subset \g_{k,t_k}$ with $\p\g_k^*\subset \p C_{1}(0^{n+1})\cap\{y_1=t_k\}$ such that
\begin{equation}\aligned\label{Hn-1Gi}
\mathcal{H}^{n-1}(\g_k^*)<(1+\ep)\omega_{n-1},
\endaligned
\end{equation}
and
\begin{equation}\aligned\label{Hn-1Gi*}
\mathcal{H}^{n-1}(\g_k^*\cap W^*_{k})<\ep\omega_{n-1},
\endaligned
\end{equation}
where $\ep$ is the small positive constant defined above.
Denote $\G_k=\g_k^*\times(-1,1)$.
Since $C_{1}(0^{n+1})\setminus M_k^*$ has only two components and $C_{1}(0^{n+1})\cap M_k^*$ converges in the Hausdorff sense to
$$\{(x_1,\cdots,x_n)\in\R^{n}|\, x_1=0,\, x_2^2+\cdots+x_n^2\le1\}\times[-1,1],$$
then for each $k$ there is an open set $\Om_k$ in $C_{1}(0^{n+1})$ such that $(C_{1}(0^{n+1})\cap M_k^*)\setminus \G_k\subset \p\Om_k$ and
$$\p\Om_k\subset M_k^*\cup \G_k\cup(B_{1}(0^n)\times\{-1,1\})\cup S_k,$$
where $S_k\subset\p B_1(0^n)\times[-1,1]$ is a closed set with $\lim_{k\to\infty}\mathcal{H}^n(S_k)=0$.

From \eqref{viD1u1i} and Lemma \ref{app1} in the Appendix I, we have
\begin{equation}\aligned\label{||YiEn|}
|\lan Y_k,E_1\ran|=v_k\left|\sum_{j} g^{1j}_k D_ju^1_k\right|\le1+\ep
\endaligned
\end{equation}
on $M^*_{k,t_k}\setminus W_k^*$ for the sufficiently large $k$.
From \eqref{CompS}, we have
\begin{equation}\aligned
\int_{C_{1}(0^{n+1})\cap M_k^*}\f{|D u^1_k|^2}{\sqrt{1+|D u^1_k|^2}v_k}\le\int_{\G_k}\left|\lan Y_k,E_1\ran\right|+\int_{S_k}|Y_k|.
\endaligned
\end{equation}
Combining \eqref{|Yk|cLa}\eqref{Hn-1Gi}\eqref{Hn-1Gi*}\eqref{||YiEn|}, one has
\begin{equation}\aligned\label{CMkDu1kvk}
&\int_{C_{1}(0^{n+1})\cap M_k^*}\f{|D u^1_k|^2}{\sqrt{1+|D u^1_k|^2}v_k}\\
\le&2(1+\ep)\mathcal{H}^{n-1}(\g_k^*\setminus W^*_{k})+2c_\La\mathcal{H}^{n-1}(\g_k^*\cap W^*_{k})+c_\La\mathcal{H}^n(S_k)\\
\le&2(1+\ep)^2\omega_{n-1}+2c_\La\ep\omega_{n-1}+c_\La\mathcal{H}^n(S_k).
\endaligned
\end{equation}
With \eqref{viD1u1i} and $\lim_{k\to\infty}\mathcal{H}^n(S_k)=0$, letting $k\rightarrow\infty$ in \eqref{CMkDu1kvk} infers
\begin{equation}\aligned\label{limkC1Mk*W*k}
\limsup_{k\rightarrow\infty}\mathcal{H}^n\left(C_{1}(0^{n+1})\cap M_k^*\setminus W^*_k\right)\le2(1+\ep)^2\omega_{n-1}+2c_\La\ep\omega_{n-1}.
\endaligned
\end{equation}
Recalling \eqref{AreaEst1}, we complete the proof by letting $\ep\rightarrow0$ in \eqref{limkC1Mk*W*k}.
\end{proof}

\begin{thm}\label{Stable}
If $|M_k|$ converges to a cylindrical varifold $V$ in $\R^{n+m}$ with
$$\mathrm{spt}V=\{(x_1,\cdots,x_n,y_1,\cdots,y_m)\in\R^{n+m}|\ (x_1,\cdots,x_n)\in V_*,y_{2}=\cdots=y_m=0\}$$
for a closed set $V_*$ in $\R^n$,
then $\mathrm{spt}V$ is stable.
\end{thm}
\textbf{Remark.} Here, the stable $\mathrm{spt}V$ means that $\mathrm{spt}V$ is stable outside its singular set. 
\begin{proof}
Let $M=\mathrm{reg}V$ be the regular part of $V$.
Let $\nu_M$ denote the unit normal vector of $M$ in $\R^{n+1}$.
Since $M$ is open in $\mathrm{reg}V$ by Allard's regularity Theorem \cite{a}, then
for any  point $q\in M$, there is a constant $r=r_q$ such that $\mathbf{B}_{2r}(q)\cap \mathrm{spt}V\subset M$.
Let $q_k\in M_k$ with $q_k\rightarrow q$.
Let $\la_{1,k},\cdots,\la_{n,k}$ be the singular values of $Du_k$ with $\la_{1,k}\ge\cdots\ge\la_{n,k}\ge0$.
From Allard's regularity Theorem \cite{a} and Lemma \ref{multi1}, $M_k\cap \mathbf{B}_{\f74r}(q_k)$ converges to $M\cap \mathbf{B}_{\f74r}(q)$ smoothly.
Then $$\lim_{k\rightarrow\infty}\inf_{M_k\cap \mathbf{B}_{\f32r}(q_k)}\la_{1,k}=\infty.$$
For a point $z\in\mathbf{B}_r(q)\cap M$, let $z_k\in\mathbf{B}_r(q)\cap M_k$ be a sequence of points with $z_k\rightarrow z$ such that $M_k$ is smooth at $z_k$ for each $k$.
Let $\{\nu_k^\a\}_{\a=1}^m$ be a local orthonormal frame of the normal bundle $NM_k$ on $\mathbf{B}_r(q_k)\cap M_k$ such that
\begin{equation}\aligned\label{nuka*}
\nu_k^\a=\f1{\sqrt{1+\la_{\a,k}^2}}\left(-\la_{\a,k}\mathbf{E}_\a+\mathbf{E}_{n+\a}\right)\qquad \mathrm{at}\ z_k
\endaligned
\end{equation}
for each $k,\a$. Then $\nu_k^1(z_k)\rightarrow-\mathbf{E}_1$, $\nu_k^\be(z_k)\rightarrow-\mathbf{E}_{n+\be}$ for each $\be=2,\cdots,m$.

Let us now choose a local orthonormal tangent frame field $\{e_j^k\}_{j=1,...,n}$ on $\mathbf{B}_r(q_k)\cap M_k$, such that
\begin{equation}\aligned\label{ekj*}
e_j^k=\f1{\sqrt{1+\la_{j,k}^2}}\left(\mathbf{E}_j+\la_{j,k}\mathbf{E}_{n+j}\right)\qquad \mathrm{at}\ z_k
\endaligned
\end{equation}
for each $k,j$.
Let $h^{k}_{\a, ij}$ denote the components of the second fundamental form of $M_k$ defined by
$$h^{k}_{\a,ij}=\left\lan\bn_{e_i^k}e_j^k,\nu^\a_k\right\ran.$$
Since $M_k\cap \mathbf{B}_{\f74r}(q_k)$ converges to $M\cap \mathbf{B}_{\f74r}(q)$ smoothly, it follows that
\begin{equation}\aligned
\lim_{k\rightarrow\infty}\sup_{\mathbf{B}_{r}(q_k)\cap M_k}\sum_{\a=2}^m\sum_{i,j=1}^n(h^k_{\a,ij})^2=0.
\endaligned
\end{equation}
With the Cauchy inequality and the above limit, we obtain
\begin{equation}\aligned\label{naMinuai*}
\sum_{l,i\neq j}|h^k_{i,jl}h^k_{j,il}|+\sum_{l,i\neq j}|h^k_{i,il}h^k_{j,jl}|\le\ep_k|B_{M_k}|^2+\ep_k
\endaligned
\end{equation}
on $\mathbf{B}_{r}(q_k)\cap M_k$ for some sequence $\ep_k\to0$ as $k\rightarrow\infty$.
Let $\De_{M_k}$ and $B_{M_k}$ denote the Laplacian and the second fundamental form of $M_k$, respectively.
Combining  \eqref{DeMv-1}, \eqref{naMinuai*} and the bounded 2-dilation condition, we have
\begin{equation}\aligned
\De_{M_k} v_k^{-1}=&-v_k^{-1}\left(\sum_{\a,i,j}(h^k_{\a,ij})^2+\sum_{l,i\neq j}\la_{i,k}\la_{j,k}h^k_{i,jl}h^k_{j,il}-\sum_{l,i\neq j}\la_{i,k}\la_{j,k}h^k_{i,il}h^k_{j,jl}\right)\\
\le&-v_k^{-1}\left(\sum_{\a,i,j}(h^k_{\a,ij})^2+\La\sum_{l,i\neq j}|h^k_{i,jl}h^k_{j,il}|+\La\sum_{l,i\neq j}|h^k_{i,il}h^k_{j,jl}|\right)\\
\le&-v_k^{-1}\left((1-\ep_k)|B_{M_k}|^2-\ep_k\right)
\endaligned
\end{equation}
at $z_k$. Hence the inequality
\begin{equation}\aligned\label{DeMkvkepk}
\De_{M_k} v_k^{-1}\le-v_k^{-1}\left((1-\ep_k)|B_{M_k}|^2-\ep_k\right)
\endaligned
\end{equation}
holds on $\mathbf{B}_{r}(q_k)\cap M_k$.

Let $\phi$ be a smooth function in $M$ with compact support and $\mathrm{spt} \phi\cap \mathrm{spt}V\subset M$. From \eqref{DeMkvkepk} and the covering lemma, we have
\begin{equation}\aligned
\De_{M_k} v_k^{-1}\le-v_k^{-1}\left((1-\widetilde{\ep_k})|B_{M_k}|^2-\widetilde{\ep_k}\right)
\endaligned
\end{equation}
on $\mathrm{spt} \phi\cap M_k$ for some sequence of positive numbers $\widetilde{\ep_k}$ with $\lim_{k\rightarrow\infty}\widetilde{\ep_k}=0$.
Then with the Cauchy inequality, we have
\begin{equation}\aligned
&\int_{M_k}\left((1-\widetilde{\ep_k})|B_{M_k}|^2-\widetilde{\ep_k}\right)\phi^2\\
\le&-\int_{M_k}\phi^2v_k\De_{M_k} v_k^{-1}=-\int_{M_k}v_k^{-2}\lan\nabla_{M_k}(\phi^2v_k),\nabla_{M_k}v_k\ran\\
=&-\int_{M_k}\phi^2v_k^{-2}\left|\nabla_{M_k}v_k\right|^2-2\int_{M_k}\phi v_k^{-1}\lan\nabla_{M_k}\phi,\nabla_{M_k}v_k\ran
\le\int_{M_k}\left|\nabla_{M_k}\phi\right|^2.
\endaligned
\end{equation}
Let $B_M$ denote the second fundamental form of $M$ in $\R^{n+1}$.
Letting $k\rightarrow\infty$ in the above inequality implies
\begin{equation}\aligned
\int_{M}|B_{M}|^2\phi^2\le\int_{M}\left|\nabla_{M}\phi\right|^2.
\endaligned
\end{equation}
We complete the proof.
\end{proof}
Let $M$ be a locally Lipschitz minimal graph over $\R^n$ of codimension $m\ge1$ with bounded
2-dilation of its graphic function.
From Lemma \ref{VGM} and the compactness theorem for integral varifolds (Theorem 42.7 in \cite{s}), we can suppose that a minimal cone $C$ is a tangent cone of $M$ at infinity.
Namely, there is a sequence $r_k\to\infty$ such that $|\f1{r_k}M|$ converges in the varifold sense to $C$ in $\R^{n+m}$ with $0^{n+m}\in C$.
Combining Lemma \ref{multi1} and Theorem \ref{Stable}, we can get Theorem \ref{main2} immediately.

\Section{Neumann-Poincar\'e inequality on stationary indecomposable currents}{Neumann-Poincar\'e inequality on stationary indecomposable currents}

For integers $n\ge2,m\ge1$, a constant $\La>0$ and an open set $\Om\subset\R^n$,
let $\mathcal{M}_{n,m,\La,\Om}$ denote the set containing all the locally Lipschitz
minimal graphs over $\Om$ of arbitrary codimension $m\ge1$ with 2-dilation of
their graphic functions $\le\La$.
Let $\overline{\mathcal{M}}_{n,m,\La,\Om}$ be the closure of the currents associated with minimal graphs in $\mathcal{M}_{n,m,\La,\Om}$. Namely, for an integral current $T$ in $\Om\times\R^m$, we say $T\in\overline{\mathcal{M}}_{n,m,\La,\Om}$ if and only if there is a sequence of minimal graphs $M_k\in\mathcal{M}_{n,m,\La,\Om}$ such that for any open $W\subset\subset\Om\times\R^m$, $T\llcorner W$ is the (weak) limit of $[|M_k\cap W|]$ as $k\to\infty$.
\begin{lem}\label{sptTV}
For a sequence $M_k\in\mathcal{M}_{n,m,\La,\Om}$, let $T$ be a current in $\R^{n+m}$ and $V$ be a multiplicity one rectifiable stationary $n$-varifold in $\R^{n+m}$ so that for any open $W\subset\subset\Om\times\R^m$, $[|M_k\cap W|]$ converges weakly to $T\llcorner W$ and $|M_k\cap W|$ converges to $V\llcorner W$ in the varifold sense. Then $|T|=V$ in $\Om\times\R^m$.
\end{lem}
\begin{proof}
For any open $W\subset\subset\Om\times\R^m$ and any $\omega\in\mathcal{D}^{n-1}(W)$,
\begin{equation}\aligned
\lan\p T,\omega\ran=\lan T,d\omega\ran=\lim_{k\rightarrow\infty}\lan [|M_k|],d\omega\ran=\lim_{k\rightarrow\infty}\lan \p[|M_k|],\omega\ran=0,
\endaligned
\end{equation}
which means $\p T\llcorner W=0$.
Let $W'$ be an open set in $W$ such that $\mathrm{spt}V\cap \overline{W'}$ is contained in the regular part of spt$V$.
From Allard's regularity theorem, $M_k\cap W'$ converges to spt$V\cap W'$ smoothly as $V$ has multiplicity one from Lemma \ref{multi1}.
Hence, for any $p\in\mathrm{spt}V\cap W$ and any $\ep\in(0,d(p,\p W))$, there is an orientation $\xi$ of spt$V\cap \mathbf{B}_\ep(p)$ such that
\begin{equation}\aligned
\lim_{k\rightarrow\infty}\int_{M_k\cap \mathbf{B}_\ep(p)}\lan \xi_k,\e\ran=\int_{\mathrm{spt}V\cap \mathbf{B}_\ep(p)}\lan \xi,\e\ran\qquad \mathrm{for\ any}\ \e\in\mathcal{D}^{n}(\R^{n+m}),
\endaligned
\end{equation}
where $\xi_k$ denotes the orientation of $M_k$ defined by \eqref{orientation}.
For any $\ep'>0$, there is a smooth $\xi_*\in\mathcal{D}^{n}(\mathbf{B}_\ep(p))$ such that
$$\int_{\mathrm{spt} V\cap \mathbf{B}_\ep(p)}(1-|\lan\xi,\xi_*\ran|)<\ep'\mathcal{H}^n(\mathrm{spt}V\cap \mathbf{B}_\ep(p)).$$ Hence
\begin{equation}\aligned\label{Txi*Vep'}
\lan T,\xi_*\ran=&\lim_{k\rightarrow\infty}\lan [|M_k|],\xi_*\ran=\lim_{k\rightarrow\infty}\int_{M_k}\lan \xi_k,\xi_*\ran\\
=&\int_{\mathrm{spt}V\cap \mathbf{B}_\ep(p)}\lan \xi,\xi_*\ran
>(1-\ep')\mathcal{H}^n(\mathrm{spt}V\cap \mathbf{B}_\ep(p)),
\endaligned
\end{equation}
which implies spt$V\cap W\subset\mathrm{spt}T$.
So we obtain spt$T\cap W=\mathrm{spt}V\cap W$. From \eqref{Txi*Vep'}, we get $\mathbb{M}(T\llcorner\mathbf{B}_\ep(p))=\mathcal{H}^n(\mathrm{spt}V\cap \mathbf{B}_\ep(p))$, which means that $T$ has multiplicity one on spt$T\cap W$.
\end{proof}
As a corollary, we immediately have the following corollary.
\begin{cor}\label{MultioneT}
Any current $T\in\overline{\mathcal{M}}_{n,m,\La,\Om}$ has multiplicity one on $\mathrm{spt}T\cap(\Om\times\R^m)$.
\end{cor}

Let $M$ be a locally Lipschitz minimal graph over $\Om$ in $\R^{n+m}$. From \eqref{SobM0}, there holds the isoperimetric inequality
\begin{equation}\aligned\label{ISO**}
\left(\mathcal{H}^{n}(K)\right)^{\f{n-1}n}\le c_n\mathcal{H}^{n-1}(\p K)
\endaligned
\end{equation}
for every bounded closed subset $K$ of $M$ with countably rectifiable boundary $\p K$, where $c_n>0$ is a constant depending only on $n$.

In \cite{b-g}, Bombieri-Giusti proved that any codimension one minimizing current in Euclidean space is indecomposable, and established a Neumann-Poincar$\mathrm{\acute{e}}$ inequality on such currents.
Inspired by their ideas in \cite{b-g}, we introduce a concept '\emph{stationary indecomposable}' for integral currents associated with stationary varifolds as follows. 
\begin{defi}Let $T$ be an integral current such that $|T|$ is a stationary varifold. 
We say $T$ \emph{stationary decomposable} in an open set $W$ if there are two components $T_1,T_2$ of $T\llcorner W$ such that $|T_1|,|T_2|$ are stationary varifolds in $W$. On the contrary, we say $T$ \emph{stationary indecomposable} in $W$.
Furthermore, $T_1$ is said to be a \emph{stationary indecomposable} component of $T\llcorner W$ if $T_1$ is a component of $T\llcorner W$, and $T_1$ is stationary indecomposable in $W$.
\end{defi}
By the above definition, for an integral current $S$ with stationary $|S|$, if $S$ is indecomposable, then $S$ is stationary indecomposable.
\begin{rem}
In general, for an integral decomposable current $T$ with $|T|$ stationary, the indecomposable components of $T$ may be not stationary. For instance, Let $U_1=\{(r\cos\th,r\sin\th)\in\R^2|\ r>0,|\th|<\ep\}$, $U_2=\{(r\cos\th,r\sin\th)\in\R^2|\ r>0,|\th-\pi/3|<\ep\}$, $U_3=\{(r\cos\th,r\sin\th)\in\R^2|\ r>0,|\th-2\pi/3|<\ep\}$, and $U=U_1\cup U_2\cup U_3$. Then $[|\p U|]$ is an integral current with $|\p U|$ stationary, and $[|\p U_1|]$, $[|\p U_2|]$, $[|\p U_3|]$ are 3 components of $[|\p U|]$. Clearly, for each $i=1,2,3$, $[|\p U_i|]$ is not stationary for the suitably small $\ep>0$. However,  for any codimension one area-minimizing current $S$ in Euclidean space,  $S$ is not only indecomposable but also stationary indecomposable from the proof of Theorem 1 in \cite{b-g}.
\end{rem}
\begin{lem}\label{TindecdeT}
Let $\mathbf{T}$ be an integral current in $\overline{\mathcal{M}}_{n,m,\La,B_3}$, and $T$ is a stationary indecomposable component of $\mathbf{T}\llcorner\mathbf{B}_2$.
Then there exists a constant $\de_T>0$ depending on $T$ such that
\begin{equation}\aligned\label{TSB21}
\mathcal{H}^{n-1}(\mathrm{spt}T\cap \p U\cap \mathbf{B}_2)
\ge\de_T\left(\min\{\mathcal{H}^n(\mathrm{spt}T\cap U\cap \mathbf{B}_1),\mathcal{H}^n(\mathrm{spt}T\cap \mathbf{B}_1\setminus U)\}\right)^{\f{n-1}n}
\endaligned
\end{equation}
for any open $U\subset \mathbf{B}_2$ with $(n-1)$-rectifiable $\mathrm{spt}T\cap\p U$.
\end{lem}
\begin{rem}
We do not know yet whether limits of stationary indecomposable currents are still stationary indecomposable. Hence, the coefficient $\de_T$ in \eqref{TSB21} depends on the current $T$.
\end{rem}
\begin{proof}
Let us prove \eqref{TSB21} by contradiction. Suppose that there is a sequence of open $U_k\subset\mathbf{B}_2$ with $(n-1)$-rectifiable $\mathrm{spt}T\cap\p U_k$ such that
\begin{equation}\aligned\label{Tf1kUk}
\mathcal{H}^{n-1}(\mathrm{spt}T\cap \p U_k\cap \mathbf{B}_2)
<\f1k\left(\min\{\mathcal{H}^n(\mathrm{spt}T\cap U_k\cap \mathbf{B}_1),\mathcal{H}^n(\mathrm{spt}T\cap \mathbf{B}_1\setminus U_k)\}\right)^{\f{n-1}n}.
\endaligned
\end{equation}
Let $T_k^+=T\llcorner U_k$ and $T_k^-=T\llcorner\left(\mathbf{B}_2\setminus \overline{U_k}\right)$. Then all $T_k^\pm$ are integer multiplicity currents.
Without loss of generality we can assume $\mathcal{H}^{n-1}(\mathbf{T}\cap\p\mathbf{B}_2)<\infty$, or else from co-area formula we consider a sequence of balls $\mathbf{B}_{2-s_k}$ for some sequence $0<s_k\to0$ with $\mathcal{H}^{n-1}(\mathbf{T}\cap\p\mathbf{B}_{2-s_k})<\infty$.
Hence with \eqref{Tf1kUk}, $\mathbb{M}(\p T^\pm_k)$ are uniformly bounded independent of $k$.
Clearly, $\mathbb{M}(T^+_k\llcorner W)+\mathbb{M}(T^-_k\llcorner W)=\mathbb{M}(T\llcorner W)$.
By Federer-Fleming compactness theorem, there are two integer multiplicity currents $T_*^+,T_*^-$ with spt$T_*^\pm\subset\mathrm{spt}T$ such that $T_k^\pm$ converges weakly to $T^\pm_*$ as $k\to\infty$ up to a choice of a subsequence.

For any open $W\subset \mathbf{B}_2$, $|\omega|_{\mathbf{B}_2}\le1,\omega\in\mathcal{D}^n(\mathbf{B}_2),\mathrm{spt}\omega\subset W$ we have
\begin{equation}\aligned
&T(\omega)=\lim_{k\to\infty}(T^+_k+T^-_k)(\omega)=T^+_*(\omega)+T^-_*(\omega)\le\mathbb{M}(T^+_*\llcorner W)+\mathbb{M}(T^-_*\llcorner W),
\endaligned
\end{equation}
which implies
\begin{equation}\aligned\label{p*T+T-0}
\mathbb{M}(T\llcorner W)\le\mathbb{M}(T^+_*\llcorner W)+\mathbb{M}(T^-_*\llcorner W).
\endaligned
\end{equation}
Moreover, \begin{equation*}\aligned
\mathbb{M}(T^+_*\llcorner W)+\mathbb{M}(T^-_*\llcorner W)\le\liminf_{k\to\infty}\mathbb{M}(T^+_k\llcorner W)+\liminf_{k\to\infty}\mathbb{M}(T^-_k\llcorner W)
\le\mathbb{M}(T\llcorner W).
\endaligned
\end{equation*}
Hence, we deduce
\begin{equation}\aligned\label{T*T+T-}
\mathbb{M}(T\llcorner W)=\mathbb{M}(T^+_*\llcorner W)+\mathbb{M}(T^-_*\llcorner W),
\endaligned
\end{equation}
and\begin{equation}\aligned\label{T*-WTk}
\mathbb{M}(T^\pm_*\llcorner W)=\lim_{k\to\infty}\mathbb{M}(T^\pm_k\llcorner W).
\endaligned
\end{equation}
For any $|\omega'|_{\mathbf{B}_2}\le1,\omega'\in\mathcal{D}^{n-1}(\mathbf{B}_2),\mathrm{spt}\omega'\subset W$, from \eqref{Tf1kUk} we have
\begin{equation}\aligned
\p T^+_*(\omega')=&T^+_*(d\omega')=\lim_{k\to\infty}T^+_k(d\omega')=\lim_{k\to\infty}(T\llcorner \p U_k)(\omega')\\
\le&\limsup_{k\to\infty}\mathcal{H}^{n-1}(\mathrm{spt}T\cap \p U_k\cap \mathbf{B}_2)=0,
\endaligned
\end{equation}
which implies
\begin{equation}\aligned\label{pT+-***}
\mathbb{M}(\p T^+_*\llcorner W)=\mathbb{M}(\p T^-_*\llcorner W)=0.
\endaligned
\end{equation}
Since $T$ has multiplicity one on spt$T\cap W$ for any open $W\subset \mathbf{B}_2$ from Corollary \ref{MultioneT}, $T^\pm_*$ has multiplicity one on its support.

From the co-area formula, for almost all $1< t<2$, we have
\begin{equation}\aligned\label{pptT*Bt}
\f{\p}{\p t}\mathbb{M}(T_k^\pm\llcorner \mathbf{B}_t)\ge\mathbb{M}(\p(T_k^\pm\llcorner \mathbf{B}_t))-\mathbb{M}(\p T_k^\pm\llcorner \mathbf{B}_t).
\endaligned
\end{equation}
With \eqref{ISO**} and \eqref{Tf1kUk}, for almost all $1<t<2$, we get
\begin{equation}\aligned
\f{\p}{\p t}\mathbb{M}(T_k^\pm\llcorner \mathbf{B}_t)>&\f1{c_n}\left(\mathbb{M}(T_k^\pm\llcorner \mathbf{B}_t)\right)^{\f{n-1}n}-\f1k\left(\mathbb{M}(T_k^\pm\llcorner \mathbf{B}_1)\right)^{\f{n-1}n}\\
\ge&\left(\f1{c_n}-\f1k\right)\left(\mathbb{M}(T_k^\pm\llcorner \mathbf{B}_t)\right)^{\f{n-1}n},
\endaligned
\end{equation}
which implies $\mathbb{M}(T_k^\pm\llcorner \mathbf{B}_t)>0$ for any $t>1$ and any $k>c_n$. Then we solve the above differential inequality and from \eqref{T*-WTk} we get
\begin{equation}\aligned\label{T*Btlow}
\mathbb{M}(T_*^\pm\llcorner \mathbf{B}_t)=\lim_{k\to\infty}\mathbb{M}(T_*^\pm\llcorner \mathbf{B}_t)\ge\left(\f{t-1}{2nc_n}\right)^n
\endaligned
\end{equation}
for each $t\in[1,2]$.

For any small fixed $\ep>0$ and any integer $k\ge0$, there is a collection of balls $\{\mathbf{B}_{r_l}(\mathbf{x}_l)\}_{l=1}^{N_{k,\ep}}$ with $r_l<\ep$ such that $\mathrm{spt}T\cap\p U_k\cap\mathbf{B}_2\subset\cup_{l=1}^{N_{k,\ep}}\mathbf{B}_{r_l}(\mathbf{x}_l)$, and
\begin{equation}\aligned\label{rln-1TpUkep}
\omega_{n-1}\sum_{l=1}^{N_{k,\ep}}r_l^{n-1}<\mathcal{H}^{n-1}(\mathrm{spt}T\cap \p U_k\cap \mathbf{B}_2)+\ep.
\endaligned
\end{equation}
Let $\e_l$ be a Lipschitz function on $\mathbf{B}_{3}$ with $0\le\e_l\le1$ such that $\e_l=0$ on $\mathbf{B}_{r_l}(\mathbf{x}_l)$, $\e_l=1$ on $\mathbf{B}_{3}\setminus\mathbf{B}_{2r_l}(\mathbf{x}_l)$,  $|\bn \e_l|=r_l^{-1}$ on $\mathbf{B}_{2r_l}(\mathbf{x}_k)\setminus\mathbf{B}_{r_l}(\mathbf{x}_k)$, where $\bn$ denotes the Levi-Civita connection of $\R^{n+m}$.
Set $\e_{k,\ep}=\prod_{l=1}^{N_{k,\ep}}\e_l\in C^1$. Then $\e_{k,\ep}=0$ on a neighborhood of $\mathrm{spt}T\cap \p U_k\cap \mathbf{B}_2$ and
\begin{equation}\aligned\label{naekepchi}
|\bn\e_{k,\ep}|\le\sum_{l=1}^{N_{k,\ep}}|\bn \e_l|\le\sum_{l=1}^{N_{k,\ep}}r_l^{-1}\chi_{_{\mathbf{B}_{2r_l}(\mathbf{x}_l)}}.
\endaligned
\end{equation}
Let $\phi$ be a smooth function with compact support in $\mathbf{B}_2$. Let $\n_T$ denote the Levi-Civita connection of the regular part of $|T|$. Since $|T^+_k|$ is a multiplicity one stationary $n$-varifold in $U_k$, then every position function $x_i$ is weakly harmonic on spt$T^+_k$ for each $i=1,\cdots,m+n$ (see \cite{c-m1} for instance). Hence
\begin{equation*}\aligned
0=\int_{\mathrm{spt}T^+_k}\lan\n_T x_i,\n_T(\phi\e_{k,\ep})\ran=\int_{\mathrm{spt}T^+_k}\e_{k,\ep}\lan\n_T x_i,\n_T\phi\ran+\int_{\mathrm{spt}T^+_k}\phi\lan\n_T x_i,\n_T\e_{k,\ep}\ran.
\endaligned
\end{equation*}
With \eqref{naekepchi}, it follows that
\begin{equation}\aligned\label{sptT+kxipe}
&\left|\int_{\mathrm{spt}T^+_k}\e_{k,\ep}\lan\n_T x_i,\n_T\phi\ran\right|\le\int_{\mathrm{spt}T^+_k}|\phi||\n_T\e_{k,\ep}|\\
\le&\sup_{\mathbf{B}_2}|\phi|\sum_{l=1}^{N_{k,\ep}}r_l^{-1}\mathcal{H}^n\left(\mathrm{spt}T\cap\mathbf{B}_{2r_l}(\mathbf{x}_l)\right).
\endaligned
\end{equation}
From Lemma \ref{VGM}, there is a constant $c_{n,\La}\ge1$ depending only on $n,\La$ so that
\begin{equation}\aligned\label{VolTBrx}
\mathcal{H}^n(\mathrm{spt}T\cap \mathbf{B}_r(\mathbf{x}))\le c_{n,\La}\sqrt{m}\omega_nr^n
\endaligned
\end{equation}
for any $\mathbf{B}_r(\mathbf{x})\subset\mathbf{B}_{5/2}$.  Combining \eqref{rln-1TpUkep}\eqref{sptT+kxipe}\eqref{VolTBrx}, we have
\begin{equation}\aligned
&\left|\int_{\mathrm{spt}T^+_k}\e_{k,\ep}\lan\n_T x_i,\n_T\phi\ran\right|\le 2^nc_{n,\La}\sqrt{m}\omega_n\sup_{\mathbf{B}_2}|\phi|\sum_{l=1}^{N_{k,\ep}}r_l^{n-1}\\
\le& 2^nc_{n,\La}\sqrt{m}\f{\omega_n}{\omega_{n-1}}\sup_{\mathbf{B}_2}|\phi|\left(\mathcal{H}^{n-1}(\mathrm{spt}T\cap \p U_k\cap \mathbf{B}_2)+\ep\right).
\endaligned
\end{equation}
Letting $\ep\to0$ in the above inequality implies
\begin{equation}\aligned
&\left|\int_{\mathrm{spt}T^+_k}\lan\n_T x_i,\n_T\phi\ran\right|\le 2^nc_{n,\La}\sqrt{m}\f{\omega_n}{\omega_{n-1}}\sup_{\mathbf{B}_2}|\phi|\mathcal{H}^{n-1}\left(\mathrm{spt}T\cap\mathbf{B}_2\cap\p U_k\right).
\endaligned
\end{equation}
Up to a choice of a subsequence, we can assume that $|T_k^\pm|$ converges to $|T^\pm_*|$ in the varifold sense as $k\to\infty$. With \eqref{Tf1kUk}, we have
\begin{equation}\aligned
&\left|\int_{\mathrm{spt}T^+_*}\lan\n_T x_i,\n_T\phi\ran\right|=\lim_{k\to\infty}\left|\int_{\mathrm{spt}T^+_k}\lan\n_T x_i,\n_T\phi\ran\right|=0
\endaligned
\end{equation}
for each $i=1,\cdots,n+m$. In other words, $T^+_*$ is stationary, and similarly $T^-_*$ is also stationary.
Combining \eqref{T*T+T-}\eqref{pT+-***}\eqref{T*Btlow}, we conclude that $T$ is stationary decomposable in $\mathbf{B}_2$. It is a contradiction.
This completes the proof.
\end{proof}

Using Lemma \ref{TindecdeT}, we can prove a Neumann-Poincar$\mathrm{\acute{e}}$ inequality on stationary indecomposable components of limits of minimal graphs.
\begin{lem}\label{NPITHM}
Let $\mathbf{T}$ be an integral current in $\overline{\mathcal{M}}_{n,m,\La,B_{3}}$, and $T$ be a stationary indecomposable component of $\mathbf{T}\llcorner\mathbf{B}_{2}$. Denote $\mathcal{S}$ be the singular set of $T$, then
there exists a constant $\th_T>0$ depending on $T$ such that
\begin{equation}\aligned\label{NPI}
\int_{\mathrm{spt}T\cap \mathbf{B}_r}|f-\bar{f}_r|\le \th_Tr\int_{\mathrm{spt}T\cap \mathbf{B}_{2r}}|\n_T f|
\endaligned
\end{equation}
for any $r\in(0,1]$ and any bounded $C^1$-function $f$ on $\mathbf{B}_{2}\setminus\mathcal{S}$, where
$\n_T$ is the Levi-Civita connection of the regular part of $|T|$, and $\bar{f}_r=\f1{\mathcal{H}^n(\mathrm{spt}T\cap \mathbf{B}_r)}\int_{\mathrm{spt}T\cap \mathbf{B}_r}f$.
\end{lem}
\begin{proof}
Without loss of generality, let $f$ be not a constant, then we only need to prove \eqref{NPI} for $r=1$. Let $M=\mathrm{spt}T$, and $\bar{f}$ be the average of $f$ on $M\cap \mathbf{B}_1$, i.e.,
$$\bar{f}=\f1{\mathcal{H}^n(M\cap \mathbf{B}_1)}\int_{M\cap \mathbf{B}_1}f.$$
Let $U^+_{s,t}=\{y\in M\cap\mathbf{B}_s\setminus\mathcal{S}|\ f(y)>\bar{f}+t\}$, $U^-_{s,t}=\{y\in M\cap\mathbf{B}_s\setminus\mathcal{S}|\ f(y)<\bar{f}+t\}$ for all $s>0$ and $t\in\R$.
From Sard's theorem, for almost all $t$, $\p U^\pm_{s,t}$ is $C^1$ in $\mathbf{B}_s$ outside $\mathcal{S}$. In particular, $\p U^\pm_{s,t}$ is $(n-1)$-rectifiable for almost all $t$.

Without loss of generality, we assume $\mathcal{H}^n(U^+_{1,0})\le\mathcal{H}^n(U^-_{1,0})$. Then $\mathcal{H}^n(U^+_{1,t})\le\mathcal{H}^n(U^-_{1,t})$ for any $t\ge0$.
From Lemma \ref{TindecdeT}, we have
$$\mathcal{H}^{n-1}(\p U^+_{2,t}\cap \mathbf{B}_{2})\ge\de_T\left(\mathcal{H}^n(U^+_{1,t})\right)^{\f{n-1}n}
\ge\de_T\left(\mathcal{H}^n(M\cap \mathbf{B}_{1})\right)^{-\f{1}n}\mathcal{H}^n(U^+_{1,t}).$$
Using now the co-area formula,
\begin{equation}\aligned
\int_{U^+_{1,0}}(f-\bar{f})=&\int_0^\infty\mathcal{H}^n(U^+_{1,t})dt
\le\f{\left(\mathcal{H}^n(M\cap \mathbf{B}_{1})\right)^{\f{1}n}}{\de_T}\int_0^\infty\mathcal{H}^{n-1}(\p U^+_{2,t}\cap \mathbf{B}_{2})dt\\
\le& \f1{\de_T}\left(\mathcal{H}^n(M\cap \mathbf{B}_{1})\right)^{\f{1}n}\int_{M\cap \mathbf{B}_{2}}|\n_T f|,
\endaligned
\end{equation}
and then
\begin{equation}\aligned
&\int_{M\cap \mathbf{B}_{1}}|f-\bar{f}|=\int_{U^+_{1,0}}(f-\bar{f})-\int_{U^-_{1,0}}(f-\bar{f})\\
=&2\int_{U^+_{1,0}}(f-\bar{f})\le \f2{\de_T}\left(\mathcal{H}^n(M\cap \mathbf{B}_{1})\right)^{\f{1}n}\int_{M\cap \mathbf{B}_{2}}|\n_T f|.
\endaligned
\end{equation}
Using  \eqref{VGM1}, we complete the proof.
\end{proof}

Let $\mathbf{T}\in\overline{\mathcal{M}}_{n,m,\La,B_{3}}$ with $0^{n+m}\in\rm{spt}\mathbf{T}$. Then $\mathcal{H}^n(\mathrm{spt}\mathbf{T}\cap \mathbf{B}_{r})\ge\omega_nr^n$ for every $r\in(0,3)$.
From \eqref{VGM1}, the exterior ball $\mathrm{spt}\mathbf{T}\cap \mathbf{B}_{r}$ admits volume doubling property. Namely, there is a constant $c_{n,m,\La}\ge1$ depending only on $n,m,\La$ such that
\begin{equation}\aligned
\mathcal{H}^n(\mathrm{spt}\mathbf{T}\cap \mathbf{B}_{2r})\le c_{n,m,\La}\omega_n2^nr^n\le c_{n,m,\La}2^n\mathcal{H}^n(\mathrm{spt}\mathbf{T}\cap \mathbf{B}_{r})
\endaligned
\end{equation}
for every $r\in(0,1]$.
From the Sobolev inequality \cite{m-s}, nonnegative subharmonic functions on stationary varifolds admit the mean value inequality on spt$\mathbf{T}$
(see \cite{g-t} for instance).
Since Neumann-Poincar$\mathrm{\acute{e}}$ inequality \eqref{NPI}  holds on a stationary indecomposable component $T$ of $\mathbf{T}$, by De Giorgi-Nash-Moser iteration (see \cite{m0}\cite{m}, or \cite{li}, or Theorem 3.2 in \cite{d3} for instance)
there holds the mean value inequality for superharmonic functions on spt$T$.  Hence, we get Harnack's inequality for weakly harmonic functions on spt$T$ as follows.
\begin{pro}\label{Tsuperhar}
Let $\mathbf{T}$ be an integral current in $\overline{\mathcal{M}}_{n,m,\La,B_{3}}$, and $T$ be a stationary indecomposable component of $\mathbf{T}\llcorner\mathbf{B}_{2}$ with $0\in\mathrm{spt}T$.
For any $f\in C^1(\mathbf{B}_{2})$, if $f$ satisfies $\De_{T} f=0$ in the distribution sense, and $f\ge0$ on $\mathbf{B}_{2}\cap \mathrm{spt}T$, then
\begin{equation}\aligned
\sup_{\mathrm{spt}T\cap \mathbf{B}_r}f\le \Th_T\inf_{\mathrm{spt}T\cap \mathbf{B}_r}f\qquad \mathrm{for\ any}\ r\in(0,1],
\endaligned
\end{equation}
where $\De_T$ is the Laplacian of the regular part of spt$T$, $\Th_{T}>0$ is a constant depending on $n,m,\La,T$.
\end{pro}

\Section{A Liouville theorem for minimal graphs of bounded 2-dilation}{A Liouville theorem for minimal graphs of bounded 2-dilation}

For codimension 1, De Giorgi \cite{dg} proved that any limit of non-flat minimal graphs over $\R^n$ in $\R^{n+1}$ is a cylinder. For arbitrary codimensions, we have the following splitting.
\begin{lem}\label{codim1splitting}
Let $n\ge2$, $m\ge1$ be integers, and $\La$ be a positive constant.
For a current $T\in\overline{\mathcal{M}}_{n,m,\La,\R^n}$, if spt$T$ is a non-flat cone living in $\R^{n+1}\subset\R^{n+m}$, then spt$T$ splits off a line isometrically perpendicular to the $n$-plane $\{(x,0^m)\in\R^n\times\R^m|\, x\in\R^n\}$ .
\end{lem}
\begin{proof}
Let $\pi^*$ be the projection defined in \eqref{pi*}.
From the assumption, we can treat $\pi^*(\mathrm{spt}T)$ as a codimension one cone in $\R^{n+1}\times\{0^{m-1}\}=\{(x_1,\cdots,x_{n+1},0^{m-1})\in\R^{n+m}|\ (x_1,\cdots,x_{n+1})\in\R^{n+1}\}$.
From Lemma \ref{sptTV}, $T$ has multiplicity one, and $|\mathrm{spt}T|$ is a minimal cone in $\R^{n+1}$.
If $\mathrm{spt}T\cap\{(0^n,y)\in\R^n\times\R|\, y\in\R^m\}=\{0^{n+m}\}$, then spt$T$ can be written as an entire minimal graph in $\R^{n+1}$, which implies flatness of spt$T$ by the regularity result of  De Giorgi \cite{dg0} since a regular everywhere cone is flat. Hence, without loss of generality, we assume $(0^n,-1,0^{m-1})\in\mathrm{spt}T$.

Let $M_k$ be a sequence of minimal graphs over $\R^{n}$ in $\R^{n+m}$ of 2-dilation bounded by $\La$ such that the $n$-current $[|M_k|]\in\mathcal{D}_n(\R^{n+m})$ (associated with $M_k$) converges weakly to $T$. Let $\omega$ be a smooth $n$-form defined by $\sum_{i=1}^{n+1}f_idx_1\wedge\cdots\wedge \widehat{dx_i}\wedge\cdots\wedge x_{n+1}$ with compact support in $\R^{n+1}$. Let $\tilde{\omega}$ be a smooth $n$-form with compact support in $\R^{n+m}$ so that $\tilde{\omega}(x_1,\cdots,x_{n+1},0,\cdots,0)=\omega(x_1,\cdots,x_{n+1})$ for each $(x_1,\cdots,x_{n+1})\in\R^{n+1}$.
Then from \eqref{PushforwfT}, it follows that
\begin{equation}\aligned\label{TomegaMk*}
\lan \pi^*(T),\omega\ran=\lan T,\tilde{\omega}\ran=\lim_{k\to\infty}\lan [|M_k|],\tilde{\omega}\ran=\lim_{k\to\infty}\lan [|\pi^*(M_k)|],\omega\ran.
\endaligned
\end{equation}
Let $M_k^*=\pi^*(M_k)\subset\R^{n+1}$, then the above limit implies that $[|M_k^*|]$ converges weakly to $\pi^*(T)$.
Let $u_k$ denote the graphic function of $M_k^*$, and
$$U_k=\{(x,t)\in\R^{n+1}|\, t<u_k(x)\}.$$
From Lemma \ref{VGM*} and Federer-Fleming compactness theorem, $[|U_k|]$ converges weakly to a current $[|U|]$ for some open subset $U\subset\R^{n+1}$. With \eqref{TomegaMk*}, we have
\begin{equation}\aligned
\lan [|\p U|],\omega\ran=&\lan [|U|],d\omega\ran=\lim_{k\to\infty}\lan [|U_k|],d\omega\ran=\lim_{k\to\infty}\lan [|\p U_k|],\omega\ran\\
=&\lim_{k\to\infty}\lan [|M_k^*|],\omega\ran=\lan \pi^*(T),\omega\ran,
\endaligned
\end{equation}
which implies $\pi^*(T)=[|\p U|]$.
In particular, $U$ is also a cone. Since $(0^n,-1,0^{m-1})\in\mathrm{spt}T$, we consider a family of open sets
$$U_{0^n,t}=\{y+(0^n,t)\in\R^{n+1}|\, y\in U\}=\{y+(0^n,1)\in\R^{n+1}|\, y\in tU\}\supset U$$
for each $t>0$. By Federer-Fleming compactness theorem again, there is a sequence $t_k\to\infty$ such that $[|U_{0^n,t_k}|]$ converges weakly to a current $[|W|]$ for some open subset $W=\{(x,t)\in\R^{n+1}|\, x\in \mathscr{W}\}$ with some open $\mathscr{W}\subset\R^n$.
It's clear that $U\subset W$.
Since $\p W=\p\mathscr{W}\times\R$ is stable minimal from Theorem \ref{Stable}, then $|\p\mathscr{W}|$ is a stable minimal cone in $\R^n$.

Let $\Si=\p U\setminus\p W$. If $\Si=\emptyset$, then $\p U=\mathrm{spt}T$ splits off a line $\{(0^n,t)\in\R^{n+1}|\, t\in\R\}$ isometrically. Now let us assume $\Si\neq\emptyset$, or else we complete the proof.
Let us deduce a contradiction. From Theorem 3.2 in \cite{wn1} by Wickramasekera, it follows that
\begin{equation}\aligned\label{pU*U*inftyBry}
\mathcal{H}^{n-1}(\p U\cap\p W\cap B_r(y))>0\qquad \mathrm{for\ any}\ y\in\p U\cap\p W,\ r>0.
\endaligned
\end{equation}
Let $\mathcal{S}$ denote the singular set of $\p W$.
By the strong maximum principle, $\p U\cap\p W$ is a closed subset in $\mathcal{S}$.

For any $\be\ge0$, let $\mathcal{H}^\be_\infty$ be a measure defined by
\begin{equation}\aligned\label{HbeinftyW}
\mathcal{H}^\be_\infty(W)=\omega_\be 2^{-\be}\inf\left\{\sum_{k=1}^\infty(\mathrm{diam} U_k)^\be\bigg|\, W\subset\bigcup_{k=1}^\infty U_k\subset\R^{n+1}\right\}
\endaligned
\end{equation}
for any set $W$ in $\R^{n+1}$, where $\omega_\be=\f{\pi^{\be/2}}{\G(\f \be2+1)}$, and $\G(r)=\int_0^\infty e^{-t}t^{r-1}dt$ is the gamma function for $0<r<\infty$.
From Lemma 11.2 in \cite{gi}, if $\mathcal{H}^\be(W)>0$, then $\mathcal{H}^\be_\infty(W)>0$.
From the argument of Proposition 11.3 in \cite{gi} and \eqref{pU*U*inftyBry}, there is a point $q\in \p U\cap\p W\setminus\{0^{n+1}\}$ and a sequence $r_k\rightarrow0$ such that
\begin{equation}\aligned\label{n-1HpUpW}
\mathcal{H}^{n-1}_\infty\left(\p U\cap\p W\cap B_{r_k}^{n+1}(q)\right)>2^{-n-2}\omega_{n-1} r_k^{n-1}.
\endaligned
\end{equation}
Let $U_{q,k}=\f1{r_k}\{x+q|\, x\in U\}$, $W_{q,k}=\f1{r_k}\{x+q|\, x\in W\}$, $\G_{q,k}=\p U_{q,k}\cap\p W_{q,k}$, $\mathcal{S}_{q,k}=\f1{r_k}\{x+q|\, x\in \mathcal{S}\}$.
Then \eqref{n-1HpUpW} implies
\begin{equation}\aligned
\mathcal{H}^{n-1}_\infty\left(\G_{q,k}\cap B^{n+1}_{1}(0^{n+1})\right)>2^{-n-2}\omega_{n-1}.
\endaligned
\end{equation}
From Federer-Fleming compactness theorem, without loss of generality, there are two open sets $U_*,W_*$ in $\R^{n+1}$ so that $[|U_{q,k}|],[|W_{q,k}|]$ converge to $[|U_*|],[|W_*|]$ in the current sense as $k\rightarrow\infty$, respectively. From the constructions of $U_{q,k},W_{q,k}$, $|\p U_*|$, $|\p W_*|$ are minimal cones both splitting off a line $l\neq\{tE_{n+1}|\ t\in\R\}$ isometrically.
Up to choosing the subsequence, we may assume that $\G_{q,k}$ converges to a closed set $\G_*$ in the Hausdorff sense.
Let $\mathcal{S}_*$ be the singular set of $\p W_*$.
If $y_k\in \mathcal{S}_{q,k}$ and $y_k\rightarrow y_*\in\p W_*$, then
it's clear that $y_*$ is a singular point of $\p W_*$ by Allard's regularity theorem and multiplicity one of $\p W_*$, which implies $\limsup_{k\rightarrow\infty}\mathcal{S}_{q,k}\subset\mathcal{S}_*$.
With $\p U\cap\p W\subset\mathcal{S}$, it follows that $\G_*\subset\mathcal{S}_*$.
Analog to the proof of Lemma 11.5 in \cite{gi}, we have
\begin{equation}\aligned
\mathcal{H}^{n-1}_\infty\left(\G_{*}\cap B^{n+1}_{1}(0^{n+1})\right)\ge 2^{-n-2}\omega_{n-1}.
\endaligned
\end{equation}

Let us continue the above procedure. By dimension reduction argument,
there are a 2-dimensional open cone $V_0\subset\R^2$ with $|\p V_0|$ minimal, an open cone $V\subset V_0\times\R\subset\R^3$ with $|\p V|$ minimal, a sequence of open sets $V_i,W_i$ (obtained from scalings and translations of $U,W$, respectively) such that $\p V_0$ has an isolated singularity at the origin, and $[|W_i|]$ converges to $[|V_0\times\R^{n-1}|]$, $[|V_i|]$ converges to $[|V\times\R^{n-2}|]$ in the current sense.
Since $\Si$ is a minimal graph over $\mathscr{W}$, then $\Si$ is smooth stable.
From Theorem 2 of \cite{s-s} by Schoen-Simon (see also Lemma \ref{cross}), we get $\p V\neq\p V_0\times\R$.
It's well-known that a smooth 1-dimensional minimal surface(geodesic) in $\mathbb{S}^2$ is a collection of circles of radius one, which implies that $\p V$ is a collection of planes through $0^3\in\R^3$. Hence $\p V=\p V_0\times\R$. It's a contradiction. We complete the proof.
\end{proof}
\begin{rem}
Cheeger-Naber \cite{c-n} showed the Minkowski content estimation on the quantitative singular sets of stationary varifolds.
Using it we can simplify the proof of Lemma \ref{codim1splitting}.
Namely, without dimension reduction argument, we immediately have the following conclusion in Lemma \ref{codim1splitting}:
there are a point $x^*\in\p U\cap\p W$, a 2-dimensional open cone $V_0\subset\R^2$ with $|\p V_0|$ minimal, and an open cone $V\subset V_0\times\R\subset\R^3$ with $|\p V|$ minimal such that $\p V_0$ has an isolated singularity at the origin, and $\f1{r_i}([|W|],x^*)$ converges to $([|V_0\times\R^{n-1}|],0^{n+1})$, $\f1{r_i}([|U|],x^*)$ converges to $([|V\times\R^{n-2}|],0^{n+1})$ for some sequence $r_i\to0$.
\end{rem}

From Proposition \ref{Tsuperhar} and Lemma \ref{codim1splitting}, we can obtain a Liouville type theorem for minimal graphs as follows.
\begin{thm}\label{Lioua1}
Let $M=\mathrm{graph}_u$ be a locally Lipschitz minimal graph over $\R^n$ of codimension $m\ge2$ with bounded 2-dilation of $u=(u^1,\cdots,u^m)$. Suppose
\begin{equation}\aligned\label{uaxline}
\limsup_{r\rightarrow\infty}\left(r^{-1}\sup_{\mathbf{B}_r\cap M}u^\a\right)\le0
\endaligned
\end{equation}
for each $\a\in\{2,\cdots,m\}$. If
\begin{equation}\aligned\label{u1xline}
\liminf_{r\rightarrow\infty}\left(r^{-1}\sup_{B_r}u^1\right)<\infty,
\endaligned
\end{equation}
then $M$ is flat.
\end{thm}
\begin{proof}
Assume that $u$ has bounded $2$-dilation by a constant $\La>0$.
From \eqref{u1xline}, there are a constant $\Th>0$ and a sequence of numbers $r_k\rightarrow\infty$ such that
\begin{equation}\aligned\label{u1xiline}
\sup_{B_{r_k}}u^1\le \Th r_k.
\endaligned
\end{equation}
Recalling that $[|M|]\in\mathcal{D}_n(\R^{n+m})$ is the $n$-current associated with $M$ and its orientation \eqref{orientation}.
From Lemma \ref{sptTV}, we can assume that $[|\f1{r_k}M|]$ converges weakly as $k\to\infty$ to a multiplicity one current $T\neq0$ with $0\in\mathrm{spt}T$ and $\p T=0$.
Moreover, the varifold associated with spt$T$ is a minimal cone in $\R^{n+m}$ with the vertex at the origin.
Let $\mathbf{x}=(x_1,\cdots,x_{n+m})$ be the position vector in $\R^{n+m}$. From \eqref{uaxline}, we get
\begin{equation}\aligned\label{sptxnaR}
\sup_{\mathrm{spt}T\cap\mathbf{B}_1}x_{n+\a}\le0
\endaligned
\end{equation}
for each $\a\in\{2,\cdots,m\}$.

Suppose that $T$ is stationary decomposable in $\mathbf{B}_1$. Let $T'$ be a stationary component of $T\llcorner\mathbf{B}_1$.
Then from $\p T'=0$ in $\mathbf{B}_1$, we conclude that spt$T'$ is a truncated cone. In particular, $0\in\mathrm{spt}T'$. Hence we have
\begin{equation}\aligned\label{T'lowomega}
\mathbb{M}(\mathrm{spt}T'\cap \mathbf{B}_r)\ge\omega_nr^n\qquad \mathrm{for\ any}\  r\in(0,1].
\endaligned
\end{equation}
From Lemma \ref{VGM*}, we have
\begin{equation}\aligned\label{TupLaomega}
\mathbb{M}(\mathrm{spt}T\cap \mathbf{B}_r)\le C_{n,\La}\omega_nr^n.
\endaligned
\end{equation}
If $T'$ is stationary decomposable in $\mathbf{B}_1$, then we consider a stationary component $T''$ of $T'\llcorner\mathbf{B}_1$. Clearly, spt$T''$ is a truncated cone and $0\in\mathrm{spt} T''$.
Combining \eqref{T'lowomega}\eqref{TupLaomega}, the procedure of decomposition will cease after finite times.
Hence, there is a collection of indecomposable stationary components $T_1,\cdots,T_l$ of $T\llcorner\mathbf{B}_1$, where $l$ is a positive integer $\le C_{n,\La}$. In particular, all $\mathrm{spt}T_1,\cdots,\mathrm{spt}T_l$ are truncated cones.

From Proposition \ref{Tsuperhar} and \eqref{sptxnaR}, we get $x_{n+\a}\equiv0$ on the truncated cone spt$T_k$ for each $\a\in\{2,\cdots,m\}$ and $k\in\{1,\cdots,l\}$ as $x_{n+\a}$ is weakly harmonic in spt$T_k\cap \mathbf{B}_1$.
In particular, the varifold associated with spt$T$ is a minimal cone living in an $(n+1)$-dimensional Euclidean space $\R^{n+1}$.
From \eqref{u1xiline}, we conclude that
\begin{equation}\aligned
\sup_{\mathrm{spt}T\cap(B_1\times\R^m)}x_{n+1}<\infty.
\endaligned
\end{equation}
From Lemma \ref{codim1splitting}, it follows that
\begin{equation}\aligned
\sup_{\mathrm{spt}T\cap(B_{1}\times\R^m)}|x_{n+1}|<\infty.
\endaligned
\end{equation}
Note that $[|\f1{r_k}M|]\rightharpoonup T$ and spt$T$ is a cone. Then spt$T$ can be written as an entire graph over $\R^n$ with the graphic function $(\phi,0,\cdots,0)$, where $\phi$ is 1-homogenous on $\R^n$.
Then
\begin{equation}\aligned
\mathrm{graph}_\phi=\{(x,\phi(x))\in\R^n\times\R|\, x\in\R^n\}
\endaligned
\end{equation}
is a (Lipschitz) minimal graph in $\R^{n+1}$. Therefore, the regularity theorem of  De Giorgi \cite{m} implies that $\phi$ is linear, and then spt$T$ is flat.
Since $T$ has multiplicity one on spt$T$, then Allard's regularity theorem yields the proof.
\end{proof}

\Section{Bernstein theorem for minimal graphs of bounded slope}{Bernstein theorem for minimal graphs of bounded slope}

For a domain $\Om\subset\R^n$, let $M=\mathrm{graph}_u$ be a smooth minimal graph over $\Om$ of codimension $m\ge2$.
Let $g_{ij}=\de_{ij}+\sum_{\a=1}^m \p_iu^\a\p_ju^\a$,
and $(g^{ij})$ be the inverse matrix of $(g_{ij})$.
Let $v$ be the {\it slope} function of $M$ defined by
\begin{equation}\aligned\label{v}
v=\sqrt{\det g_{ij}}=\sqrt{\det\left(\de_{ij}+\sum_{\a=1}^m \f{\p u^\a}{\p x_i}\f{\p u^\a}{\p x_j}\right)}.
\endaligned
\end{equation}
\begin{lem}\label{sqrt2logv}
Let $\la_1,\cdots,\la_n$ be the singular eigenvalues of $Du$ at any point of $\Om$.
If $\la_1\ge\cdots\ge\la_n\ge0$ and $\la_1^2\la_i^2\le2+\la_i^2$ for all $i\ge2$, then
\begin{equation}\aligned\label{logvsqrt2}
\De_M\log v\ge\sum_{\a>n,i,j}h^2_{\a,ij}+\sum_i(1+\la_i^2)h_{i,ii}^2.
\endaligned
\end{equation}
Moreover, when the above inequality becomes an  equality, $(2+\la_i^2)h_{i,ij}^2+h_{j,ii}^2+2\la_i\la_jh_{i,ji}h_{j,ii}=0$ for all $i\neq j$.
\end{lem}
\begin{proof}
From \eqref{Delogv}, we have (see also the decomposition (4.16) in \cite{d-x-y})
\begin{equation}\aligned\label{DeMlogdecom}
&\De_M\log v=\sum_{\a,i,j}h^2_{\a,ij}+\sum_{i,j}\la_i^2h^2_{i,ij}+\sum_{k,i\neq j}\la_i\la_jh_{i,jk}h_{j,ik}\\
=&\sum_{\a>n,i,j}h^2_{\a,ij}+\sum_i(1+\la_i^2)h_{i,ii}^2+\sum_{i\neq j}\left((2+\la_i^2)h_{i,ij}^2+h_{j,ii}^2+2\la_i\la_jh_{i,ji}h_{j,ii}\right)\\
&+\sum_{i,j,k\ \mathrm{mutually\ distinct}}\left(h_{k,ij}^2+\la_i\la_jh_{i,jk}h_{j,ik}\right),
\endaligned
\end{equation}
where
\begin{equation}\aligned\label{ijkmutdistinct}
&\sum_{i,j,k\ \mathrm{mutually\ distinct}}\left(h_{k,ij}^2+\la_i\la_jh_{i,jk}h_{j,ik}\right)\\
=&2\sum_{i<j<k}\left(h_{i,jk}^2+h_{k,ij}^2+h_{j,ki}^2+\la_i\la_jh_{i,jk}h_{j,ki}+\la_i\la_kh_{i,kj}h_{k,ij}+\la_j\la_kh_{j,ki}h_{k,ji}\right).
\endaligned
\end{equation}
Without loss of generality, we assume $\la_1\ge\la_2\ge\cdots\ge\la_n\ge0$.
Let $$f(x,y,z)=x^2+y^2+z^2+\la_i\la_jxy+\la_j\la_kyz+\la_i\la_kxz$$
for every $x,y,z\in\R$ with mutually distinct $i,j,k$. Then
\begin{equation}\label{Hessf}
\mathbf{Hess}_f=
\left( \begin{array}{ccc}
2 & \la_i\la_j & \la_i\la_k \\
\la_i\la_j & 2 & \la_j\la_k \\
\la_i\la_k & \la_j\la_k & 2
\end{array} \right).
\end{equation}
From a direct computation, we have
\begin{equation}\aligned
\det\mathbf{Hess}_f=8+2\la_i^2\la_j^2\la_k^2-2\la_i^2\la_j^2-2\la_i^2\la_k^2-2\la_j^2\la_k^2.
\endaligned
\end{equation}
Then we consider a function
\begin{equation}\aligned
\phi(\mu_1,\mu_2,\mu_3)=4+\mu_1\mu_2\mu_3-\mu_1\mu_2-\mu_1\mu_3-\mu_2\mu_3
\endaligned
\end{equation}
on $V=\{(\mu_1,\mu_2,\mu_3)\in\R^3|\ 0\le\mu_3\le\mu_2\le\mu_1,\ \ \mu_1\mu_2\le2+\f{2}{\mu_1-1}\}$. From Lemma \ref{mu123}  in the Appendix I, $\phi\ge0$ on $V$,
which implies
\begin{equation}\aligned\label{detHessf}
\det\mathbf{Hess}_f\ge0
\endaligned
\end{equation}
combining $\la_1^2\la_i^2\le2+\la_i^2$ for all $i\ge2$.
We further conclude that all the eigenvalues of $\mathbf{Hess}_f$ are non-negative,
then it follows that
$$f\ge0\qquad \mathrm{on}\ \ \R^3.$$
From \eqref{ijkmutdistinct}, we have
\begin{equation}\aligned\label{ijkdiff}
\sum_{i,j,k\ \mathrm{mutually\ distinct}}\left(h_{k,ij}^2+\la_i\la_jh_{i,jk}h_{j,ik}\right)\ge0.
\endaligned
\end{equation}

Moreover, by the assumption for $i\ge2$ and $\la_1\ge1$ we have
\begin{equation}\aligned
2+\la_i^2-\la_i^2\la_1^2=2+\la_i^2\la_1^2(\la_1^{-2}-1)\ge2+\left(2+\f2{\la_1^2-1}\right)(\la_1^{-2}-1)=0.
\endaligned
\end{equation}
From the Cauchy inequality, we have
\begin{equation}\aligned\label{ineqjest}
(2+\la_i^2)h_{i,ij}^2+h_{j,ii}^2+2\la_i\la_jh_{i,ji}h_{j,ii}\ge0.
\endaligned
\end{equation}
Substituting \eqref{ijkdiff}\eqref{ineqjest} into \eqref{DeMlogdecom}, we have
\begin{equation}\aligned
\De_M\log v\ge&\sum_{\a>n,i,j}h^2_{\a,ij}+\sum_i(1+\la_i^2)h_{i,ii}^2.
\endaligned
\end{equation}
When the above inequality becomes an equality, we clearly have
$$\sum_{i\neq j}\left((2+\la_i^2)h_{i,ij}^2+h_{j,ii}^2+2\la_i\la_jh_{i,ji}h_{j,ii}\right)=0,$$ which completes the proof.
\end{proof}
As a corollary, we have the following result.
\begin{cor}\label{Lalogv}
Let $\la_1,\cdots,\la_n$ be the singular eigenvalues of $Du$ at any point of $\Om$.
If $\la_1\ge\cdots\ge\la_n\ge0$ and $\sup_{i\ge2}\la_1\la_i\le\La$ for some $0<\La\le\sqrt{2}$, then
\begin{equation}\aligned\label{logvLa}
\De_M\log v\ge\left(1-\f{\La}{\sqrt{2}}\right)|B_M|^2+\f1n|\n\log v|^2,
\endaligned
\end{equation}
where $B_M$ denotes the second fundamental form of $M$ in $\R^{n+m}$.
\end{cor}
\begin{proof}
From \eqref{Hessf} and Corollary \ref{mu123La} in the Appendix I, we have
\begin{equation}\aligned\label{ijkestLa}
\sum_{i,j,k\ \mathrm{mutually\ distinct}}\left(h_{k,ij}^2+\la_i\la_jh_{i,jk}h_{j,ik}\right)
\ge(2-\sqrt{2})(2-\La^2)\sum_{i,j,k\ \mathrm{mutually\ distinct}}h_{k,ij}^2.
\endaligned
\end{equation}
From the Cauchy inequality, we have
\begin{equation}\aligned\label{ineqjestLa}
\f{\La}{\sqrt{2}}\left(2h_{i,ij}^2+h_{j,ii}^2\right)+2\la_i\la_jh_{i,ji}h_{j,ii}\ge\f{\La}{\sqrt{2}}\left(2h_{i,ij}^2+h_{j,ii}^2\right)-2\La|h_{i,ji}h_{j,ii}|\ge0.
\endaligned
\end{equation}
Substituting \eqref{ijkestLa}\eqref{ineqjestLa} into \eqref{DeMlogdecom} gets
\begin{equation}\aligned
\De_M\log v=&\sum_{\a>n,i,j}h^2_{\a,ij}+\sum_i(1+\la_i^2)h_{i,ii}^2+\sum_{i\neq j}\left(1-\f{\La}{\sqrt{2}}\right)\left(2h_{i,ij}^2+h_{j,ii}^2\right)\\
&+\sum_{i\neq j}\la_i^2h_{i,ij}^2+(2-\sqrt{2})(2-\La^2)\sum_{i,j,k\ \mathrm{mutually\ distinct}}h_{k,ij}^2\\
\ge&\left(1-\f{\La}{\sqrt{2}}\right)\sum_{\a,i,j}h^2_{\a,ij}+\sum_{i,j}\la_i^2h_{i,ij}^2.
\endaligned
\end{equation}
Combining $|\n\log v|^2=\sum_j\left(\sum_{i}\la_ih_{i,ij}\right)^2$ and the Cauchy inequality, we complete the proof.
\end{proof}

For any considered point $p\in\Om$, up to rotations of $\R^n$, $\R^m$, we assume $D_ku^\a(p)=\de_{k,\a}\la_k$.
Let $e_i=\left(1+\sum_{\a=1}^m(D_iu^\a)^2\right)^{-1/2}\left(\mathbf{E}_i+\sum_{\a=1}^mD_iu^\a \mathbf{E}_{n+\a}\right)$ for $i=1,\cdots,n$, and
$\nu_\a=\left(1+|Du^\a|^2\right)^{-1/2}\left(\sum_{j=1}^nD_ju^\a \mathbf{E}_{j}+\mathbf{E}_{n+\a}\right)$ for $\a=1,\cdots,m$.
Then $\{e_i\}_{i=1}^n\cup\{\nu_\a\}_{\a=1}^m$ forms a local frame field in $\R^{n+m}$ (see also \eqref{nuka*}\eqref{ekj*}), which is orthonormal at $p$. Moreover, at $p$
\begin{equation}\aligned\label{haijua}
&h_{\a,ij}=\lan\bn_{e_j}e_i,\nu_\a\ran\\
=&\left(1+\sum_{\a=1}^m(D_iu^\a)^2\right)^{-1/2}\left(1+\sum_{\a=1}^m(D_ju^\a)^2\right)^{-1/2}\left(1+|Du^\a|^2\right)^{-1/2}\f{\p^2u^\a}{\p x_ix_j}\\
=&\f1{(1+\la_i^2)(1+\la_j^2)(1+\la_\a^2)}D_{ij}u^\a.
\endaligned
\end{equation}
\begin{lem}\label{Harm}
Suppose that the 2-dilation $|\La^2du|$ of $u$ satisfies $|\La^2du|^2\le2+\left|\f{2}{(\mathrm{Lip}\,u)^2-1}\right|$ and $v$ is a constant on $\Om$. Then $\De_{\R^n}u^\a=0$ on $\Om$ for each $\a=1,\cdots,m$.
\end{lem}
\begin{proof}
At any  point $p$ in $\Om$, from Lemma \ref{sqrt2logv} we get
\begin{equation}\aligned\label{pxilai20}
\sum_{\a>n,i,j}h^2_{\a,ij}+\sum_ih_{i,ii}^2=\sum_{i\neq j}\left((2+\la_i^2)h_{i,ij}^2+h_{j,ii}^2+2\la_i\la_jh_{i,ji}h_{j,ii}\right)=0,
\endaligned
\end{equation}
where $\la_1\ge\cdots\ge\la_n\ge0$ are the singular eigenvalues of $Du$ at $p$. From the assumption, we have
\begin{equation}\aligned\label{la1lai2ge0**}
2+\la_i^2-\la_i^2\la_1^2\ge0\qquad \mathrm{for\ each}\ i=2,\cdots,n.
\endaligned
\end{equation}
\begin{itemize}
  \item Case 1: $\la_1>\la_2$ at $p$.  For $j\ge2$ and $i\neq j$, we have $2+\la_i^2-\la_i^2\la_j^2>0$ from \eqref{la1lai2ge0**}.
Then from \eqref{haijua}\eqref{pxilai20}, we deduce that $u^\a_{ii}=0$ for all $\a=2,\cdots,n$ and $i=1,\cdots,n$. In particular, $\De_{\R^n}u^\a=0$ for all $\a\ge2$ from \eqref{logvsqrt2} and the constant $v$.
Suppose $\la_2=\cdots=\la_k>\la_{k+1}\ge\cdots\ge\la_n$ for some integer $k\ge2$.
From \eqref{pxilai20} again, there hold $u^1_{11}=0$ and $u^1_{ii}=0$ for $i\ge k+1$.
From \eqref{Nms}, we get $\sum_{i=2}^k\f1{1+\la_2^2}u^1_{ii}=0$, which implies $\De_{\R^n}u^1=0$.

  \item Case 2: $\la_1=\la_2=\cdots=\la_k>\la_{k+1}\ge\cdots\ge\la_n$ at $p$ for some integer $k\ge2$.
From \eqref{la1lai2ge0**}, $\la_1\le\sqrt{2}$, and $2+\la_i^2-\la_i^2\la_j^2>0$ for all $i\neq j$ with $\max\{i,j\}\ge k+1$.
Then from \eqref{pxilai20}, $u^\a_{ii}=0$ for $\max\{\a,i\}\ge k+1$ (similar to case 1).
In particular, $\De_{\R^n}u^\a=0$ for all $\a\ge k+1$.
Using \eqref{Nms}, we get $\sum_{i=1}^k\f1{1+\la_1^2}u^j_{ii}=0$ for $j=1,\cdots,k$. Therefore, $\De_{\R^n}u^j=0$ for $j=1,\cdots,k$.
\end{itemize}
This completes the proof.
\end{proof}

Now we consider minimal graphs with bounded slope.
\begin{lem}\label{Reggraphcone}
Let $M_k=\mathrm{graph}_{u_k}$ be a family of Lipschitz minimal graphs over $\R^n$ of codimension $m\ge1$ with $\sup_k\mathbf{Lip}\,u_k<\infty$ and $|\La^2du_k|^2\le\f{2(\mathrm{Lip}\,u_k)^2}{|(\mathrm{Lip}\,u_k)^2-1|}$ a.e. on $\R^n$.
Let $V$ be the limit of $|M_k|$ in the varifold sense.
If spt$V$ is a regular cone or there is a regular $l$-dimensional cone $C$ with spt$V=C\times\R^{n-l}$, then spt$V$ is flat.
\end{lem}
\begin{proof}
Denote $M=\mathrm{spt}V$, which is a Lipschitz minimal graph over $\R^n$ for some graphic function $u=(u^1,\cdots,u^m)$.
From Lemma \ref{Multi1} in the appendix II and Allard's regularity theorem, $u_k$ converges to $u$ in $C^3$-sense at any regular point of $u$. Hence, we have $|\La^2du|^2\le\f{2(\mathrm{Lip}\,u)^2}{|(\mathrm{Lip}\,u)^2-1|}$ a.e. on $\R^n$ from the assumption of $u_k$.

Now we assume $l\ge2$.
Since $M=C\times\R^{n-l}$, then $C$ lives in an $(m+l)$-dimensional Euclidean space.
Up to a rotation, $C$ can be represented as a graph of a 1-homogenous vector-valued function $\phi=(\phi^1,\cdots,\phi^m)$ on $\R^l$.
Then up to two rotations of $\R^n$ and $\R^m$, there are a constant matrix $(c^\a_{j})$ for $j=l+1,\cdots,n$ and $\a=1,\cdots,m$ such that
$$u^\a(x_1,\cdots,x_n)=\phi^\a(x_1,\cdots,x_l)+\sum_{j=l+1}^nc^\a_jx_j$$
for each $\a$.
After a rotation of $\R^{n-l}$, we can further assume
\begin{equation}\aligned\label{uaxka}
u^\a(x_1,\cdots,x_n)=\phi^\a(x_1,\cdots,x_l)+c_\a x_{l+\a}
\endaligned
\end{equation}
for each $\a=1,\cdots,m$, where we let $c_{n+j}=0$ for any positive integer $j$.

Let $g_{ij}=\de_{ij}+\sum_\a \p_iu^\a\p_ju^\a$.
From \eqref{uaxka}, $g_{ij}$ is a function of $x_1,\cdots,x_l$, and $v=\sqrt{\det g_{ij}}$ can be seen as a function of $x_1,\cdots,x_l$. From Lemma \ref{sqrt2logv}, we have
\begin{equation}\aligned
\De_M\log v\ge0
\endaligned
\end{equation}
on the regular part of $M$.
Note that $\log v$ is smooth on $\R^l\setminus\{0\}$. Since  $\log v$ is
0-homogenous, it achieves its maximum on $B^l_2\setminus B^l_{1/2}$ at a point in $\p B^l_1$.
From the strong maximum principle, $v$ is a constant.
With Lemma \ref{Harm}, $u^\a$ is harmonic on $\R^n\setminus(\{0^l\}\times\R^{n-l})$ for each $\a=1,\cdots,m$. Namely, $\phi^\a=\phi^\a(x_1,\cdots,x_l)$ is harmonic on $\R^l\setminus\{0\}$.
Note that $\phi^\a$ is 1-homogenous. Then $\phi^\a$ is harmonic on $\R^l$ for each $\a$, and it must be affine, i.e., $\phi^\a-\phi^\a(0)$ is linear.
From \eqref{uaxka}, it follows that each $u^\a$ is affine and then $M$ is flat.

For $l<2$, $M$ is regular, then $M$ is flat from the above argument.
This completes the proof.
\end{proof}
Let us prove a Bernstein theorem for minimal graphs with bounded slope. 
\begin{thm}\label{BDMGAFFINE}
Let $M=\mathrm{graph}_u$ be a Lipschitz minimal graph over $\R^n$ of codimension $m\ge2$.
If $|\La^2du|^2\le\f{2(\mathrm{Lip}\,u)^2}{|(\mathrm{Lip}\,u)^2-1|}$ a.e. on $\R^n$, then $M$ is flat.
\end{thm}
\begin{proof}
Assume that $M$ is not flat. From Lemma \ref{Multi1} in the appendix II, there is a sequence $r_k\rightarrow\infty$ such that $|\f1{r_k}M|$ converges to a minimal cone $C$ of multiplicity one in the varifold sense, where spt$C$ can be rewritten as a graph over $\R^n$ with a Lipschitz homogenous graphic function $u_\infty=(u^1_\infty,\cdots,u^m_\infty)$ satisfying
$$\mathbf{Lip}\,u_\infty\le \mathbf{Lip}\,u\le L$$
for some constant $L>0$. Then from Allard's regularity theorem, when $W$ is a bounded open set with $\overline{W}\cap \mathrm{spt}C$ belonging to the regular part of $C$, $W\cap\f1{r_k}M$ converges to $W\cap \mathrm{spt}C$ smoothly. Hence we have
$$|\La^2du_\infty|^2\le\f{2(\mathrm{Lip}\,u_\infty)^2}{|(\mathrm{Lip}\,u_\infty)^2-1|}\qquad a.e. \ \mathrm{on}\ \R^n.$$

If spt$C$ is a regular cone, then spt$C$ is flat from Lemma \ref{Reggraphcone}. Now we assume that there is a singular point $q$ in spt$C\setminus\{0^{n+m}\}$.
We blow  $C$ up at the point $q$, and get a nonflat minimal cone $C'$ whose support
splits off $\R$ isometrically.
If $u'_\infty$ denotes the graphic function of spt$C'$, then we have
$$\mathbf{Lip}\,u'_\infty\le \mathbf{Lip}\,u_\infty\le \mathbf{Lip}\,u\le L,$$
and $|\La^2du'_\infty|^2\le\f{2(\mathrm{Lip}\,u'_\infty)^2}{|(\mathrm{Lip}\,u'_\infty)^2-1|}$ a.e. on $\R^n$.
By dimensional reduction argument, we get a sequence of minimal graphs $M_k$ (which are scalings and rigid motions of $M$) such that $|M_k|$ converges to a nonflat minimal cone $C^*$, where spt$C^*=C_*\times\R^l$ for a regular $(n-l)$-cone $C_*$.
Moreover, the graphic function $u^*_\infty$ of spt$C^*$ satisfies
$$\mathbf{Lip}\,u^*_\infty\le \mathbf{Lip}\,u_\infty\le \mathbf{Lip}\,u\le L,$$
and $|\La^2du^*_\infty|^2\le\f{2(\mathrm{Lip}\,u^*_\infty)^2}{|(\mathrm{Lip}\,u^*_\infty)^2-1|}$ a.e. on $\R^n$.
From Lemma \ref{Reggraphcone}, we get the flatness of spt$C^*$, which is a contradiction. Hence,
spt$C$ is flat. Since $C$ has multiplicity one, then Allard's regularity theorem implies that $M$ is flat.
\end{proof}
\begin{rem}
Lawson-Osserman \cite{l-o} constructed the minimal Hopf cones in $\R^{2m}\times\R^{m+1}$ for $m=2,4,8$.
For $m=2$, the Hopf cone has the Lipschitz graphic function $w=(w^1,w^2,w^3)$ over $\R^4$ given by
$$w=\f{\sqrt{5}}2|x|\e\left(\f{x}{|x|}\right),$$
where $\e=(|z_1|^2-|z_2|^2,2z_1\bar{z_2})$ is the Hopf map from $S^3$ to
$S^2$. It is clear that $\mathrm{Lip}\,w\equiv\sqrt{5}$, and
$|\La^2dw|\equiv5$
on $\R^n$ (see also the appendix in \cite{j-x-y}).
\end{rem}

By a contradiction argument, we have the following curvature estimate.
\begin{thm}\label{CUVest}
For every constant $L$, there exists a constant $\Th_{n,m,L}>0$ depending on $n,m,L$ such that if $M=\mathrm{graph}_u$ is a Lipschitz minimal graph over $B_2$ of codimension $m\ge2$ with $\mathbf{Lip}\,u\le L$ and $|\La^2du|^2\le\f{2(\mathrm{Lip}\,u)^2}{|(\mathrm{Lip}\,u)^2-1|}$ a.e. on $B_2$, then
\begin{equation}\aligned\label{D2uThnmL}
\sup_{B_1}|D^2u|\le\Th_{n,m,L}.
\endaligned
\end{equation}
In particular, $M$ is smooth.
\end{thm}
\begin{proof}
Let us prove \eqref{D2uThnmL} by contradiction.
Suppose that there is a sequence of Lipschitz minimal graphs $\mathrm{graph}_{u_k}$ over $B_2$ of codimension $m\ge2$ with $\mathbf{Lip}\,u_k\le L$ and $|\La^2du_k|^2\le\f{2(\mathrm{Lip}\,u_k)^2}{|(\mathrm{Lip}\,u_k)^2-1|}$ such that
\begin{equation}\aligned\label{D2ukinfty}
\lim_{k\rightarrow\infty}\sup_{B_{3/2}}\left(\f32-|x|\right)|D^2u_k|(x)=\infty.
\endaligned
\end{equation}
If $u$ is not smooth at a point $q\in B_2$, then we blow $M$ up at $(q,u(q))$, and get a contradiction from Theorem \ref{BDMGAFFINE}.
Hence $M$ is smooth.

From \eqref{D2ukinfty}, there exists a sequence of points $p_k\in B_{\f32}$ such that
$$r_k\triangleq\left(\f32-|p_k|\right)|D^2u_{k}(p_k)|=\sup_{B_{3/2}}\left(\f32-|x|\right)|D^2u_{k}(x)| \rightarrow\infty$$
as $k\rightarrow\infty$.
Put $\tau_k=\f32-|p_k|$, and $R_k=2r_k/\tau_k$. Let
$$w_k(x)=R_ku_{k}\left(\f{x}{R_k}+p_k\right)\qquad \mathrm{on}\ \ B_{r_k}.$$
Then
\begin{equation}\aligned
D^2w_k(x)=\f1{R_k}D^2u_{k}\left(\f{x}{R_k}+p_k\right).
\endaligned
\end{equation}
Note that $\f{\tau_k}2\le\f32-|x|$ for all $x\in B_{\f{\tau_k}2}(p_k)$.
We have
\begin{equation}\aligned\label{Bsqrtriwi}
\sup_{B_{r_k}}\left|D^2w_k\right|=\f1{R_k}\sup_{B_{\f{\tau_k}2}(p_k)}\left|D^2u_{k}\right|
\le\f2{R_k\tau_k}\sup_{B_{\f{\tau_k}2}(p_k)}\left(\f32-|x|\right)\left|D^2u_{k}\right|(x)\le1,
\endaligned
\end{equation}
and
\begin{equation}\aligned\label{D2wi0}
\left|D^2w_k\right|(0)=\f1{R_k}\left|D^2u_{k}\right|(p_k)
=\f1{R_k}\f{1}{\tau_k}\sup_{B_{3/2}}\left(\f32-|x|\right)\left|D^2u_{k}\right|(x)=\f12.
\endaligned
\end{equation}
Moreover, from $\mathbf{Lip}\,u_{k}\le L$, $|\La^2du_{k}|^2\le\f{2(\mathrm{Lip}\,u_{k})^2}{|(\mathrm{Lip}\,u_{k})^2-1|}$ on $B_{2}$,
we have $\mathbf{Lip}\,w_k\le L$ and $|\La^2dw_k|^2\le\f{2(\mathrm{Lip}\,w_{k})^2}{|(\mathrm{Lip}\,w_{k})^2-1|}$ on $B_{r_k}$.
With \eqref{Bsqrtriwi}, after choosing a subsequence, graph$_{w_k}$ converges smoothly to a minimal graph with the graphic function $w_*$ such that
$\mathrm{Lip}\,w_*\le L$, $|\La^2dw_*|^2\le\f{2(\mathrm{Lip}\,w_{*})^2}{|(\mathrm{Lip}\,w_{*})^2-1|}$ a.e. and $\left|D^2w_*\right|\le1$ on $\R^n$.
Moreover, \eqref{D2wi0} implies $\left|D^2w_*\right|(0)=\f12$. However, this contradicts to Theorem \ref{BDMGAFFINE}. We complete the proof.
\end{proof}

\Section{Quasi-cylindrical minimal cones from minimal graphs}{Quasi-cylindrical minimal cones from minimal graphs}

Let $\{\mathbf{E}_i\}_{i=1}^{n+m}$ denote the standard basis of $\R^{n+m}$ such that $\mathbf{E}_i$ corresponds to the axis $x_i$.
\begin{lem}\label{codim3}
Let $M_k\in\mathcal{M}_{n,m,\La,B_2}$ for some constant $0<\La<\sqrt{2}$.
Assume that $[|M_k|]\llcorner\mathbf{B}_2$ converges to a stationary varifold $V$ in the varifold sense.
If $\mathcal{S}$ denotes the singular set of $\mathrm{spt}V\cap \mathbf{B}_1$, then $\mathcal{S}$ has Hausdorff dimension $\le n-3$.
\end{lem}
\begin{proof}
We assume that $\mathcal{S}$ has Hausdorff dimension $>n-3$.
Then there is a constant $\be>n-3$ so that $\be$-dimensional Hausdorff measure of $\mathcal{S}$ satisfies $\mathcal{H}^\be(\mathcal{S})>0$.
From Lemma 11.2 in \cite{gi}, $\mathcal{H}^\be(\mathcal{S})>0$ implies $\mathcal{H}^\be_\infty(\mathcal{S})>0$ (see \eqref{HbeinftyW} for its definition with $\R^{n+1}$ there replaced by $\R^{n+m}$).
From the argument of Proposition 11.3 in \cite{gi}, there are a point $q\in \mathcal{S}$ and a sequence $r_k\rightarrow0$ (as $k\to\infty$) such that
\begin{equation}\aligned
\mathcal{H}^\be_\infty\left(\mathcal{S}\cap \mathbf{B}_{r_k}(q)\right)>2^{-\be-1}\omega_\be r_k^\be.
\endaligned
\end{equation}
Let $\mathcal{S}_k=\f1{r_k}\left(\mathcal{S}\cap \mathbf{B}_{r_k}(q)\right)$. Then
\begin{equation}\aligned
\mathcal{H}^\be_\infty\left(\mathcal{S}_k\cap \mathbf{B}_{1}(0)\right)>2^{-\be-1}\omega_\be.
\endaligned
\end{equation}
Without loss of generality, we assume that $\f1{r_k}(V,q)$ converges to a tangent cone $(V_*,0^{n+m})$ in $\R^{n+m}$ in the varifold sense as $k\rightarrow\infty$. By the definition of $V$, there is a sequence of minimal graphs $M'_k$ in $\R^{n+m}$ (rigid motions of $M_k$) such that $|M'_k|$ converges to the minimal cone $V_*$ in the varifold sense as $k\rightarrow\infty$.

Let $\mathcal{S}_*$ be the singular set of $V_*$.
If $y_k\in \mathcal{S}_k$ and $y_k\rightarrow y_*\in V_*$, then
it's clear that $y_*$ is a singular point of $V_*$ by Allard's regularity theorem and multiplicity one of $V_*$, which implies $\limsup_{k\rightarrow\infty}\mathcal{S}_k\subset\mathcal{S}_*$. Analog to the proof of Lemma 11.5 in \cite{gi}, we have
\begin{equation}\aligned
\mathcal{H}^\be_\infty\left(\mathcal{S}_{*}\cap \mathbf{B}_{1}(0)\right)\ge2^{-\be-1}\omega_\be.
\endaligned
\end{equation}
Let us continue the above procedure.
By dimension reduction argument, there is a $l(l\le2)$-dimensional non-flat regular minimal cone $C\subset\R^{m+k}$ such that
there is a sequence of minimal graphs $\Si_k\in\mathcal{M}_{n,m,\La,B_{R_k}}$ (which are scalings and translations gotten from $M_k$) with $R_k\rightarrow\infty$ so that $\Si_k$ converges to a minimal cone $C_*$ in the varifold sense, which is a trivial product of $C$ and $\R^{n-l}$.

From Lemma \ref{cross} in the Appendix III, $l=1$ is impossible.
For $l=2$, $\mathrm{spt}C\cap\p B_1(0^{m+2})$ is smooth minimal in $\p B_1(0^{m+2})$, hence it is a disjoint union of geodesic circles in a sphere. So $\mathrm{spt}C_*\cap\R^{n+m}$ is a union of $n$-planes $P_1,\cdots,P_{j_0}$ with $j_0\ge2$ as $\mathrm{spt}C$ is non-flat. If there is a unit vector $\xi\in \mathrm{spt}C_*$ with $\lan\xi,\mathbf{E}_i\ran=0$ for each $i=1,\cdots,n$, then Theorem \ref{Stable} contradicts to that $\mathrm{spt}C$ splits off $\R^{n-2}$ isometrically. Therefore, $\mathrm{spt}C_*$ can be written as a graph over $\R^n$, and then $\mathrm{spt}C_*$ is an $n$-plane. It is a contradiction.
This completes the proof.
\end{proof}
\textbf{Remark}. The above dimension estimate is not sharp.
We will establish a sharp one through the  Bernstein theorem for minimal graphs in Lemma \ref{codim7}.

With Lemma \ref{codim3}, we can use De Giorgi-Nash-Moser iteration for nonnegative superharmonic functions on the regular part of stationary indecomposable currents as follows.
\begin{pro}\label{MeanVSuph}
Let $\mathbf{T}$ be an integral current in $\overline{\mathcal{M}}_{n,m,\La,B_{3}}$ for some constant $0<\La<\sqrt{2}$, and
$T$ be a stationary indecomposable component of $\mathbf{T}\llcorner\mathbf{B}_{2}$ with $0\in\mathrm{spt}T\triangleq M$.
Then there exist a constant $\de_*>0$ depending only on $n$, a constant $\Th_{T}>0$ depending on $n,m,\La,T$ such that
if $f$ is a nonnegative smooth bounded function on the regular part of $T$ satisfying $\De_{M} f\le0$ on the regular part of $T$,
then
\begin{equation}\aligned\label{MVf}
r^{-n}\int_{M\cap \mathbf{B}_r}f^{\de_*}\le \Th_T f^{\de_*}(0)\qquad \mathrm{for\ any}\ 0<r<1,
\endaligned
\end{equation}
where $\De_M$ is the Laplacian of $M$.
\end{pro}
\begin{proof}
For any small $\ep>0$, let $\mathcal{S}$ denote the singular set of $\mathbf{T}\llcorner\mathbf{B}_{2-\ep}$. From Lemma \ref{codim3}, there is a collection of balls $\{\mathbf{B}_{r_k}(\mathbf{x}_k)\}_{k=1}^{N_\ep}$ with $r_k<\ep/2$ such that $\mathcal{S}\subset\cup_{k=1}^{N_\ep}\mathbf{B}_{r_k}(\mathbf{x}_k)$, and
\begin{equation}\aligned\label{rkn-2ep}
\sum_{k=1}^{N_\ep}r_k^{n-2}<\ep.
\endaligned
\end{equation}
By Besicovitch covering lemma, we can further require that there is a constant $c_{n+m}>0$ depending only on $n+m$ so that
\begin{equation}\aligned\label{chijneqk}
\int_E\chi_{_{\mathbf{B}_{2r_j}(\mathbf{x}_j)}}\sum_{k\neq j}\chi_{_{\mathbf{B}_{2r_k}(\mathbf{x}_k)}}\le c_{n+m}\mathcal{H}^n(E\cap\mathbf{B}_{2r_j}(\mathbf{x}_j))
\endaligned
\end{equation}
for each $j=1,\cdots,N_\ep$ and each $n$-rectifiable set $E\subset\R^{n+m}$.
Let $\e_k$ be a $C^2$ function on $\mathbf{B}_{2}$ with $0\le\e_k\le1$ such that $\e_k=0$ on $\mathbf{B}_{r_k}(\mathbf{x}_k)$, $\e_k=1$ on $\mathbf{B}_{2}\setminus\mathbf{B}_{2r_k}(\mathbf{x}_k)$ and  
$$r_k|\bn \e_k|+r_k^2|\la(\bn^2\e_k)|\le c\qquad \mathrm{on}\ \mathbf{B}_{2r_k}(\mathbf{x}_k)\setminus\mathbf{B}_{r_k}(\mathbf{x}_k),$$
where $\bn$ denotes the Levi-Civita connection of $\R^{n+m}$, $|\la(\bn^2\e_k)|$ denotes the maximum of the absolution of eigenvalues of the Hessian of $\e_k$ on $\R^{n+m}$, $c$ is an absolute positive constant.
Let $e_1,\cdots,e_n$ be a local orthonormal tangent frame field of $M$. Since $M$ is minimal, we have (see (5.4) in \cite{d-j-x0} for instance)
\begin{equation}\aligned\label{DeMek}
|\De_M\e_k|=\left|\sum_{i=1}^{n}\mathbf{Hess}_{\e_k}(e_i,e_i)\right|\le ncr_k^{-2}\chi_{_{\mathbf{B}_{2r_k}(\mathbf{x}_k)}},
\endaligned
\end{equation}
where $\mathbf{Hess}_{\e_k}$ denotes the Hessian of $\e_k$ on $\R^{n+m}$.
Set $\e_\ep=\prod_{k=1}^{N_\ep}\e_k\in C^2$. Then $\e_\ep=0$ on a neighborhood of $\mathcal{S}$. Let $\n_M$ be the Levi-Civita connection of $M$, and $\mathrm{div}_M$ be the divergence of $M$.
From \eqref{DeMek} and the Cauchy inequality, one has
\begin{equation}\aligned\label{DeMeep}
&|\De_M\e_\ep|\le\sum_{k=1}^{N_\ep}|\De_M\e_k|+\sum_{j\neq k}|\n_M \e_j||\n_M \e_k|\\
\le& nc\sum_{k=1}^{N_\ep}r_k^{-2}\chi_{_{\mathbf{B}_{2r_k}(\mathbf{x}_k)}}+c^2\sum_{j\neq k}r_j^{-1}r_k^{-1}\chi_{_{\mathbf{B}_{2r_j}(\mathbf{x}_j)}}\chi_{_{\mathbf{B}_{2r_k}(\mathbf{x}_k)}}\\
\le& nc\sum_{k=1}^{N_\ep}r_k^{-2}\chi_{_{\mathbf{B}_{2r_k}(\mathbf{x}_k)}}+c^2\sum_{j\neq k}r_j^{-2}\chi_{_{\mathbf{B}_{2r_j}(\mathbf{x}_j)}}\chi_{_{\mathbf{B}_{2r_k}(\mathbf{x}_k)}}.
\endaligned
\end{equation}
Then with \eqref{chijneqk} we deduce
\begin{equation}\aligned
\int_{\mathrm{spt}T}|\De_M\e_\ep|\le(nc+c^2c_{n+m})\sum_{k=1}^{N_\ep}r_k^{-2}\mathcal{H}^n\left(M\cap\mathbf{B}_{2r_k}(\mathbf{x}_k)\right).
\endaligned
\end{equation}
Similarly,
\begin{equation}\aligned
\int_{M}|\n_M \e_\ep|^2\le&\int_{M}\left(\sum_{k=1}^{N_\ep}|\n_M \e_k|^2+\sum_{j\neq k}|\n_M \e_j||\n_M \e_k|\right)\\
\le&(1+c_{n+m})c^2\sum_{k=1}^{N_\ep}r_k^{-2}\mathcal{H}^n\left(M\cap\mathbf{B}_{2r_k}(\mathbf{x}_k)\right).
\endaligned
\end{equation}
Combining Lemma \ref{VGM1} and \eqref{rkn-2ep}, we deduce
\begin{equation}\aligned\label{MB3neep}
\lim_{\ep\to0}\int_{M}\left(|\De_M\e_\ep|+|\n_M \e_\ep|^2\right)=0.
\endaligned
\end{equation}
Let $\varphi$ be a nonnegative Lipschitz function with compact support in $\mathbf{B}_2$. Then the support of $\varphi$ is in $\mathbf{B}_{2-\ep}$ for small $\ep>0$. With integrating by parts, we have
\begin{equation}\aligned
0\le&-\int_{M}\varphi\e_\ep\De_M f=\int_{M}\lan\n_M(\varphi\e_\ep),\n_M f\ran\\
=&\int_{M}\e_\ep\lan\n_M\varphi,\n_M f\ran-\int_{M}f\mathrm{div}_M\left(\varphi\n_M\e_\ep\right) \\
=&\int_{M}\e_\ep\lan\n_M\varphi,\n_M f\ran-\int_{M}f\lan\n_M\varphi,\n_M\e_\ep\ran-\int_{M}f\varphi\De_M\e_\ep.
\endaligned
\end{equation}
Combining \eqref{MB3neep} and Cauchy inequality, we conclude that
\begin{equation}\aligned
0\le&\int_{M}\lan\n_M\varphi,\n_M f\ran,
\endaligned
\end{equation}
which means that $f$ is superharmonic on $M$ in the distribution sense.
Now we can follow the argument of De Giorgi-Nash-Moser iteration (see Theorem 3.2 in \cite{d3} for instance) and finish the proof.
Note that the constant $\de_*$ in \eqref{MVf} is obtained from the dimension
$n$ and the exponent of the Sobolev inequality, which implies that $\de_*$ only depends on $n$.
\end{proof}

Let $M_k=\mathrm{graph}_{u_k}\in\mathcal{M}_{n,m,\La,B_{R_k}}$ with $R_k\rightarrow\infty$ for some constant $0<\La<\sqrt{2}$.
Let
$g_{ij}^k=\de_{ij}+\sum_{\a=1}^m \p_i u^\a_k\p_j u^\a_k,$
and $v_k=\sqrt{\det g_{ij}^k}$.
We suppose that $M_k$ converges in the varifold sense to a minimal cone $C$ in $\R^{n+m}$ with the vertex at the origin.
Since the multiplicity function of spt$C$ is one from Theorem \ref{multi1}, we denote spt$C$ by $C$ for simplicity.
Without loss of generality, we assume that $[|M_k|]$ converges to an integral current $T$ in $\R^{n+m}$.
From Lemma \ref{sptTV}, $\p T=0$ and $T$ has multiplicity one on spt$T=\mathrm{spt}C$.
As before (in \S 4), $\pi_*$ denotes the projection from $\R^{n+m}$ into $\R^{n}$ by
$$\pi_*(x_1,\cdots,x_{n+m})=(x_1,\cdots,x_n),$$
$\mathbf{C}_r(\mathbf{x})=B_r(\pi_*(\mathbf{x}))\times B_r(x_{n+1},\cdots,x_{n+m})$ denotes the cylinder in $\R^{n+m}$ for any $\mathbf{x}=(x_1,\cdots,x_{n+m})\in\R^{n+m}$, and $\mathbf{C}_r=\mathbf{C}_r(0^{n+m})$.
\begin{lem}\label{almostver}
Suppose that for any regular point $q\in C$ there is a sequence of points $M_k\ni q_k\to q$ with $\lim_{k\rightarrow\infty}v_k(q_k)=\infty$.
Then $C$ is a quasi-cylinder in $\R^{n+m}$ with countably $(n-1)$-rectifiable $\pi_*(C)$, and $C\cap (B_r(x)\times\{0^m\})$ is isometric to
$\pi_*(C\cap (B_r(x)\times\p B_{s}(0^m)))$ for any $r,s>0$, $x\in\R^n$.
\end{lem}
\begin{proof}
For any $x\in\pi_*(C)$, there is a point $\mathbf{x}\in C$ such that $\pi_*(\mathbf{x})=x$. Then $t\mathbf{x}\in C$ implies $tx=\pi_*(t\mathbf{x})\in\pi_*(C)$. This means that $\pi_*(C)$ is a cone.
It is easy to check that $\pi_*(C)$ is closed in $\R^n$.
For any point $q\in \mathrm{reg}C$, from the assumption there is a unit vector $\e_q\in T_qC$ such that $\lan\mathbf{E}_i,\e_q\ran=0$ for each integer $i=1,\cdots,n$.
There are a small constant $r_q>0$ and a local orthonormal tangent frame $\{e_i\}_{i=1}^n$ on $\mathbf{C}_{r_q}(q)\cap C$ such that $\lan e_1(z),\mathbf{E}_i\ran=0$ for any $z\in\mathbf{C}_{r_q}(q)\cap C$ and $i=1,\cdots,n$.
In other words, $e_1$ is a $C^1$ tangent vector field on $\mathbf{C}_{r_q}(q)\cap C$ with $\pi_*(e_1(z))=0$ for any $z\in\mathbf{C}_{r_q}(q)\cap C$.
After choosing the  constant $r_q>0$ suitably small, for each $y\in \mathbf{C}_{r_q}(q)\cap \G_q$ there is an integral curve $\g_y$ in $\mathbf{C}_{r_q}(q)\cap C$ with $\dot{\g}_y=e_1\circ\g_y$.

We write $q=(\pi_*(q),q')\in\R^n\times\R^m$, and denote $\G_{z'}=C\cap(\R^n\times\{z'\})$ for each $z'\in\R^m$.
For any $z\in \mathbf{C}_{r_q}(q)\cap C$, let $\g_z$ be the integral curve with $\g_z(0)=z$. For any vector $\xi$ spanned by $\mathbf{E}_1,\cdots,\mathbf{E}_n$, we have
\begin{equation}\aligned
\lan \g_{z}(t)-\g_{z}(0),\xi\ran=\int_0^{t}\lan \dot{\g}_{z}(s),\xi\ran ds=\int_0^{t}\lan e_1(\g_z(s)),\xi\ran ds=0,
\endaligned
\end{equation}
which implies $\pi_*(\g_z(t))=\pi_*(z)$. Hence, we conclude that
$\pi_*(\mathbf{C}_{r_q}(q)\cap C)=\pi_*(\mathbf{C}_{r_q}(q)\cap \G_{z'})$ for any $|z'-q'|<r_q$.
With the dimensional estimates from Lemma \ref{codim3}, we deduce
\begin{equation}\aligned
\pi_*(C\cap (B_r(x)\times\p B_{r_1}(0^m)))=\pi_*(C\cap (B_r(x)\times\p B_{r_2}(0^m)))
\endaligned
\end{equation}
for any $r_1,r_2>0$ and $B_r(x)\subset B_1$. Letting $r_1\to0$ implies
\begin{equation}\aligned\label{pi*CB1}
\pi_*(C\cap (B_r(x)\times\{0^m\}))=\pi_*(C\cap (B_r(x)\times\p B_{r_2}(0^m))),
\endaligned
\end{equation}
which is isometric to $C\cap (B_r(x)\times\{0^m\})$.

From the slicing lemma for $T$ in \cite{s}, $\pi_*(C)=\pi_*(\mathrm{spt}T)$ is a countably $(n-1)$-rectifiable cone in $\R^n$.
With Lemma \ref{codim3}, for almost all $x\in\pi_*(C)$, there is a countably 1-rectifiable normalized curve $\g_x:\,\R\rightarrow\R^m$ such that
$\{(x,y)\in\R^n\times\R^m|\, x\in\pi_*(C),y\in\g_x\}$ is the support of $C$ up to a zero $\mathcal{H}^n$-set.
\end{proof}

With Proposition \ref{Tsuperhar}, we can weaken the condition in Lemma \ref{almostver} and get the same conclusion.
\begin{thm}\label{mainCONE}
If there is a sequence of points $q_k\in M_k$ with $\limsup|q_k|<\infty$ such that $\lim_{k\rightarrow\infty}v_k(q_k)=\infty$,
then $C$ is a multiplicity one quasi-cylindrical minimal cone.
\end{thm}
\begin{proof}
Let $\De_{M_k}$ and $\n_{M_k}$ denote the Laplacian and the Levi-Civita connection of $M_k$, respectively.
From Corollary \ref{Lalogv},
\begin{equation}\aligned\label{Mkvk-1f1n}
\De_{M_k} v_k^{-1/n}=\De_{M_k}e^{-\f1n\log v_k}=-\f1nv_k^{-1/n}\De_{M_k}\log v_k+\f1{n^2}v_k^{-1/n}|\n_{M_k}\log v_k|^2\le0
\endaligned
\end{equation}
on $M_k$.
Let $\xi_k$ denote an orientation of $M_k$, namely, $T_{x}M_k$ can be represented by the unit $n$-vector $\xi_k(x)$, such that $v_k^{-1}=\lan\xi_k,\mathbf{E}_1\wedge\cdots\wedge \mathbf{E}_n\ran$.
For any regular point $z\in C$, there is a constant $r_z>0$ such that $\mathbf{B}_{2r_z}(z)\cap C$ is smooth.
From Allard's regularity theorem, $M_k\cap \textbf{B}_{3r_z/2}(z)$ converges to $C\cap \textbf{B}_{3r_z/2}(z)$ smoothly.
Recalling $[|M_k|]\rightharpoonup T$. Let $\xi$ denote the orientation of $T$, and $v_T^{-1}=\lan\xi,\mathbf{E}_1\wedge\cdots\wedge \mathbf{E}_n\ran$ on reg$T$ (the regular part of $T$). We extend $v_T$ to sing$T$ (the singular part of $T$) by letting
\begin{equation}\aligned\label{DEFvT-1}
v_T^{-1}(x)=\inf_{\mathrm{reg}T\ni y\to x}\lan\xi(y),\mathbf{E}_1\wedge\cdots\wedge \mathbf{E}_n\ran\qquad \mathrm{for\ any}\ x\in\mathrm{sing}T.
\endaligned
\end{equation}
Then from \eqref{Mkvk-1f1n} we have
\begin{equation}\aligned\label{TvT-1f1n}
\De_{T} v_T^{-1/n}\le0\qquad \mathrm{on}\ \mathrm{reg}T,
\endaligned
\end{equation}
where $\De_{T}$ denotes the Laplacian of $T$.

From monotonicity of the density and Lemma \ref{VGM*}, there is a constant $c_n\ge\omega_n$ depending only on $n$
such that
\begin{equation}
\omega_nr^n\le\mathbb{M}(T\llcorner\mathbf{B}_r(\mathbf{x}))\le c_nr^n
\end{equation}
for any $r>0$ and $\mathbf{x}\in\mathrm{spt}T$.
From the proof of Theorem \ref{Lioua1},
we assume that there exist indecomposable multiplicity one currents $T_1,\cdots,T_l\in\mathcal{D}_n(\R^{n+m})$ for $l\ge1$, $T_1,\cdots,T_l\neq0$ such that
\begin{equation}\aligned\label{MTWWpm}
\mathbb{M}(T\llcorner W)=\sum_{j=1}^l\mathbb{M}(T_j\llcorner W),\qquad \sum_{j=1}^l\mathbb{M}(\p T_j\llcorner W)=0
\endaligned
\end{equation}
for any open $W\subset\subset\R^{n+m}$. From Proposition \ref{MeanVSuph} and \eqref{TvT-1f1n},
there are a constant $\de_*$ depending only on $n$, and a constant $\Th_T$ depending only on $n,m,\La,T$ such that for any $j$, $q\in\mathrm{spt}T_j$, $r>0$
\begin{equation}\aligned\label{meanvT}
r^{-n}\int_{\mathrm{spt}T_j\cap \textbf{B}_{r}(q)}v_T^{-\de_*/n}\le \Th_Tv_T^{-\de_*/n}(q).
\endaligned
\end{equation}

There is a point $q\in\mathrm{spt}T_j$ for some integer $j\in\{1,\cdots,l\}$ so that $q_k\to q$ up to a choice of a subsequence.
From $\lim_{k\rightarrow\infty}v_k(q_k)=\infty$, it follows that $v_T^{-1}(q)=0$ if $q$ is a regular point of $C$.
Now we assume that $q$ is a singular point of $C$ and $v_T^{-1}(q)>0$. Then we blow $C$ up at $q$, and get a contradiction from \eqref{DEFvT-1} and Theorem \ref{BDMGAFFINE}.
Hence we always have $v_T^{-1}(q)=0$.
From \eqref{meanvT}, we conclude that $v_T^{-1}=0$ on spt$T_j$.
Similarly, if there is an integer $j'\in\{1,\cdots,l\}$ such that $\sup_{\mathrm{spt}T_{j'}}v_T^{-1}>0$, then \eqref{meanvT} implies $v_T^{-1}>0$ on $\mathrm{spt}T_{j'}$.
By a blowing up argument and Theorem \ref{BDMGAFFINE}, we conclude that $\mathrm{spt}T_{j'}$ is an $n$-plane, which implies $\mathrm{spt}T_{j'}=\mathrm{spt}T$.
Therefore, such $j'$ does not exist, and we get
\begin{equation}\aligned
v_T^{-1}=\lan\xi,\mathbf{E}_1\wedge\cdots\wedge \mathbf{E}_n\ran=0\qquad \mathrm{on}\ \mathrm{spt}T.
\endaligned
\end{equation}
From Allard's regularity theorem, any sequence of points $y_k\in M_k$ converging to a regular point $y$ of $\mathrm{spt}T$ satisfies $\lim_{k\rightarrow\infty}v_k(y_k)=\infty$.
Recalling Lemma \ref{almostver}, we complete the proof.
\end{proof}

In Lemma 6.1.1 (p. 42) in \cite{ss}, J. Simons proved the following well-known result.
\begin{lem}\label{JSimons}
Let $\Si$ be a closed co-dimension 1 minimal variety in $\mathbb{S}^n$. Suppose $\Si$ is not the totally geodesic $\mathbb{S}^{n-1}$. Then if $n\le 6$, the cone $C\Si$ is not stable.
\end{lem}
Using Lemma \ref{JSimons}, Simons proved the celebrated Bernstein theorem in Theorem 6.2.2 in \cite{ss} with the help of Fleming's and De Giorgi's arguments. In high codimensions, we have the following Bernstein theorem based on Simons' work.
\begin{lem}\label{Cflat}
If $n\le7$, then spt$C$ is an $n$-plane.
\end{lem}
\begin{proof}
If $C$ is an entire graph over $\R^n$, then $C$ is an $n$-plane from Theorem \ref{BDMGAFFINE}.
Hence we can assume that there is a point $y_*=(0^n,y^*)\in C$ with $0\neq y^*\in\R^{m}$.
Without loss of generality, we assume $y^*=(1,0,\cdots,0)$.
From Lemma \ref{Cy*vmc} and Lemma \ref{multi1}, $C_t=C-ty_*$ converges as $t\rightarrow\infty$ to a minimal cone
$$\{(x_1,\cdots,x_n,y_1,\cdots,y_m)\in\R^n\times\R^m|\ (x_1,\cdots,x_n)\in C_{y_*},\ y_2=\cdots=y_m=0\},$$
where $C_{y_*}$ is a minimal cone in $\R^n$
with multiplicity one. From Theorem \ref{Stable}, $C_{y_*}$ is stable with the dimension $n-1\le 6$. With Lemma \ref{codim3} and dimension reduction argument, we get the flatness of $C_{y_*}$ by Lemma \ref{JSimons}.
In other words, $y_*$ is a regular point of $C$. Denote $C=\mathrm{spt}C$.
Hence, there is a constant $r>0$ such that $M_k\cap B_{2r}(y_*)$ converges smoothly to $C\cap B_{2r}(y_*)$ as $k\rightarrow\infty$.
Let $\xi$ denote the orientation of reg$C$.
From the argument in the proof of Theorem \ref{mainCONE}, we get
\begin{equation}\aligned\label{vwinfty}
\lan\xi,\mathbf{E}_1\wedge\cdots\wedge \mathbf{E}_n\ran=0\qquad \mathrm{on}\ \mathrm{reg}C.
\endaligned
\end{equation}
There is a ball $\mathbf{B}_{\de_y}(y_*)$ with the radius $\de_y>0$ such that $\mathbf{B}_{\de_y}(y_*)\cap C$ is regular everywhere.
From Lemma \ref{almostver}, $\pi_*(\mathbf{B}_{\de}(y_*)\cap C)$ contains a neighborhood of the origin in $\pi_*(C)$ for any $\de\in(0,\de_y]$, and the origin is the regular point of $C$ in particular. Hence $C$ is flat. We complete the proof.
\end{proof}
Analogously to the argument in Lemma \ref{Cflat}, we have the following sharp dimension estimate.
\begin{lem}\label{codim7}
Let $M_k\in\mathcal{M}_{n,m,\La,B_2}$ for some constant $0<\La<\sqrt{2}$.
Assume that $|M_k|\llcorner\mathbf{B}_2$ converges to a stationary varifold $V$ in the varifold sense.
If $\mathcal{S}$ denotes the singular set of $\mathrm{spt}V\cap \mathbf{B}_1$, then $\mathcal{S}$ has Hausdorff dimension $\le n-7$.
\end{lem}
\begin{proof}
Let us prove it by following the steps in the proof of Lemma \ref{codim3}. We assume that $\mathcal{S}$ has Hausdorff dimension $>n-7$.
By the dimension reduction argument, there is a $k(k\le6)$-dimensional non-flat regular minimal cone $C\subset\R^{m+k}$ such that
there is a sequence of minimal graphs $\Si_k=\mathcal{M}_{n,m,\La,B_{r_k}}$ (which are scalings and translations gotten from $M_k$) with $r_k\rightarrow\infty$ so that $\Si_k$ converges to a minimal cone $C_*$ in the varifold sense, which is a trivial product of $C$ and $\R^{n-k}$.

From Theorem \ref{BDMGAFFINE}, there is a point $y_*=(0^n,y^*)\in C_*$ with $0\neq y^*\in\R^{m}$.
If $y_*$ is a singular point of $C_*$, then we blow $C_*$ up at $y_*$, and get the contradiction by Theorem \ref{Stable}, Lemma \ref{codim3} and Lemma \ref{JSimons}.
If $y_*$ is a regular point of $C_*$, then there is a constant $\de_y>0$ such that $\mathbf{B}_{\de_y}(y_*)\cap C_*$ is regular. From Lemma \ref{almostver}, we conclude that the origin is a regular point of $C_*$. It's a contradiction. We complete the proof.
\end{proof}

We summarize Theorem \ref{mainCONE}, Lemma \ref{Cflat} and Lemma \ref{codim7}, and complete the proof of Theorem \ref{main4}.

\Section{Appendix I}{Appendix I}

In this appendix, we will derive several auxiliary algebraic results.
Let $a$ be a real matrix $(a_{\a i})_{m\times n}$, and $b$ be a matrix $(b_{ij})_{n\times n}$ defined by
$$b_{ij}=\de_{ij}+\sum_{\a=1}^m a_{\a i}a_{\a j}.$$
Let $(b^{ij})$ be the inverse matrix of $(b_{ij})$, and $\xi$ be a matrix $(\xi_{\a j})_{m\times n}$ defined by
\begin{equation}\aligned
\xi_{\a j}=\sqrt{\mathrm{det}\,b}\sum_{i=1}^n b^{ij}a_{\a i}.
\endaligned
\end{equation}
\begin{lem}\label{app1}
Let $\la_1,\cdots,\la_n$ be the singular values of $a$ satisfying $\la_1\ge\la_2\ge\cdots\ge\la_n\ge0$ and $\la_1\la_2\le \La$ for a constant $\La>0$. If
\begin{equation}\aligned\label{aaijep*}
a_{11}\ge(1-\ep)\sqrt{\mathrm{det}\,b}
\endaligned
\end{equation}
for a positive constant $\ep<<1$,
then $|\xi_{11}|<1+\psi(\ep|\La)$, where $\psi(\ep|\La)$ is a positive function of $\ep,\La$ with $\lim_{\ep\rightarrow0}\psi(\ep|\La)=0$.
\end{lem}
\begin{proof}
For the fixed $\La>0$, we put $\psi_\ep=\psi(\ep|\La)$ for convenience.
Since $a_{11}\le\la_1$ and $\det\, b=\prod_{k=1}^n(1+\la_k^2)$, then from \eqref{aaijep*} we have
\begin{equation}\aligned\label{la2aai**}
\sum_{k\ge2}\la_k<\psi_\ep,\qquad \sum_{\a+i\ge3}|a_{\a i}|\le\psi_\ep a_{11}\le\psi_\ep\la_1.
\endaligned
\end{equation}
There are two real orthonormal matrices $p=(p_{ij})_{n\times n}$ and $q=(q_{\a\be})_{m\times m}$ such that
$$a_{\a i}=\sum_{j=1}^nq_{\a j}\la_jp_{ji}.$$
From \eqref{aaijep*}, $p_{11}q_{11}\ge1-\psi_\ep$, which implies
\begin{equation}\aligned\label{p11ij}
|1-p_{11}|+\sum_{i+j\ge3}|p_{ij}|\le\psi_\ep.
\endaligned
\end{equation}
We define an $m\times n$ matrix $a^*$ with the element
$$a_{\a i}^*=a_{\a i}-q_{\a 1}\la_1p_{1i}=\sum_{j\ge2}q_{\a j}\la_jp_{ji}.$$
Let $c_i=\sum_\a q_{\a1}a_{\a i}^*$ and $c^*=(c_{ij}^*)$ be an $(n\times n)$-matrix with the elements
$c_{ij}^*=\sum_\a a_{\a i}^*a_{\a j}^*-c_ic_j$.
From \eqref{p11ij}, we have
\begin{equation}\aligned\label{a1iaa1}
\sum_{\a,i}|a_{\a i}^*|\le\psi_\ep\la_2,\quad \sum_i|c_i|\le\psi_\ep\la_2, \quad\sum_{i,j}|c_{ij}^*|\le \psi_\ep\la_2^2.
\endaligned
\end{equation}
By a direct computation,
\begin{equation}\aligned\label{bij1}
b_{ij}=&\de_{ij}+\sum_{\a=1}^m(a_{\a i}^*+q_{\a 1}\la_1p_{1i})(a_{\a j}^*+q_{\a 1}\la_1p_{1j})\\
=&\de_{ij}+\la_1^2p_{1i}p_{1j}+\la_1p_{1i}\sum_\a q_{\a1}a_{\a j}^*+\la_1p_{1j}\sum_\a q_{\a1}a_{\a i}^*+\sum_\a a_{\a i}^*a_{\a j}^*\\
=&\de_{ij}+\la_1^2p_{1i}p_{1j}+\la_1p_{1i}c_j+\la_1p_{1j}c_i+c_ic_j+c_{ij}^*\\
=&\de_{ij}+(\la_1p_{1i}+c_i)(\la_1p_{1j}+c_j)+c_{ij}^*.
\endaligned
\end{equation}
Let $g=(g_{ij})$ be the matrix with the elements
$$g_{ij}=b_{ij}-c_{ij}^*=\de_{ij}+(\la_1p_{1i}+c_i)(\la_1p_{1j}+c_j),$$
then its inverse matrix $(g^{ij})$ satisfies
\begin{equation}\aligned
g^{ij}=\de_{ij}-\f{(\la_1p_{1i}+c_i)(\la_1p_{1j}+c_j)}{1+\sum_k(\la_1p_{1k}+c_k)^2}.
\endaligned
\end{equation}
Then
\begin{equation}\aligned
\sum_jg^{ij}(\la_1p_{1j}+c_j)=&\sum_j\left(\de_{ij}-\f{(\la_1p_{1i}+c_i)(\la_1p_{1j}+c_j)}{1+\sum_k(\la_1p_{1k}+c_k)^2}\right)(\la_1p_{1j}+c_j)\\
=&\f{\la_1p_{1i}+c_i}{1+\sum_k(\la_1p_{1k}+c_k)^2}
\endaligned
\end{equation}
and
\begin{equation}\aligned\label{gija1j}
\sum_jg^{ij}a_{1j}=&\sum_jg^{ij}\left(q_{11}\la_1p_{1j}+a^*_{1j}\right)=q_{11}\sum_jg^{ij}\left(\la_1p_{1j}+c_j\right)+\sum_jg^{ij}\left(a^*_{1j}-q_{11}c_j\right)\\
=&\f{(\la_1p_{1i}+c_i)q_{11}}{1+\sum_k(\la_1p_{1k}+c_k)^2}+\sum_jg^{ij}\left(a^*_{1j}-q_{11}c_j\right).
\endaligned
\end{equation}
From \eqref{p11ij}\eqref{a1iaa1}, it's clear that $\sum_{j\ge1}|g^{1j}|\le \psi_\ep$, then from \eqref{gija1j} it follows that
\begin{equation}\aligned\label{bij+}
\la_1\left|\sum_{j\ge1}g^{1j}a_{1j}\right|\le1+\psi_\ep.
\endaligned
\end{equation}
From \eqref{a1iaa1}, the elements of $g^{-1}c^*$ satisfy $\sum_{i,j}|(g^{-1}c^*)_{ij}|\le\psi_\ep\la_2^2$.
Let $f$ be an $n\times n$ matrix  $f=(I+g^{-1}c^*)^{-1}-I\triangleq(f_{ij})$ with the unit matrix $I$, then $\sum_{i,j}|f_{ij}|\le\psi_\ep\la_2^2$.
Hence
\begin{equation}\aligned\label{bij-}
\la_1\left|\sum_j(fg^{-1})_{1j}a_{1j}\right|\le\psi_\ep.
\endaligned
\end{equation}
Since the inverse matrix of $b=g+c^*$ satisfies
\begin{equation}\aligned
b^{-1}=(g+c^*)^{-1}=(I+g^{-1}c^*)^{-1}g^{-1}=(f+I)g^{-1}=g^{-1}+fg^{-1},
\endaligned
\end{equation}
combining this with  \eqref{bij+}\eqref{bij-}, we have
\begin{equation}\aligned
&|\xi_{11}|=\sqrt{\prod_{k=1}^n(1+\la_k^2)}\left|\sum_j b^{1j}a_{1j}\right|\\
\le&\sqrt{\prod_{k=1}^n(1+\la_k^2)}\left(\left|\sum_jg^{1j}a_{1j}\right|+\left|\sum_j(fg^{-1})_{1j}a_{1j}\right|\right)\le1+\psi_\ep.
\endaligned
\end{equation}
This completes the proof.
\end{proof}

\begin{lem}\label{mu123}
Let $\phi$ be a function defined by
\begin{equation}\aligned
\phi(\mu_1,\mu_2,\mu_3)=4+\mu_1\mu_2\mu_3-\mu_1\mu_2-\mu_1\mu_3-\mu_2\mu_3
\endaligned
\end{equation}
on $(\R^+)^3$ with $\prod_{i=1}^3(1+\mu_i)=v_0^2$ for some constant $v_0>1$. If $\mu_i\mu_j\le2+\f{2}{\max_k\mu_k-1}$ for all $i\neq j$, then
$\phi\ge0$, where the equality is attained at $(2,2,\f19v_0^2-1)$ for $3\le v_0\le3\sqrt{3}$, or $(\mu_*,\mu_*,1+2/\mu_*)$ with $\mu_*\le2$ being the unique positive solution to $2(\mu_*+1)^3=\mu_* v_0^2$.
\end{lem}
\begin{rem}\label{muimujleq4}
Suppose $\mu_1\ge\mu_2\ge\mu_3\ge0$. If $\mu_i\mu_j\le2+\f{2}{\max_k\mu_k-1}$ for all $i\neq j$,  we have
$\mu_1\mu_2\le2+\f{2}{\mu_1-1}=\f{2\mu_1}{\mu_1-1}$, and then $\mu_2\le\f2{\mu_1-1}$. Hence
\begin{equation}\aligned
\mu_1\mu_2\le2+\mu_2\le2+\sqrt{\mu_1\mu_2},
\endaligned
\end{equation}
which implies $\mu_1\mu_2\le4$. In other words, $\mu_i\mu_j\le2+\f{2}{\max_k\mu_k-1}$ for all $i\neq j$ implies $\mu_i\mu_j\le4$ for all $i\neq j$.
\end{rem}
\begin{proof}
By the definition of $v_0$, we have
\begin{equation}\aligned
\phi(\mu_1,\mu_2,\mu_3)
=&4-\mu_1\mu_2+\left(\mu_1\mu_2-\mu_1-\mu_2\right)\left(\f{v_0^2}{(1+\mu_1)(1+\mu_2)}-1\right)\\
=&4+\f{\mu_1\mu_2-\mu_1-\mu_2}{(1+\mu_1)(1+\mu_2)}v_0^2-2\mu_1\mu_2+\mu_1+\mu_2\triangleq\psi(\mu_1,\mu_2)
\endaligned
\end{equation}
with $(1+\mu_1)(1+\mu_2)\le v_0^2$.
From $\mu_1\mu_2\le4$ in Remark \ref{muimujleq4}, $\psi(\mu_1,\mu_2)=4-\mu_1\mu_2\ge0$ clearly on the set
$\{\mu_1\ge0,\,u_2\ge0|\,(1+\mu_1)(1+\mu_2)= v_0^2\}$, where the equality is attained at $(\mu_1,\mu_2,\mu_3)=(2,2,0)$ with $v_0=3$.
Since
\begin{equation*}\aligned
\p_{\mu_1}\psi=&\left(\f{\mu_2-1}{1+\mu_1}-\f{\mu_1(\mu_2-1)-\mu_2}{(1+\mu_1)^2}\right)\f{v_0^2}{1+\mu_2}-2\mu_2+1\\
=&(2\mu_2-1)\f{v_0^2}{1+\mu_2}-2\mu_2+1=(2\mu_2-1)\left(\f{v_0^2}{1+\mu_2}-1\right),
\endaligned
\end{equation*}
and $\psi(\mu_1,\mu_2)=\psi(\mu_2,\mu_1)$,
it follows that 
$$\{\mu_1\ge0,\,u_2\ge0|\,D\psi(\mu_1,\mu_2)=0\}\subset\{\mu_1\ge0,\,u_2\ge0|\,\mu_1=\mu_2\}.$$ 
Note that
\begin{equation}\aligned\label{psimumu}
\psi(\mu,\mu)=4+\f{\mu^2-2\mu}{(1+\mu)^2}v_0^2-2\mu^2+2\mu=\f{\mu-2}{(1+\mu)^2}\left(\mu v_0^2-2(1+\mu)^3\right).
\endaligned
\end{equation}
We suppose $\mu_i\mu_j\le2+\f{2}{\max_k\mu_k-1}$ for all $i\neq j$, and $\mu_1=\mu_2$.
\begin{itemize}
  \item If $\mu_3\le\mu_1$, then $\mu_1\mu_2\le2+\f{2}{\mu_1-1}=\f{2\mu_1}{\mu_1-1}$, which implies $\mu_1\le2$.
Note that $v_0^2\le(1+\mu_1)^3$, then from \eqref{psimumu}
\begin{equation}\aligned
\psi(\mu_1,\mu_1)\ge\f{\mu_1-2}{(1+\mu_1)^2}\left(\mu_1 (1+\mu_1)^3-2(1+\mu_1)^3\right)\ge0,
\endaligned
\end{equation}
where the equality is attained at $(\mu_1,\mu_2,\mu_3)=(2,2,\f19v_0^2-1)$ for $3\le v_0\le3\sqrt{3}$.
  \item If $\mu_3\ge\mu_1$, then
$$\mu_1\mu_3\le2+\f{2}{\mu_3-1}=\f{2\mu_3}{\mu_3-1},$$
which implies $\mu_3\le1+2/\mu_1$. Note that $\mu_1\le\mu_3\le1+2/\mu_1$ implies $\mu_1\le2$.
Hence from \eqref{psimumu},
\begin{equation}\aligned
\psi(\mu_1,\mu_1)\ge\f{\mu_1-2}{(1+\mu_1)^2}\left(\mu_1 (1+\mu_2)^2(1+1+2/\mu_1)-2(1+\mu_1)^3\right)=0,
\endaligned
\end{equation}
where the equality is attained at $(\mu_1,\mu_2,\mu_3)=(2,2,2)$ with $v_0=3\sqrt{3}$, or $(\mu_*,\mu_*,1+2/\mu_*)$ with $\mu_*\le2$ being the unique positive solution to $2(\mu_*+1)^3=\mu_* v_0^2$.
\end{itemize}
This completes the proof.
\end{proof}

\begin{cor}\label{mu123La}
Let $\phi$ be a function defined by
\begin{equation}\aligned
\phi(\mu_1,\mu_2,\mu_3)=4+\mu_1\mu_2\mu_3-\mu_1\mu_2-\mu_1\mu_3-\mu_2\mu_3
\endaligned
\end{equation}
on $(\R^+)^3$. If $\sup_{1\le i<j\le3}\mu_i\mu_j\le\La^2$ for some constant $0<\La\le\sqrt{2}$, then
$\phi\ge(2-\sqrt{2})(2-\La^2)$.
\end{cor}
\begin{proof}
Let $\tilde{\mu}_i=\sqrt{2}\mu_i/\La$, then
\begin{equation}\aligned
&\phi=4+\f{\La^3}{2\sqrt{2}}\tilde{\mu}_1\tilde{\mu}_2\tilde{\mu}_3-\f{\La^2}2\left(\tilde{\mu}_1\tilde{\mu}_2+\tilde{\mu}_1\tilde{\mu}_3+\tilde{\mu}_2\tilde{\mu}_3\right)\\
=&4-2\La^2+\tilde{\mu}_1\tilde{\mu}_2\tilde{\mu}_3\left(\f{\La^3}{2\sqrt{2}}-\f{\La^2}2\right)
+\f{\La^2}2\left(4+\tilde{\mu}_1\tilde{\mu}_2\tilde{\mu}_3-\tilde{\mu}_1\tilde{\mu}_2-\tilde{\mu}_1\tilde{\mu}_3-\tilde{\mu}_2\tilde{\mu}_3\right).
\endaligned
\end{equation}
From the assumption, $\sup_{1\le i<j\le3}\tilde{\mu}_i\tilde{\mu}_j\le2$, and $\tilde{\mu}_1\tilde{\mu}_2\tilde{\mu}_3\le2\sqrt{2}$.
Combining Lemma \ref{mu123}, we get
\begin{equation}\aligned
\phi\ge&4-2\La^2+2\sqrt{2}\left(\f{\La^3}{2\sqrt{2}}-\f{\La^2}2\right)\ge4-2\La^2+2\sqrt{2}\left(\f{\La^3}{2\sqrt{2}}-\f{\La}{\sqrt{2}}\right)\\
=&(2-\La)(2-\La^2)\ge(2-\sqrt{2})(2-\La^2).
\endaligned
\end{equation}
\end{proof}

\Section{Appendix II}{Appendix II}

Analog to Lemma 4.3 in \cite{d3}, we have the following multiplicity one convergence for Lipschitz minimal graphs.
Let $\Om$ be a domain in $\R^{n}$ with countably $(n-1)$-rectifiable boundary $\p\Om$.
\begin{lem}\label{Multi1}
Let $M_k=\mathrm{graph}_{u_k}$ be a sequence of Lipschitz minimal graphs over $\Om$ of codimension $m\ge1$ with $\sup_k\mathbf{Lip}\,u_k<\infty$.
Then there are a Lipschitz function $u_\infty:\overline{\Om}\rightarrow\R^m$ with $\mathbf{Lip}\,u_\infty\le\sup_k\mathbf{Lip}\,u_k$, and a multiplicity one $n$-varifold $V$ in $\overline{\Om}\times\R^m$ with $\mathrm{spt}V=\{(x,u_\infty(x))\in\R^n\times\R^m|\, x\in \overline{\Om}\}$ such that
up to a choice of the subsequence $|M_k|$ converges as $k\to\infty$ to $V$ in $\Om\times\R^m$ in the varifold sense.
\end{lem}
\textbf{Remark}. In Proposition 11.53 of \cite{g-m}, Giaquinta and Martinazzi have already proved the multiplicity one of $V$ in the above lemma. Here, we give an alternative proof for completeness.
\begin{proof}
By Arzela-Ascoli theorem, up to a choice of the subsequence, we assume that there is a Lipschitz function $u_\infty$ on $\overline{\Om}$ with $\mathbf{Lip}\,u_\infty\le\sup_k\mathbf{Lip}\,u_k$.
By compactness of varifolds (see \cite{s}), there is an $n$-varifold $V$ in $\overline{\Om}\times\R^m$ such that up to a choice of the subsequence, $|M_k|$ converges to an integer multiplicity stationary varifold $V$ in $\Om\times\R^m$ in the varifold sense.
Let $\mu_V$ denote the Radon measure associated to $V$. By monotonicity of the density of $V$, for any $\mathbf{x}_*\in\mathrm{spt}V\cap(\Om\times\R^m)$ we have
$\mu_V(\mathbf{B}_r(\mathbf{x}_*))\ge\omega_nr^n$ for sufficient small $r>0$. By the convergence of $|M_k|$, there is a sequence $\mathbf{x}_k\in M_k$ with $\mathbf{x}_k\rightarrow\mathbf{x}_*$. Denote $\mathbf{x}_k=(x_k,u_k(x_k))$. Then $x_k$ converges to a point $x_*$ with $\pi_*(\mathbf{x}_*)=x_*$, where $\pi_*$ is defined in \eqref{pi*****}. Therefore, $\mathbf{x}_*=\lim_{k\rightarrow\infty}\mathbf{x}_k=\lim_{k\rightarrow\infty}(x_k,u_k(x_k))=(x_*,u_\infty(x_*))$, which implies the support of $V$
\begin{equation}\aligned\label{sptVVVVV}
\mathrm{spt}V\subset\{(x,u_\infty(x))\in\overline{\Om}\times\R^m|\, x\in \R^n\}.
\endaligned
\end{equation}
Note that for any $z\in\Om$, $\mathcal{H}^n(\mathbf{B}_r(\mathbf{z}_k)\cap M_k)\ge\omega_nr^n$ for all suitably small $r>0$ with $\mathbf{z}_k=(z,u_k(z))$.
Since $(z,u_k(z))\rightarrow(z,u_\infty(z))$ as $k\rightarrow\infty$, from the convergence of $|M_k|$ we get $\mu_V(\mathbf{B}_r(\mathbf{z}))\ge\omega_nr^n$ for $\mathbf{z}=(z,u_\infty(z))$.
In particular, $\mathbf{z}\in\mathrm{spt}V$, which implies
\begin{equation}\aligned\label{sptVVVVV*}
\{(x,u_\infty(x))\in\R^n\times\R^m|\, x\in \overline{\Om}\}\subset\mathrm{spt}V.
\endaligned
\end{equation}

Now it only remains to prove that $V$ has multiplicity one. Let reg$V$ denote the regular part of $V$ in $\Om\times\R^m$.
For any $\mathbf{y}\in\mathrm{reg}V$, let $T_{\mathbf{y}}V$ denote the tangent plane of spt$V$ at $\mathbf{y}$.
Let $\xi_1,\cdots,\xi_n$ be an orthonormal basis of $T_{\mathbf{y}}V$. 
From Lemma 22.2 in \cite{s},
\begin{equation}\aligned\label{Mkeix}
\lim_{r\rightarrow0}\left(r^{-n}\lim_{k\rightarrow\infty}\int_{M_k\cap \mathbf{B}_{r}(\mathbf{y})}\left|e_{k,1}\wedge\cdots\wedge e_{k,n}-\xi_{1}\wedge\cdots\wedge \xi_{n}\right|^2\right)=0,
\endaligned
\end{equation}
where $e_{k,1},\cdots,e_{k,n}$ is a local orthonormal tangent frame of $M_k$ for each $k$.
We also treat $e_{k,i}$ as a vector on $\pi_*(M_k)$ by letting $e_{k,i}(x)=e_{k,i}(x,u_k(x))$ for each $i=1,\cdots,n$ and $k\ge1$.
Let $\{\mathbf{E}_i\}_{i=1}^{n+m}$ denote the standard orthonormal basis of $\R^{n+m}$ such that $\mathbf{E}_i$ corresponds to the axis $x_i$ for $i=1,\cdots,n+m$, $v_k$ be a function on $\R^n$ defined by $v_k^{-1}=\left|\lan e_{k,1}\wedge\cdots\wedge e_{k,n},\mathbf{E}_1\wedge\cdots\wedge \mathbf{E}_n\ran\right|$. From $\sup_k\mathbf{Lip}\,u_k<\infty$ and \eqref{Mkeix}, we get
\begin{equation}\aligned\label{Mkeix*}
\lim_{r\rightarrow0}\left(r^{-n}\lim_{k\rightarrow\infty}\int_{B_r(y)}\left|e_{k,1}\wedge\cdots\wedge e_{k,n}-\xi_{1}\wedge\cdots\wedge \xi_{n}\right|^2v_k\right)=0
\endaligned
\end{equation}
with $\mathbf{y}=(y,u_\infty(y))$.
Let $v_\infty=\left|\lan \xi_1\wedge\cdots\wedge \xi_n,\mathbf{E}_1\wedge\cdots\wedge \mathbf{E}_n\ran\right|^{-1}$.
From \eqref{Mkeix*} and
\begin{equation}\aligned
&\int_{B_{r}(y)}\left|1-v_kv_\infty^{-1}\right|=\int_{B_{r}(y)}\left|v_k^{-1}-v_\infty^{-1}\right|v_k\\
\le&\int_{B_{r}(y)}\left|\lan e_{k,1}\wedge\cdots\wedge e_{k,n}-\xi_1\wedge\cdots\wedge \xi_n,\mathbf{E}_1\wedge\cdots\wedge \mathbf{E}_n\ran\right|v_k\\
\le&\int_{B_{r}(y)}\left|e_{k,1}\wedge\cdots\wedge e_{k,n}-\xi_1\wedge\cdots\wedge \xi_n\right|v_k,
\endaligned
\end{equation}
with the Cauchy inequality we get
\begin{equation}\aligned
&\lim_{r\rightarrow0}\left(r^{-n}\lim_{k\rightarrow\infty}\int_{B_{r}(y)}\left|1-v_kv_\infty^{-1}\right|\right)\\\le&\lim_{r\rightarrow0}\left(r^{-n}\lim_{k\rightarrow\infty}\left(\int_{B_{r}(y)}v_k\int_{B_{r}(y)}\left|e_{k,1}\wedge\cdots\wedge e_{k,n}-\xi_{1}\wedge\cdots\wedge \xi_{n}\right|^2v_k\right)\right)=0,
\endaligned
\end{equation}
which implies
\begin{equation}\aligned
&\lim_{r\rightarrow0}\left(r^{-n}\mu_V(B_{r}(y)\times\R^m)\right)
=\lim_{r\rightarrow0}\left(r^{-n}\lim_{k\rightarrow\infty}\mathcal{H}^n\left(M_k\cap(B_{r}(y)\times\R^m)\right)\right)\\
=&\lim_{r\rightarrow0}\left(r^{-n}\lim_{k\rightarrow\infty}\int_{B_{r}(y)}v_k\right)=\omega_n v_\infty=\omega_n\left|\lan \xi_1\wedge\cdots\wedge \xi_n,\mathbf{E}_1\wedge\cdots\wedge \mathbf{E}_n\ran\right|^{-1}.
\endaligned
\end{equation}
With \eqref{sptVVVVV}, we conclude that $V$ has multiplicity one everywhere on spt$V$. This completes the proof.
\end{proof}

\Section{Appendix III}{Appendix III}

Let $\La$ be a positive constant $<\sqrt{2}$, and $M_k\in \mathcal{M}_{n,\La}$ for each integer $k\ge1$.
From Theorem \ref{CUVest}, $M_k$ is smooth for each $k$.
From \eqref{logvLa}, integrating by parts infers that there is a constant $c_{n,\La}>0$ depending only on $n$ and $\La$
such that
\begin{equation}\aligned\label{Ai}
\int_{M_k\cap\mathbf{B}_{\r}(p)}|B_{M_k}|^2\le c_{n,\La}\r^{n-2}
\endaligned
\end{equation}
for any $p\in M_k$ and any $\r>0$.
Here, $B_{M_k}$ is the second fundamental form of $M_k$ in $\R^{n+m}$.

\begin{lem}\label{cross}
Suppose that $|M_k|$ converges to a nontrivial stationary varifold $V$ in the varifold sense. If $V$ splits off $\R^{n-1}$ isometrically,
then $\mathrm{spt}V$ is an $n$-plane.
\end{lem}
The proof is similar to the argument in the proof of Theorem 2 of \cite{s-s} by Schoen-Simon. For self-containment, we give the proof here.
\begin{proof}
Let us prove it by contradiction. Suppose that there is a varifold $T$ in $\R^{m+1}$ such that $\mathrm{spt}T$ is not a line in $\R^{m+1}$, and
$$\mathrm{spt}V=\{(x,y)\in\R^{n-1}\times\R^{m+1}|\, y\in\mathrm{spt}T\}=\R^{n-1}\times\mathrm{spt}T.$$
We write
\begin{equation}\aligned
T=\sum_{j=1}^ln_j|R_j|,\qquad R_j=\{\la p_j|\ \la>0\},
\endaligned
\end{equation}
with $l\ge2,n_j$ positive integers, $|p_j|=1$ and $p_1,\cdots,p_l$ spanning a space of dimension $\ge2$.
Let $\xi_k$ be the orientation of $M_k$ defined in \eqref{orientation}. Note that $T$ has multiplicity one. For any $0<\r\le1/2$, from Allard's regularity theorem there is a constant $k_\r>0$ such that
\begin{equation}\aligned
M_k\cap\left(\{x\}\times\{y\in\R^{m+1}|\ \r<|y|<2\r\}\right)=\bigcup_{j=1}^{l'}\g_{j}^k(x)
\endaligned
\end{equation}
for each $k\ge k_\r$ and each $x\in\R^{n-1}$ with $|x|\le1$, where $l'=\sum_{j=1}^ln_{l}$, $\g_{j}^k(x)$ are smooth properly embedded Jordan arcs having their endpoints in
$\{x\}\times\{y\in\R^{m+1}|\ |y|=\r\ \mathrm{or}\ |y|=2\r\}$, and satisfying
\begin{equation}\aligned
&\lim_{k\rightarrow\infty}\max_{j'=1,\cdots,l'}\sup_{x\in\R^{n-1},\ |x|\le1}\mathrm{dist}\left(\g_{j'}^k(x),\R^{n-1}\times T\right)=0,\\
&\lim_{k\rightarrow\infty}\min_{j=1,\cdots,l}\sup_{\mathbf{x}\in\g_{j'}^i(x),\ j'=1,\cdots,l'}\left|\xi_k(\mathbf{x})\wedge(0^{n-1},p_j)\right|=0.
\endaligned
\end{equation}
Here, $0^{n-1}$ denotes the origin of $\R^{n-1}$.

On the other hand, for each $\mathbf{x}\in M_k$, there is a constant $r_{k,\mathbf{x}}$ such that for every $0<r<r_{k,\mathbf{x}}$, each component of $M_k\cap \textbf{B}_{r}(\mathbf{x})$ is embedded, and can be written as a graph over the tangent plane of $M_k\cap B_{r}(\mathbf{x})$ at $\mathbf{x}$ with the graphic function $w_{k,\mathbf{x}}$ so that $|Dw_{k,\mathbf{x}}|<\f1{10}r$.
Covering lemma and Sard's theorem imply that for almost all $x\in\R^{n-1}$ with $|x|\le1$, and for each $k\ge1$ we have
\begin{equation}\aligned
M_k\cap\left(\{x\}\times\{y\in\R^{m+1}|\ |y|\le\r\}\right)=\left(\bigcup_{j=1}^{l_1}\G_j^k(x)\right)\cup\left(\bigcup_{j=1}^{l_2}\Upsilon_j^k(x)\right),
\endaligned
\end{equation}
where $l_1$ is a positive integer (depending on
$i, x$), $\G_j^k(x)$ are smooth properly embedded arcs with endpoints contained
in $\{x\}\times\{y\in\R^{m+1}|\ |y|=\r\}$, $l_2$ is a non-negative integer, and $\Upsilon_j^k(x)$ are
smooth properly embedded curves (with no endpoints).
Hence for almost all $x\in\R^{n-1}$ with $|x|\le1$, $M_k\cap\left(\{x\}\times\{y\in\R^{m+1}|\ |y|\le2\r\}\right)$ is a union of several embedded smooth arcs or curves with their endpoints in
$\{x\}\times\{y\in\R^{m+1}|\ |y|=2\r\}$.

Clearly, there are a constant $\be>0$ independent of $\r$ and a large constant $k^*_\r$
such that for all $k\ge k^*_\r$ and for almost all $x\in\R^{n-1}$ with $|x|\le1$, there is an embedded smooth arcs $\g^k_{*,\r}$ in $M_k\cap\left(\{x\}\times\{y\in\R^{m+1}|\ |y|\le2\r\}\right)$ with their endpoints in
$\{x\}\times\{y\in\R^{m+1}|\ |y|=2\r\}$ (depending on $\r$) so that
\begin{equation}\aligned
\sup_{\mathbf{x}_1,\mathbf{x}_2\in\g^k_{*,\r}}\left|\xi_k(\mathbf{x}_1)-\xi_k(\mathbf{x}_2)\right|\ge\be.
\endaligned
\end{equation}
Let $\n_{\dot{\g}^k_{*,\r}}\xi_k$ denote the directional derivative of $\xi_k$ in the direction of the tangent to $\g^k_{*,\r}$. Then from the above inequality one has
\begin{equation}\aligned
\be\le\sup_{\mathbf{x}_1,\mathbf{x}_2\in\g^k_{*,\r}}\left|\xi_k(\mathbf{x}_1)-\xi_k(\mathbf{x}_2)\right|\le\int_{\g^k_{*,\r}}\left|\n_{\dot{\g}^k_{*,\r}}\xi_k\right|.
\endaligned
\end{equation}
Since
$\left|\n_{\dot{\g}^k_{*,\r}}\xi_k\right|\le c_n|B_{M_k}|$ for some constant $c_n>0$, then
\begin{equation}\aligned\label{beAi}
\f{\be}{c_n}\le\int_{\g^k_{*,\r}}\left|B_{M_k}\right|\le\int_{M_k\cap\left(\{x\}\times\{y\in\R^{m+1}|\ |y|\le2\r\}\right)}\left|B_{M_k}\right|
\endaligned
\end{equation}
for almost all $x\in\R^{n-1}$ with $|x|\le1$, and $k\ge k^*_\r$. Denote $M_{k,\r}=M_k\cap\{(x,y)\in\R^{n-1}\times\R^{m+1}|\ |x|<1,\ |y|\le2\r\}$ for every $0<\r\le1/2$. Integrating \eqref{beAi} over $|x|<1$ and using the co-area formula yield
\begin{equation}\aligned\label{beAi2}
\f{\omega_{n-1}}{c_n}\be\le\int_{M_{k,\r}}\left|B_{M_k}\right|\le&\left(\mathcal{H}^n(M_{k,\r})\right)^{1/2}\left(\int_{M_{k,\r}}\left|B_{M_k}\right|^2\right)^{1/2}\\
\le&\left(\mathcal{H}^n(M_{k,\r})\right)^{1/2}\left(\int_{M_k\cap\mathbf{B}_{\sqrt{2}}}\left|B_{M_k}\right|^2\right)^{1/2}.
\endaligned
\end{equation}
There are an integer $l_\r>1$ and a finite sequence of $\{x'_j\}_{j=1}^{l_\r}\subset\R^{n-1}$ with $|x'_j|<1$ such that $\bigcup_{j=1}^{l_\r}B_{\r}(x'_j)\supset B_1(0^{n-1})$ and $l_\r\r^{n-1}<c_n'$ for some constant $c_n'$ depending only on $n$. Here, $B_{\r}(x_j')$ denotes the ball in $\R^{n-1}$ centered at $x_j'$ with the radius $\r$, $B_1(0^{n-1})$ denotes the unit ball in $\R^{n-1}$ centered at the origin. Denote $z_j=(x'_j,0^{m+1})$. Then
$$B_{\r}(x'_j)\times\{y\in\R^{m+1}|\ |y|\le2\r\}\subset\mathbf{B}_{3\r}(z_j),$$
which implies
$$B_{1}(0^{n-1})\times\{y\in\R^{m+1}|\ |y|\le2\r\}\subset\bigcup_{j=1}^{l_\r}\mathbf{B}_{3\r}(z_j).$$
Combining Lemma \ref{VGM*}, \eqref{Ai}\eqref{beAi2} and $l_\r\r^{n-1}<c_n'$, we have
\begin{equation}\aligned\label{beAi3}
\omega_{n-1}^2\be^2c_n^{-2}\le& c_{n,\La}2^{\f n2-1}\mathcal{H}^n(M_{k,\r})\le c_{n,\La}2^{\f n2-1}\sum_{j=1}^{l_\r}\mathcal{H}^n\left(M_{k}\cap \mathbf{B}_{3\r}(z_j)\right)\\
\le& c_{n,\La}2^{\f n2-1}l_\r C_{n,\La}\omega_n(3\r)^n\le 2^{\f n2-1}3^n\omega_nc_{n,\La}C_{n,\La} c_n'\r.
\endaligned
\end{equation}
However, the above inequality fails for the sufficiently small $\r>0$.
Thus, we get that $\mathrm{spt}T$ is a line in $\R^{m+1}$, and complete the proof.
\end{proof}

\bibliographystyle{amsplain}

\end{document}